\documentclass{elsearticle}
\usepackage[utf8]{inputenc}
\usepackage{amsmath, mathtools, amssymb, amsthm}
\usepackage{subfigure}
\usepackage[margin=0.8in]{geometry}
\usepackage{bm}
\usepackage{empheq}
\usepackage{color}

\newtheorem{theorem}{Theorem}[section]

\newtheorem{lemma}[theorem]{Lemma}
\newtheorem{definition}{Definition}[section]

\makeatletter
\renewcommand*\env@matrix[1][\arraystretch]{%
  \edef\arraystretch{#1}%
  \hskip -\arraycolsep
  \let\@ifnextchar\new@ifnextchar
  \array{*\c@MaxMatrixCols c}}
\makeatother

\journal{}

\begin{document}

\begin{frontmatter}

\title{Formulation of Entropy-Stable schemes for the multicomponent compressible Euler equations}

\author{Ayoub Gouasmi}
\author{Karthik Duraisamy}
\address{University of Michigan, Department of Aerospace Engineering, Ann Arbor, MI, USA}

\author{Scott M. Murman}
\address{NASA Ames Research Center, NASA Advanced Supercomputing Division, Moffett field, CA, USA}

%\date{}

\begin{abstract}
   In this work, Entropy-Stable (ES) schemes are formulated for the multicomponent compressible Euler equations. Entropy-conservative (EC) and ES fluxes are derived. Particular attention is paid to the limit case of zero partial densities where the structure required by ES schemes breaks down (the entropy variables are no longer defined). It is shown that while an EC flux is well-defined in this limit, a well-defined upwind ES flux requires appropriately averaged partial densities in the dissipation matrix. A similar result holds for the high-order TecNO reconstruction. However, this does not prevent the numerical solution from developing negative partial densities or internal energy. Numerical experiments were performed on one-dimensional and two-dimensional interface and shock-interface problems. The present scheme exactly preserves stationary interfaces. On moving interfaces, it produces spurious pressure oscillations typically observed with conservative schemes [Karni, \textit{J. Comput. Phys.}, 112 (1994) 1]. We find that these anomalies, which are not present in the single-component case, violate neither entropy stability nor a minimum principle of the specific entropy. Finally, we show that the scheme is able to reproduce the physical mechanisms of the two-dimensional shock-bubble interaction problem [Haas \& Sturtevant, J. Fluid Mech. 181 (1987) 41, Quirk \& Karni, J. Fluid Mech. 318 (1996) 129]. 
\end{abstract}
\end{frontmatter}

%\tableofcontents

\section{Introduction}
\indent Entropy-Stable (ES) schemes \cite{Tadmor} have been gaining interest over the past decade, especially in the context of under-resolved simulations of compressible turbulent flows using high-order accurate numerical methods \cite{HO, Diosady, Pazner, Fernandez_ES}. ES schemes are attractive because they provide stability in an integral sense: the total entropy of the discrete system can be made non-decreasing (2nd principle) or conserved in which case the scheme is termed Entropy-Conservative (EC). ES schemes were initially motivated by the fact that some hyperbolic systems of conservation laws:
\begin{equation}\label{eq:PDE_1D}
\frac{\partial \mathbf{u}}{\partial t}  + \frac{\partial \mathbf{f}}{\partial x}  = 0,
\end{equation}
where $\mathbf{u}$ and $\mathbf{f}$ are the state and flux vectors, admit a convex extension \cite{Friedrichs, Harten} in the sense that they imply an additional conservation equation for an \textit{entropy} $U$:
\begin{equation}\label{eq:Entropy}
    \frac{\partial U}{\partial t} + \frac{\partial F}{\partial x} = 0, 
\end{equation}
where $(U, \ F) = (U(\mathbf{u}), F(\mathbf{u})) \in \mathbb{R}^2$ is an entropy pair satisfying $U_{,\mathbf{u}}^T\mathbf{f}_{,\mathbf{u}} = F_{,\mathbf{u}}^T$ (for equation (\ref{eq:PDE_1D}) to imply equation (\ref{eq:Entropy})) and $U$ is convex. Entropy functions play an important role in the stability analysis of systems of conservation laws (see \cite{Tadmor_acta} and references therein). The entropy inequality:
\begin{equation}\label{eq:Entropy_prod}
    \frac{\partial U}{\partial t} + \frac{\partial F}{\partial x} < 0,
\end{equation}
considered in the sense of distributions, is typically used to eliminate non-physical weak solutions. For the compressible Euler equations, this leads to the well known entropy conditions which must be satisfied across a shock \cite{Lax}. Note however that the entropy inequality (\ref{eq:Entropy_prod}) alone is not enough to uniquely determine the correct weak solution \cite{NES1, NES2, NES3, Gouasmi1}. \\
\indent Tadmor \cite{Tadmor} introduced finite-volume discretizations of equation (\ref{eq:PDE_1D}) which are locally consistent with either equation (\ref{eq:Entropy}) or inequality (\ref{eq:Entropy_prod}) for a given entropy pair. This is achieved in two fundamental steps. First, an EC flux, which does not introduce any entropy production or loss, is derived by solving a scalar entropy conservation condition \cite{Tadmor}. Second, a dissipation operator is added to the EC flux in order to produce entropy across discontinuities. Despite some developments \cite{Barth, LeFloch, Tadmor_acta}, an issue that hindered the use of EC/ES schemes is the non-closed form of the first EC flux \cite{Tadmor} (path integral in phase space) in the general nonlinear systems case. Roe \cite{Roe1} addressed it by showing, for the compressible Euler equations, how to solve the EC condition algebraically to obtain a closed-form EC flux. Additionally, he proposed an upwind-like ES dissipation operator \cite{Roe1} that has become a default choice when constructing ES schemes. From there, a number a developments followed, mostly focused on high-order discretizations \cite{Fisher, Fried, Fjordholm, Pazner, Fernandez_ES}. To the best of the authors' knowledge, the most concrete improvement ES schemes brought about in Computational Fluid Dynamics (CFD) is in improving the robustness of under-resolved turbulent flow calculations using high-order methods \cite{Diosady, Pazner, Fernandez_ES}. Efforts are being spent on the development of code infrastructures which take advantage of the benefits of ES schemes. An example is the \textit{eddy} solver \cite{eddy0, Diosady} developed at NASA Ames Research Center for the simulation of turbulent separated flows. The present work is part of an effort to expand the field of ES schemes towards multi-physics applications \cite{eddy} and identify relevant research directions. \\
\indent In this work, we consider the multicomponent ($N$ species) compressible Euler equations consisting of the conservation equations for partial densities, momentum and total energy. This system can be viewed as the compressible Euler equations (conservation of total mass, momentum and total energy) complemented with $N-1$ species conservation equations. This observation motivated early multicomponent schemes such as the one by Habbal \textit{et al.} \cite{Habbal}, where the Roe scheme \cite{Roe} is applied to the Euler part of the equations and the $N-1$ remaining equations are treated separately. In Larrouturou \cite{Larrouturou}, such approaches are termed uncoupled as opposed to fully coupled approaches, which treat the multicomponent system as a whole. An example of a fully coupled approach is the extension of the Roe scheme by Fernandez \& Larroutouru \cite{Fernandez}, and Abgrall \cite{ Abgrall0}. It might be then tempting to use an existing entropy-stable scheme for the Euler part and evolve the remaining $N-1$ species equations with another scheme. A case in point can be found in Derigs \textit{et al.} \cite{Derigs} (section 3.8). While this approach has the benefit of simplicity (minimal programming and computational effort), it is lacking from a theoretical viewpoint. This approach implicitly assumes that the entropy of the single component system is an admissible entropy for the multicomponent system, meaning that it is a conserved quantity in the smooth regime and a convex function of the conserved variables. The necessity of a fully coupled approach to develop entropy-stable schemes, in the sense of Tadmor \cite{Tadmor}, for multicomponent flows is motivated by mathematical and physical arguments. \\
\indent At this juncture, it is important to note that the schemes we are interested in this work achieve entropy-stability in a specific way. That is to say that there is more than one way that a scheme can be made stable in an entropy sense, and hence be called 'entropy-stable'. Osher's family of E-schemes \cite{Osher} and Barth's space-time discontinuous galerkin schemes \cite{Barth} are conservative schemes which satisfy an entropy inequality, but their construction is different. There are also non-conservative schemes which can be designed to conserve or produce entropy \cite{Castro}. Entropy stability can be understood in a different way as well. The scheme introduced by Ma \textit{et al.} \cite{Ihme} is termed entropy-stable because it enforces a minimum principle of the specific entropy, proved by Tadmor for the Euler equations \cite{Tadmor}. In their work, stability is sought in a point-wise sense (a scheme which preserves the positivity of density and satisfies the minimum principle cannot crash in principle), not in an integral sense. Integral stability and point-wise stability are both important concepts, and in principle they do not imply each other. In either case, it is important to emphasize that the correct formulation of these schemes depends on the structure of the equations they are applied to. Chalot \textit{et al.} \cite{Chalot} and Giovangigli \cite{Giovangigli} demonstrated that the multicomponent compressible Euler equations do possess the structure that ES schemes require. Additionally, a minimum principle of the mixture's specific entropy was recently shown in Gouasmi \textit{et al.} \cite{Gouasmi4}. It is not clear to the authors whether these results extend to the systems considered in \cite{Ihme}. \\
\indent Integral entropy stability is not the only trait of ES schemes. An interesting feature of Tadmor's framework is that the amount of entropy produced by the scheme can be precisely quantified and controlled to some extent. Ismail \& Roe's work on the carbuncle problem \cite{Ismail, IsmailThesis} relies on this feature to delve into the question of appropriate entropy production across shocks. There is a sense of how entropy is being managed locally, and neither the choice of EC flux nor the choice of dissipation operator is unique. However, this freedom in the design of the scheme leaves the user the outstanding challenge of figuring out what is the right way to manage entropy in a discrete flow field. This is an extremely complex question, the scope of which is by far not limited to shocks, but also includes rarefactions \cite{Gouasmi1}, contact discontinuities (of relevance to this work) and the under-resolution of turbulence \cite{MurmanCTR}. The construction of our ES scheme follows a classic procedure \cite{Tadmor, Roe}: from the definition of the hyperbolic system, the entropy pair and the corresponding entropy variables, we derive the potential flux functions, a baseline EC flux \cite{Tadmor, Roe1} and a scaling matrix \cite{Merriam, Barth} needed to define the entropy-stable dissipation operator of Roe \cite{Roe1}. Hence, our work does not explore this question in depth, but it does provides a lucid view of how the current state of the art of ES schemes fares in multicomponent flows. \\
\indent We have come across a few obstacles during the course of this work. The first one is theoretical: the structure required by ES schemes collapses when one of the partial densities is zero. The entropy $U$, which we take as the opposite of the mixture's thermodynamic entropy \cite{Chalot, Giovangigli}, is no longer convex and the entropy variables, which are key in constructing ES schemes, are no longer defined. Upon closer examination, we observed that the EC flux we derived is well-defined in this limit and still satisfies the Entropy Conservation condition. We also found that the dissipation operator remains defined provided that the averaged partial densities involved in the dissipation matrix are evaluated in a certain way. The second issue is that while the overall scheme is always defined, there is no guarantee that it will not produce negative densities or pressure, even at first order, at the next time step. Last but not least, numerical experiments showed that while the ES scheme can handle shocks and stationary contact discontinuities correctly, it fails to preserve pressure equilibrium and constant velocity when a moving interface is simulated. This is in fact a longtime and well-known \cite{Abgrall, Karni} failure of conservative finite-volume schemes on multicomponent flows (there are no such anomalies in the single component case). We found that these anomalies neither violate entropy-stability nor a minimum principle on the specific entropy of the mixture \cite{Gouasmi4}. \\
\indent The paper is organized as follows. In section 2, we present the modeling assumptions, the governing equations and their symmetrization using the entropy variables \cite{Chalot, Giovangigli}. Section 3 is dedicated to the construction of ES schemes for multicomponent flows. We formulate an EC flux and an ES interface flux based on upwinding \cite{Roe1, Ismail} and discuss their definition in the limit of vanishing partial densities. In section 4, we discuss how this limit impacts standard high-order ES discretizations. In section 5, the numerical scheme is tested on one-dimensional and two-dimensional interface and shock-interface problems. 

\section{Governing equations, entropy variables and symmetrization}
\indent The governing equations describe the conservation of species mass, momentum and total energy. In 1D, that is system (\ref{eq:PDE_1D}) with the vector conserved variables $\mathbf{u}$ and the vector of fluxes $\mathbf{f}$ defined by:
\begin{equation*}
    \mathbf{u} := \begin{bmatrix} \rho_1 & \hdots & \rho_N & \rho u & \rho e^{t} \end{bmatrix}^T, \ 
    \mathbf{f} := \begin{bmatrix} \rho_1 u & \hdots & \rho_N u & \rho u^2 + p & (\rho e^{t} + p)u \end{bmatrix}^T.
\end{equation*}
$\rho_k$ denotes the partial density of species k, $\rho := \sum_{k=1}^N\rho_k$ denotes the total density, $u$ denotes the velocity and $e^t$ denotes the specific total energy. The pressure $p$ is given by the ideal gas law:
\begin{equation*}
p := \sum_{k=1}^N\rho_k  r_k T, \ r_k = \frac{R}{m_k},
\end{equation*}
where $m_k$ is the molar mass of species k and $R$ is the gas constant.  The temperature $T$ is determined by the internal energy $\rho e := \rho e^{t} - (\rho u)^2/(2\rho)$ which in this work is modeled following a calorically perfect gas assumption:
\begin{equation}\label{eq:T}
    \rho e = \sum_{k} \rho_k e_k, \ e_k := e_{0k} + c_{vk} T.
\end{equation}
For species k, $e_{0k}$ is a constant and $c_{vk}$ is the constant volume specific heat. $T$ is computed by solving equation (\ref{eq:T}). An additional conservation equation \cite{Chalot, Giovangigli} can be derived from the governing equations:
\begin{equation}\label{eq:PDEentropy}
    \frac{\partial \rho s}{\partial t} + \frac{\partial \rho u s}{\partial x} = 0.
\end{equation}
$\rho s$ is the thermodynamic entropy of the mixture given by:
\begin{gather*}
\rho s := \sum_{k=1}^N\rho_k s_k, \ 
s_k := c_{vk} \ln({T}) - r_k \ln (\rho_k),
\end{gather*}
where $s$ denotes the specific entropy of the mixture and $s_k$ denotes the specific entropy of species $k$. Equation (\ref{eq:PDEentropy}) can be rewritten in the form of equation (\ref{eq:Entropy}) with $(U, \ F) = (-\rho s, \ -\rho u s)$. For $\rho_k > 0$ and $T > 0$, $U$ is a convex function of the conserved variables and $(U, \ F)$ is a valid entropy pair for the multicomponent system \cite{Chalot, Giovangigli}. \\
\indent At this point, we draw the attention of the reader to the difference between \textit{thermodynamic} entropy $\rho s$ and \textit{mathematical} entropy $U$, which is a general concept proper to hyperbolic PDEs. For the Burgers equation for instance, $U = \frac{1}{2}u^2$ is an entropy. In our context, the mathematical entropy is the opposite of the thermodynamic entropy, and the statement of integral entropy stability can be interpreted either as dissipation of (mathematical) entropy or as production of (thermodynamic) entropy. We adopt the latter throughout this paper. \\
\indent For systems of conservation laws with an entropy pair, Mock \cite{Mock} proved that the mapping $\mathbf{u} \rightarrow \mathbf{v}$ where $\mathbf{v}$ denotes the entropy variables defined as:
\begin{equation*}
    \mathbf{v} := \bigg( \frac{\partial U}{\partial \mathbf{u}}\bigg)^T,
\end{equation*}
symmetrizes the original system. This means that, starting from the quasi-linear form of (\ref{eq:PDE_1D}):
\begin{equation}\label{eq:PDE_QL_1D}
    \frac{\partial \mathbf{u}}{\partial t} + A \frac{\partial \mathbf{u}}{\partial x} = 0, \ A := \frac{\partial \mathbf{f}}{\partial \mathbf{u}},
\end{equation}
where $A$ is the flux Jacobian, the change of variables $\mathbf{u} \rightarrow \mathbf{v}$ turns (\ref{eq:PDE_QL_1D}) into a system of the form:
\begin{equation}
    H\frac{\partial \mathbf{v}}{\partial t} + B \frac{\partial \mathbf{u}}{\partial x} = 0, \ H := \frac{\partial \mathbf{u}}{\partial \mathbf{v}}, \ B := \frac{\partial \mathbf{f}}{\partial \mathbf{v}} = A H,
\end{equation}
where $B$ and $H$ are symmetric, and $H$ is symmetric positive definite (such systems are called symmetric hyperbolic and are well appreciated in the analysis of PDEs \cite{Friedrichs, Giovangigli}). The matrices $A$ and $H$ will be derived in section \ref{sec:ES_flux}. For completeness, we mention that Godunov \cite{Godunov} showed that the converse of Mock's result holds as well, namely that if there exists a mapping which symmetrizes the system, then there exists a valid entropy pair. \\
\indent The fact that the entropy variables symmetrize the system (or equivalently, the fact that the matrix $H$ symmetrizes the flux Jacobian $A$ from the right) will be a key component in constructing the ES dissipation operator in section \ref{sec:ES_flux}. In addition to the entropy variables, the construction of ES schemes also involve potential functions $(\mathcal{U}, \mathcal{F})$ defined by:
\begin{equation*}
    \mathcal{U} := \mathbf{v} \cdot \mathbf{u} - U, \ \mathcal{F} := \mathbf{v} \cdot \mathbf{f} - F.
\end{equation*}
The potential functions satisfy the relationships:
\begin{equation*}
    \mathbf{u} = \bigg(\frac{\partial \mathcal{U}}{\partial \mathbf{v}}\bigg)^T, \ 
    \mathbf{f} = \bigg(\frac{\partial \mathcal{F}}{\partial \mathbf{v}}\bigg)^T.
\end{equation*}
In order to derive them, we must first derive the entropy variables. Following \cite{Chalot, Giovangigli}, we use a chain rule: 
\begin{equation*}
    \frac{\partial U}{\partial \mathbf{u}} = \frac{\partial U}{\partial Z} \bigg ( \frac{\partial \mathbf{u}}{\partial Z}\bigg)^{-1}, \ 
    Z := \begin{bmatrix} \rho_1 & \hdots & \rho_N & u & T \end{bmatrix}^T,
\end{equation*}
where $Z$ denotes the vector of primitive variables. The Gibbs identity can be written as:
\begin{equation}\label{eq:Gibbs}
    T d(\rho s) = d(\rho e) - \sum_{k=1}^Ng_k d\rho_k,
\end{equation}
where $g_k :=  h_k - T s_k$ is the Gibbs function of species k and $h_k := e_k + r_k T$ is the specific enthalpy of species $k$. We have:
\begin{equation}\label{eq:dE}
    d(\rho e) = \sum_{k=1}^Ne_k d \rho_k + \rho c_v dT, \ \rho c_v := \sum_{k=1}^N \rho_k c_{vk}.
\end{equation}
Combining equations (\ref{eq:dE}) and (\ref{eq:Gibbs}), one obtains:
\begin{equation*}
    d(\rho s) = \frac{1}{T} \bigg (\sum_{k=1}^N(e_k - g_k) d\rho_k + \rho c_v dT \bigg).
\end{equation*}
This gives:
\begin{equation}\label{eq:dSdZ}
    \frac{\partial U}{\partial Z} = \frac{1}{T} 
    \begin{bmatrix} (g_1 - e_1) & \hdots & (g_N - e_N) & 0 & -\rho c_v \end{bmatrix}.
\end{equation}
The Jacobian of the mapping $Z \rightarrow \mathbf{u}$ is given by:
\begin{equation}\label{eq:dudz}
    \frac{\partial \mathbf{u}}{\partial Z} = 
        \begin{bmatrix} 
            1 &        & 0 &    0   &    0   \\
              & \ddots &   & \vdots & \vdots \\
            0 &        & 1 &    0   &    0   \\
            u & \hdots & u &  \rho  &    0   \\
            e_{1} + \frac{1}{2}u^2 & \hdots & e_{N} + \frac{1}{2}u^2 & \rho u & \rho c_v 
        \end{bmatrix}.
\end{equation}
The inverse of this matrix is given by:
\begin{equation}\label{eq:dudz_inv}
    \bigg(\frac{\partial \mathbf{u}}{\partial Z}\bigg)^{-1} = 
        \begin{bmatrix} 
            1 &        & 0 &    0   &    0   \\
              & \ddots &   & \vdots & \vdots \\
            0 &        & 1 &    0   &    0   \\
            -u \rho^{-1} & \hdots & -u \rho^{-1} &  \rho^{-1}  &    0   \\
            (\frac{1}{2}u^2 - e_{1})(\rho c_v)^{-1} & \hdots & (\frac{1}{2}u^2 - e_{N})(\rho c_v)^{-1} &  -u (\rho c_v)^{-1} & (\rho c_v)^{-1} 
        \end{bmatrix}
\end{equation}
Combining equations (\ref{eq:dudz_inv}) and (\ref{eq:dSdZ}) yields the entropy variables \cite{Chalot, Giovangigli}:
\begin{equation}\label{eq:v}
    \mathbf{v} = \bigg(\frac{\partial U}{\partial \mathbf{u}}\bigg)^T = \frac{1}{T} 
    \begin{bmatrix} g_1 - \frac{1}{2} u^2 & \hdots & g_N - \frac{1}{2}u^2 & u & -1 \end{bmatrix}^T.
\end{equation}
From this expression, the potential functions $(\mathcal{U}, \ \mathcal{F})$ can be derived:
\begin{equation}\label{eq:fluxpotential}
    \mathcal{U} = \sum_{k=1}^Nr_k \rho_k , \ \mathcal{F} = \sum_{k=1}^Nr_k \rho_k u.
\end{equation}
\indent We conclude this section with two remarks:
\begin{enumerate}
    \item Denote $\mathbf{v} = [v_{1,k} \ \dots \ v_{1,N} \ v_2 \ v_3 ]^T$. In order to derive the mapping from entropy variables to primitive variables, one can first compute the temperature, velocity, gibbs functions and specific entropies as:
\begin{equation*}
    T := -\frac{1}{v_3}, \ u := -\frac{v_2}{v_3}, \ g_k := -\frac{v_{1,k}}{v_3} + \frac{1}{2}\bigg(\frac{v_2}{v_3}\bigg)^2, \ s_k(T,\rho_k) = c_{pk} - v_{1,k} + \frac{1}{2}\frac{v_2^2}{v_3} 
\end{equation*}
The partial densities are inferred from the specific entropies:
\begin{equation*}
    \rho_k := \exp{ \bigg( \frac{m_k}{R} (c_{vk}ln(T) - s_k) \bigg)} =
        \exp{ \Bigg( \frac{m_k}{R} \bigg(-c_{vk}ln(v_3) - c_{pk} + v_{1,k} - \frac{1}{2}\frac{v_2^2}{v_3} \bigg) \Bigg)}.
\end{equation*} 
    The requirement that $\rho_k > 0$ and $T > 0$ manifests in the definition of the entropy variables, which require the evaluation of $\ln(\rho_k)$ and $\ln(T)$. On the other hand, it is interesting to note that if one works with the entropy variables instead of the conservative variables, the requirement that $\rho_k > 0$ and $T > 0$ boils down to the single requirement that $v_3 < 0$. The remaining entropy variables can be of any sign in principle. The authors are not aware of any physical consideration which would impose the sign of $g_k - \frac{1}{2}u^2$, namely the difference between gibbs energy and kinetic energy.
    \item In the compressible Euler case with $e_{0k} = 0$, it is easy to show, using the ideal gas law $p = \rho r T$ and the Mayer relation $c_p - c_v = r$ that the vector of entropy variables (\ref{eq:v}) simplifies to :
    \begin{equation*}
        \mathbf{v} = r \begin{bmatrix} \frac{\gamma - \bar{s}}{\gamma - 1} - \frac{\rho u^2}{2p} & \frac{\rho u}{p} & -\frac{\rho}{p}
        \end{bmatrix}^T, \ \bar{s} = \ln p  - \gamma \ln \rho - \ln r.
    \end{equation*}
    This differs by a constant factor $r$ from the expression that is usually used when designing entropy stable schemes for the compressible Euler equations (see \cite{Barth, Roe1} for instance). The trivial difference comes from a different choice of entropy pair $(U,\ F)$.
\end{enumerate}

\section{Formulation}\label{sec:formulation}
The first subsection covers the definitions and fundamental results relevant to the construction of ES schemes. The following subsections cover the details relevant to the multicomponent system.
\subsection{Background}
Here, the subscript $j$ denotes the cell index, and the subscript $j + \frac{1}{2}$ refers to the interface between cell $j$ and $j+1$. We also use the notation $[.]_{j+\frac{1}{2}}$ to refer to the jump between cell $j+1$ and $j$. \\
\indent EC and ES schemes are defined as follows:
\begin{definition}[\cite{Tadmor}]
The finite volume scheme:
    \begin{equation}\label{eq:FVM}
\frac{d}{dt}\mathbf{u}_{j}(t) + \frac{1}{\Delta x}(\mathbf{f}_{j + \frac{1}{2}} - \mathbf{f}_{j - \frac{1}{2}}) = 0,
\end{equation}
where $\mathbf{f}_{j + \frac{1}{2}}$ is a consistent interface flux, is called entropy conservative if it satisfies the equation: 
 \begin{equation*}
\frac{d}{dt}U(\mathbf{u}_{j}(t)) + \frac{1}{\Delta x}[F_{j + \frac{1}{2}} - F_{j - \frac{1}{2}}] = 0,
\end{equation*}
where $F_{j+\frac{1}{2}}$ is a consistent entropy interface flux, and entropy stable if it satisfies the inequality:
\begin{equation}\label{eq:FVM_INEQ}
   \frac{d}{dt}U(\mathbf{u}_{j}(t)) + \frac{1}{\Delta x}[F_{j + \frac{1}{2}} - F_{j - \frac{1}{2}}] < 0. 
\end{equation}

\end{definition}
\indent The first step in the construction of an ES scheme is to construct an EC scheme, following:
\begin{theorem}[\cite{Tadmor}]
    The finite volume scheme (\ref{eq:FVM}) is entropy conservative if and only if the interface flux $\mathbf{f}_{j+\frac{1}{2}}$ satisfies the condition:
\begin{equation}\label{eq:ECcond0}
    [\mathbf{v}]_{j+\frac{1}{2}} \cdot \mathbf{f}_{j+\frac{1}{2}} = [\mathcal{F}]_{j+\frac{1}{2}},
\end{equation}
where $\mathcal{F}_{j} = \mathcal{F}(\mathbf{u}_j)$, is the potential function evaluated at cell $j$. In this case, $\mathbf{f}_{j+\frac{1}{2}}$ is called an entropy-conservative flux and the corresponding entropy flux $F_{j+\frac{1}{2}}$ is given by the formula:
\begin{equation*}
    F_{j+\frac{1}{2}} = \frac{1}{2}(\mathbf{v}_j + \mathbf{v}_{j+1}) \cdot \mathbf{f}_{j+\frac{1}{2}} - \frac{1}{2}(\mathcal{F}_j + \mathcal{F}_{j+1}).
\end{equation*}
\end{theorem}
For scalar PDEs, $\mathbf{v}$ is a scalar and the entropy conservation condition (\ref{eq:ECcond0}) has only one solution $\mathbf{f}_{j+\frac{1}{2}} = [\mathcal{F}]_{j+\frac{1}{2}}/[\mathbf{v}]_{j+\frac{1}{2}}$. For systems, equation (\ref{eq:ECcond0}) does not uniquely determine the interface flux. The first two EC fluxes proposed by Tadmor \cite{Tadmor, Tadmor_acta} have the inconvenience of either not having a closed-form \cite{Tadmor} expression or being relatively expensive to compute \cite{Tadmor_acta}. Using algebraic manipulations analogous to that of \cite{Roe}, Roe \cite{Roe1} proposed a simple, closed-form EC flux for the Euler equations that is more popular. The EC flux we derive for the multicomponent system in section \ref{sec:EC_flux} uses the same method. \\
\indent An ES scheme is built by complementing an EC scheme with appropriately defined dissipation operators as follows:
\begin{theorem}[\cite{Tadmor}]\label{theorem:ES}
    The finite-volume scheme (\ref{eq:FVM}) with the interface flux $\mathbf{f}_{j+\frac{1}{2}}$ defined as:
    \begin{equation*}
    \mathbf{f}_{j+\frac{1}{2}} = \mathbf{f}^*_{j+\frac{1}{2}} - D_{j+\frac{1}{2}}[\mathbf{v}]_{j+\frac{1}{2}}.
    \end{equation*}
    where $f^{*}_{j+\frac{1}{2}}$ satisfies the entropy conservation condition (\ref{eq:ECcond0}), and $D_{j+\frac{1}{2}}$ is a positive definite matrix, satisfies
    \begin{equation}\label{eq:FVM_ES}
    \frac{d}{dt}U(\mathbf{u}_{j}(t)) + \frac{1}{\Delta x}[F_{j + \frac{1}{2}} - F_{j - \frac{1}{2}}] = - \mathcal{E}_j.
    \end{equation}
    where $\mathcal{E}_j$ is given by:
    \begin{equation}
        \mathcal{E}_j = \frac{1}{2 \Delta x}([\mathbf{v}]_{j+\frac{1}{2}}^T D_{j+\frac{1}{2}} [\mathbf{v}]_{j+\frac{1}{2}} + [\mathbf{v}]_{j-\frac{1}{2}}^T D_{j-\frac{1}{2}} [\mathbf{v}]_{j-\frac{1}{2}}) > 0,
    \end{equation}
    and is therefore entropy stable. In this case, the interface entropy flux $F_{j+\frac{1}{2}}$ is given by:
    \begin{equation}\label{eq:SESflux}
        F_{j+\frac{1}{2}} = F^{*}_{j+\frac{1}{2}} - \frac{1}{2}(\mathbf{v}_{j}+\mathbf{v}_{j+1}) D_{j+\frac{1}{2}}[\mathbf{v}]_{j+\frac{1}{2}},
    \end{equation}
    where $F^{*}_{j+\frac{1}{2}}$ is the entropy flux associated with $f^{*}_{j+\frac{1}{2}}$.
\end{theorem}
Summing over all cells and assuming periodic boundary conditions leads to the  integral entropy stability statement:
\begin{equation}\label{eq:ES_global}
    \frac{d}{dt} \sum_j U(\mathbf{u}_j) = - \sum_j \mathcal{E}_j < 0.
\end{equation}
This integral stability statement is obtained as a consequence of the local statement (\ref{eq:FVM_ES}), which in itself is not a stability statement. It does not necessarily imply for instance that in every cell $j$:
\begin{equation*}
    \frac{d}{dt}U(\mathbf{u}_j) < 0.
\end{equation*}
Looking at equation (\ref{eq:FVM_ES}), we see that the local variation of $U$ in time depends on both $\mathcal{E}_j$, which is determined by the dissipation operator, and the interface entropy flux $F_{j+\frac{1}{2}}$, which according to equation (\ref{eq:SESflux}) is determined by both the EC flux and the dissipation operator. \\
\indent The second principle of thermodynamics is often invoked to defend the consistency of ES schemes with physics. First, this principle only applies to closed systems, hence it can hardly be invoked when the boundary conditions are non-trivial. Second, the second principle is not a local statement. This means that it  does not explicitly support the local inequality (\ref{eq:FVM_INEQ}).   

\subsection{Entropy-conservative flux}\label{sec:EC_flux}
In compact notation, the entropy conservation condition (\ref{eq:ECcond0}) writes:
\begin{equation}\label{eq:ECcond1}
    [\mathbf{v}] \cdot \mathbf{f}^{*} = [\mathcal{F}],
\end{equation}
where $\mathbf{f}^{*} = [f_{1,1} \ \dots \ f_{1,N} \ f_2 \ f_3]$ denotes the interface flux. Define the set of algebraic variables:
\begin{equation*}
    \mathbf{z} = \begin{bmatrix} \rho_1 & \hdots & \rho_N & u & \frac{1}{T} \end{bmatrix} = \begin{bmatrix} z_{1,1} & \hdots & z_{1,N} & z_2 & z_3 \end{bmatrix}.
\end{equation*}
The first step of Roe's technique \cite{Roe1} consists in rewriting the jump terms in equation (\ref{eq:ECcond1}) as linear combinations of the jumps in the algebraic variables. For a given quantity $a$, let $\overline{a}$ and $a^{ln}$ be the arithmetic and logarithmic averages, respectively:
\begin{equation*}
    \overline{a} = \frac{1}{2}(a_L+a_R), \ a^{ln} = \begin{cases}
        a_L, \ \ \mbox{if $a_L = a_R$} \\
        0, \ \ \mbox{if $a_L = 0$ or $a_R = 0$} \\
        \frac{a_R - a_L}{\ln(a_R)-\ln(a_L)}, \ \ \mbox{else}
    \end{cases}
\end{equation*}
Note the identities:
\begin{equation*}
    [ab] = \overline{a}[b] + \overline{b}[a], \ \ \mbox{and} \ \ [\ln(a)] = [a]/a^{ln}.
\end{equation*}
The jump in the potential function can be rewritten as:
\begin{equation} \label{eq:jumps_1}
    [\mathcal{F}] = \sum_{k=1}^Nr_k [\rho_k u] = \bigg(\sum_{k=1}^Nr_k \overline{z_{1,k}}\bigg) [z_2] + \sum_{k=1}^Nr_k \overline{z_2} [z_{1,k}].
\end{equation}
For the jump in entropy variables, we need to examine the first N components. For $1 \leq k \leq N$:
\begin{align*}
\frac{1}{T}(g_k - \frac{1}{2}u^2) = \frac{e_{0k}}{T} + c_{vk}  + r_k - c_{vk}\ln T + r_k\ln(\rho_k) - \frac{1}{2}\frac{u^2}{T} = e_{0k}z_3 + c_{vk} + r_k 
+ c_{vk}\ln(z_3) + r_k \ln(z_{1,k}) - \frac{1}{2}z_2^2 z_3.
\end{align*}
The corresponding jumps then write:
\begin{align}
\bigg[\frac{1}{T}(g_k - \frac{1}{2}u^2)\bigg] =& \ [z_{1,k}] \frac{r_k}{z_{1,k}^{ln}} - [z_2] \overline{z_2} \ \overline{z_3} + [z_3] ( e_{0k} + \frac{c_{vk}}{z_3^{ln}} - \frac{1}{2}\overline{z_2^2}) \label{eq:jumps_1N}
\end{align}
The remaining jumps are given by:
\begin{equation} \label{eq:v_jumps23}
    \bigg[\frac{u}{T}\bigg] = \overline{z_3}[z_2] + \overline{z_2}[z_3], -\bigg[\frac{1}{T}\bigg] = -[z_3].
\end{equation}
Using equations (\ref{eq:jumps_1}), (\ref{eq:jumps_1N}) and (\ref{eq:v_jumps23}), the entropy conservation condition (\ref{eq:ECcond1}) can be rewritten as the requirement that a linear combination of the jumps in the algebraic variables equals zero:
\begin{multline*}
    \sum_{k=1}^N[z_{1,k}] \bigg(\frac{r_k}{z_{1,k}^{ln}} f_{1,k}\bigg) + [z_2] \bigg( (-\overline{z_3} \ \overline{z_2})\sum_{k=1}^Nf_{1,k} + \overline{z_3} f_2 \bigg) + [z_3] \bigg(  \sum_{k=1}^N(e_{0k} + c_{vk}\frac{1}{z_3^{ln}} - \frac{1}{2}\overline{z_2^2}) f_{1,k} + \overline{z_2} f_2 - f_3 \bigg) = \\
        [z_2] \bigg(\sum_{k=1}^Nr_k \overline{z_{1,k}}\bigg)  + \sum_{k=1}^N[z_{1,k}] r_k \overline{z_2}. 
\end{multline*}
The second step turns this scalar condition into
a system of $N + 3$ equations by invoking the independence of these jumps:
\begin{gather*}
   r_k \frac{1}{z_{1,k}^{ln}} = r_k \overline{z_2} f_{1,k}, \ 1 \leq k \leq N, \ \ 
   (-\overline{z_3} \ \overline{z_2})\sum_{k=1}^Nf_{1,k} + \overline{z_3} f_2 =  \bigg(\sum_{k=1}^Nr_k \overline{z_{1,k}}\bigg), \ \
    \sum_{k=1}^N(e_{0k} + c_{vk}\frac{1}{z_3^{ln}} - \frac{1}{2}\overline{z_2^2}) f_{1,k} + \overline{z_2} f_2 - f_3 = 0.
\end{gather*}
The solution under these assumptions is:
\begin{align*}
    f_{1,k} =& \ \overline{z_2} z_{1,k}^{ln}, \\
    f_2     =& \ \frac{1}{\overline{z_3}} \bigg(\sum_{k=1}^Nr_k \overline{z_{1,k}}\bigg) + \overline{z_2} \sum_{k=1}^Nf_{1,k}, \\
    f_3     =& \ \sum_{k=1}^N(e_{0k} + c_{vk}\frac{1}{z_3^{ln}} - \frac{1}{2}\overline{z_2^2}) f_{1,k} + \overline{z_2} f_2.
\end{align*}
Therefore we obtain:
\begin{theorem}
    The interface flux $f^{*} = [f_{1,1} \ \dots \ f_{1,N} \ f_{2} \ f_{3}]$ defined by:
    \begin{align}\label{eq:ECfluxMC}
    f_{1,k} =& \  \rho_{k}^{ln}\overline{u}, \ 1 \leq k \leq N \nonumber \\
    f_2     =& \ \frac{1}{\overline{1/T}} \bigg(\sum_{k=1}^Nr_k \overline{\rho_k}\bigg) + \overline{u} \sum_{k=1}^Nf_{1,k}, \\
    f_3     =& \ \sum_{k=1}^N(e_{0k} + c_{vk}\frac{1}{(1/T)^{ln}} - \frac{1}{2}\overline{u^2}) f_{1,k} + \overline{u} f_2. \nonumber
\end{align}
is entropy conservative for the multicomponent system defined in section 2.
\end{theorem}
\indent This EC flux is the multicomponent version of Chandrasekhar's EC flux \cite{Chandrasekhar} for the compressible Euler equations. Chandrasekhar's flux was designed with the additional property of being Kinetic-Energy Preserving (KEP) in the sense of Jameson \cite{Jameson}, meaning that the kinetic energy equation is satisfied by the finite volume scheme (\ref{eq:FVM}) in a semi-discrete sense (in the same spirit as EC schemes). This property can be useful in turbulent flow simulations \cite{Subbareddy}. For the compressible Euler equations, Jameson \cite{Jameson} showed that this is achieved if the momentum flux $f^{\rho u}$ has the form $f^{\rho u} = \Tilde{p} + \overline{u} f^{\rho}$, where $f^{\rho}$ is the mass flux and $\tilde{p}$ is any consistent pressure average. The extension of Jameson's analysis \cite{Jameson} to the multicomponent case is straightforward and it can be showed that if the momentum flux has the same form as in the single component case (with $f^{\rho}$ denoting the \textit{total} mass flux), the KEP property is achieved. The EC flux we derived qualifies, with $\tilde{p}$ given by:
\begin{equation*}
    \tilde{p} = \frac{1}{\overline{(1/T)}}\sum_{k=1}^Nr_k \overline{\rho_k}.    
\end{equation*}
Note in passing that if $[p] = 0$ across the discontinuity, then $\tilde{p}$ defined above is exactly the pressure on each side. \\
\indent The flux (\ref{eq:ECfluxMC}) is well-defined in the limit case $\rho_k = 0$, where the entropy variable $v_{1,k}$ is undefined, but is the entropy conservation condition still met? Equation (\ref{eq:ECcond0}) writes 
\begin{equation*}
    \sum_{k=1}^N[v_{1,k}]f_{1,k} + [v_2]f_{2} + [v_3]f_{3} = [\mathcal{F}] .  
\end{equation*}    
The jumps in $v_{1,k}$ are undefined because they involve jumps in $\ln(\rho_{k})$. However, we note that the first term:
\begin{align*}
\sum_{k=1}^N[v_{1,k}] f_{1,k} =& \ \sum_{k=1}^N\bigg[\frac{h_k}{T} - \frac{1}{2}\frac{u^2}{T}\bigg] (\rho_k^{ln}\overline{u}) \ - \ \sum_{k} [s_k] (\rho_k^{ln}\overline{u}) \\
                         =& \ \sum_{k=1}^N\bigg[\frac{h_k}{T} - \frac{1}{2}\frac{u^2}{T}\bigg] (\rho_k^{ln}\overline{u}) \ - \ \sum_{k} c_{vk} [\ln T] (\rho_k^{ln}\overline{u}) + \ \sum_{k=1}^Nr_k [\ln \rho_k] (\rho_k^{ln}\overline{u}) \\
                         =& \ \sum_{k=1}^N\bigg[\frac{h_k}{T} - \frac{1}{2}\frac{u^2}{T}\bigg] (\rho_k^{ln}\overline{u}) \ - \ \sum_{k} c_{vk} [\ln T] (\rho_k^{ln}\overline{u}) + \ \sum_{k=1}^Nr_k [\rho_k] \overline{u}
\end{align*}
is well-defined at $\rho_k = 0$, because the logarithmic averages $\rho_k^{ln}$ compensated for the logarithmic jumps $[\ln \rho_k]$. We thus find that the entropy conservation condition is satisfied even in the limit $\rho_k = 0$. \\
\indent Note that the EC flux does not transfer mass across an interface separating two different species. 
\subsection{Entropy-stable flux}\label{sec:ES_flux}
\subsubsection{Upwind-type dissipation operator}
Let $A = R \Lambda R^{-1}$ be an eigendecomposition of $A$. A popular choice for the dissipation operator consists of recasting the upwind operator of Roe \cite{Roe} $\rightarrow \frac{1}{2} R|\Lambda|R^{-1}[\mathbf{u}]$ in terms of the entropy variables. With the differential relation $d\mathbf{u} = H d\mathbf{v}$, this leads to:
\begin{equation*}
    D[\mathbf{v}] = \frac{1}{2} R |\Lambda| R^{-1} H [\mathbf{v}].
\end{equation*}
For the compressible Euler equations, Merriam (\cite{Merriam}, section 7.3) pointed out that there exists a scaling of the columns of $R$ such that $H = R R^T$, which ultimately leads to a dissipation matrix $D = R |\Lambda| R^T$ which has the desirable property of being positive definite (Theorem \ref{theorem:ES}). Barth \cite{Barth} generalized this result to hyperbolic systems with a convex extension. The generalization takes the form of an eigenscaling theorem (Theorem 4 in \cite{Barth}) which states that for any diagonalizable matrix $A$ symmetrized on the right by a symmetric positive definite matrix $H$, there exists a symmetric block diagonal matrix $T$ that block scales the eigenvectors $R$ of $A$ in such a manner that:
\begin{equation}\label{eq:EigenScaling}
    A = (RT) \Lambda (RT)^T \ \ \mbox{and} \ \ H = (RT) (RT)^T.
\end{equation}
The dimensions of the blocks of $T$ correspond to the multiplicities of the eigenvalues of A. The second identity in equation (\ref{eq:EigenScaling}) provides an explicit expression for the squared scaling matrix $T^2 = R^{-1} H R^{-T}$. \\
\indent We now proceed to derive the scaling matrix for the multicomponent Euler system. First, the flux Jacobian is given by:
\begin{equation*}
    A = \frac{\partial \mathbf{f}}{\partial \mathbf{u}} = 
        \begin{bmatrix} 
            u(1 - Y_1) & \hdots &    -u Y_1  &  Y_1   &    0   \\
                               & \ddots &                    &  \vdots   &    0   \\
             - u Y_N    & \hdots & u(1 - Y_N) &  Y_N   &    0   \\
            \frac{(\gamma - 3)}{2}u^2 + d_1 & \hdots & \frac{(\gamma - 3)}{2}u^2 + d_N &  u(3 - \gamma)  & \gamma - 1   \\
           u(d_1 - h^t + \frac{u^2}{2}(\gamma - 1)) & \hdots & u(d_N - h^t + \frac{u^2}{2}(\gamma - 1)) &  h^t - u^2(\gamma-1)  & u \gamma 
        \end{bmatrix},
\end{equation*}
where $d_k = h_k - \gamma e_k$ and $h^t = h + \frac{1}{2}u^2$. We have $A = R \Lambda R^{-1}$ with:
\begin{gather*} 
   R = \begin{bmatrix}
            1 &        & 0 &    Y_1   &   Y_1  \\
              & \ddots &   &  \vdots  & \vdots \\
            0 &        & 1 &    Y_N   &   Y_N  \\
            u & \hdots & u &   u + a  &  u - a \\
            k - \frac{d_1}{\gamma - 1} & \hdots & k - \frac{d_N}{\gamma - 1} & h^t + u a & h^t - u a
        \end{bmatrix}, \  \Lambda = diag(u, \ \hdots , \ u, \ u + a, \ u -a), \ a^2 = \gamma r T,
\end{gather*}
where $k = \frac{u^2}{2}$. The Jacobian $H$ of the mapping $\mathbf{v} \rightarrow \mathbf{u}$ is given by \cite{Giovangigli}:
\begin{equation}\label{eq:dudv_1D}
    H = \frac{\partial \mathbf{u}}{\partial \mathbf{v}}  = 
    \begin{bmatrix}
       \rho_1/r_1 &        &             0              &  u\rho_1/r_1 & \rho_1 e_1^t / r_1 \\
                             & \ddots &                            &        \vdots          &              \vdots                        \\
                 0           &        &    \rho_N/r_N    &  u \rho_N/r_N & \rho_N e_N^t / r_1 \\
       u \rho_1/r_1 & \hdots &   u \rho_N/r_N    &  \rho T + u^2 \sum_{k=1}^N\rho_k/r_k & u(\rho T + \sum_{k=1}^N\rho_k e_k^t / r_k ) \\
        \rho_1 e_1^t / r_1 & \hdots &  \rho_N e_N^t/ r_N & u \rho T + u \sum_{k=1}^N\rho_k e_k^t / r_k  & 
        \rho T (u^2 + c_v T) + \sum_{k=1}^N\rho_k(e_k^t)^2 / r_k
    \end{bmatrix},
\end{equation}
where $e_k^t = e_k + k$. The squared scaling matrix is given by:
\begin{equation}\label{eq:T2}
    T^2 =  R^{-1}H R^{-T} = \frac{\rho}{\gamma r} diag(T^{2Y}, \ 1/2, \ 1/2),
\end{equation}
where $T^{2Y} \in \mathbb{R}^{N \times N}$ is given by:
\begin{align*}
    T^{2Y}_{ii} =& \ (\gamma - 1)Y_i^2 + \sum_{k \neq i} (\gamma r_k/r_i) Y_k Y_i, \ 1 \leq i \leq N, \\
    T^{2Y}_{ij} =& \ - Y_i Y_j, \ 1 \leq i \neq j \leq N.
\end{align*}
At this point, the dissipation operator writes:
\begin{equation}\label{eq:ES_diss1}
    D[\mathbf{v}] = \frac{1}{2} R |\Lambda| T^{2} R^{T} [\mathbf{v}], 
\end{equation}
and qualifies for the ES scheme because the matrix $R |\Lambda| T^{2} R^{T}$ is positive definite ($T^2$ and $|\Lambda|$ commute). However, as will be seen in section \ref{sec:TecNO}, a scaled form (\ref{eq:EigenScaling}) of the dissipation operator is necessary. \\
\indent For $N = 2$, we have:
\begin{equation*}
    T^{2Y} = \begin{bmatrix}
           (\gamma-1)Y_1^2 + (\gamma r_2/r_1) Y_1 Y_2  & -Y_1 Y_2 \\
            -Y_1 Y_2 & (\gamma - 1) Y_2^2 + (\gamma r_1/r_2) Y_1 Y_2 
          \end{bmatrix}.
\end{equation*}
$T^{2Y}$ is symmetric, real-valued with non-negative eigenvalues therefore there exists $T^{Y}$ with the same properties such that $T^{2Y} = (T^{Y})^2 = T^{Y} (T^{Y})^T$ ($T^{Y}$ is the square root of $T^{2Y}$). This matrix can be derived using an eigenvalue decomposition. However the expression of $T^{Y}$ is quite complicated. The square root of $T^{2Y}$ is not necessary to proceed. For $N=2$, $T^{2Y}$ can be rewritten as:
\begin{equation}\label{eq:fake_tilde}
    T^{2Y} = \mathcal{T}^{Y} (\mathcal{T}^{Y})^T, \ 
    \mathcal{T}^{Y} = 
    \begin{bmatrix}
        -\sqrt{Y_1Y_2} \sqrt{\gamma r_2/r_1} & Y_1 \sqrt{\gamma-1} \\
        \sqrt{Y_1Y_2} \sqrt{\gamma r_1/r_2} & Y_2 \sqrt{\gamma-1}
    \end{bmatrix}.
\end{equation}
$\mathcal{T}^{Y}$ is not the square root of $T^{2Y}$, however it is enough to obtain a scaled formulation (\ref{eq:EigenScaling}) because:
\begin{equation*}
    T^2 = \mathcal{T} \mathcal{T}^T, \ \mathcal{T} =  \sqrt{\frac{\rho}{\gamma r}} diag(\mathcal{T}^{Y}, \ 1/\sqrt{2}, \ 1/\sqrt{2}),
\end{equation*}
and $\mathcal{T}$ commutes with $|\Lambda|$ therefore the dissipation operator can be rewritten as:
\begin{equation}\label{eq:ESflux}
    D[\mathbf{v}] = \frac{1}{2} (R \mathcal{T}) |\Lambda | (R \mathcal{T})^T [\mathbf{v}].
\end{equation}
For $N > 2$, the expression for $T^{Y}$ becomes even more complicated. For $N = 3$, the alternative (\ref{eq:fake_tilde}) to $T^{Y}$ we proposed for $N = 2$ becomes: 
\begin{equation*}
    T^{2Y} = \mathcal{T}^Y (\mathcal{T}^Y)^T, \ 
    \mathcal{T}^{Y} = 
        \begin{bmatrix}
           -\sqrt{Y_1 Y_2} \sqrt{\gamma r_2/r_1} & -\sqrt{Y_1 Y_3} \sqrt{\gamma r_3/r_1}& 0 & -Y_1 \sqrt{\gamma-1} \\
           \sqrt{Y_1 Y_2} \sqrt{\gamma r_1/r_2} & 0 & -\sqrt{Y_2 Y_3} \sqrt{\gamma r_3/r_2} & -Y_2 \sqrt{\gamma-1} \\
           0 & \sqrt{Y_1 Y_3} \sqrt{\gamma r_1/r_3} &  \sqrt{Y_2 Y_3} \sqrt{\gamma r_2/r_3} & -Y_3 \sqrt{\gamma-1}
        \end{bmatrix}.
\end{equation*}
There is one more column compared to the $N=2$ case. The form (\ref{eq:ESflux}) still holds except that $R\mathcal{T} \in \mathbb{R}^{3 \times 4}$ instead of $\mathbb{R}^{3 \times 3}$ and $|\Lambda| \in \mathbb{R}^{4 \times 4}$ diagonal with an extra $|u|$ term. For $N$ species, the "pseudo" scaling matrix $\mathcal{T}^Y$ we described will be in $\mathbb{R}^{N \times (N(N-1)/2+1)}$. 
\subsubsection{Average state}
\indent To complete the definition of the dissipation operator, an average state (referred to with the superscript $*$) needs to be specified. \\
\indent We showed in section \ref{sec:EC_flux} that the EC flux (\ref{eq:ECfluxMC}) is well defined in the limit $\rho_k = 0$. What about the dissipation operator $R |\Lambda| T^2 R^{T} [\mathbf{v}]$? At first glance, the presence of the jump term $[\mathbf{v}]$ is problematic, because $[v_{1,k}]$ is undefined. For $N = 2$, the squared scaling matrix can be rewritten as:
\begin{equation}\label{eq:R}
    T^2 = \overline{T}^2 \mathcal{R}, \ \overline{T}^2 = \frac{1}{\gamma r}\begin{bmatrix}
            (\gamma-1) Y_1 +  \gamma r_2/r_1 Y_2  & -Y_2  & 0 & 0 \\
            -Y_1  & (\gamma-1) Y_2 + \gamma r_1/r_2 Y_1  & 0 & 0 \\
            0 & 0 & \frac{1}{2} & 0 \\
            0 & 0 & 0 & \frac{1}{2}
          \end{bmatrix}, \ \mathcal{R} = 
          \begin{bmatrix}
          \rho_1 &   0    &   0   &    0  \\
            0    & \rho_2 &   0   &    0  \\
            0    &   0    & \rho  &    0  \\
            0    &   0    &   0   &  \rho
          \end{bmatrix}.
\end{equation}
Denoting $D_k = k - \frac{d_k}{\gamma - 1}$, we have:
\begin{multline}\label{eq:rhoRdv}
   \mathcal{R}R^T[\mathbf{v}] = \begin{bmatrix}
          \rho_1^* &   0    &   0   &    0  \\
            0    & \rho_2^* &   0   &    0  \\
            0    &   0    & \rho^*  &    0  \\
            0    &   0    &   0   &  \rho^*
    \end{bmatrix}
    \begin{bmatrix}
            1   &  0  &  u^*  & D_1^* \\
            0   &  1  &  u^*  & D_2^*  \\
            Y_1^* & Y_2^* & u^*+a^* &  (h^t)^* + u^* a^* \\
            Y_1^* & Y_2^* & u^*-a^* &  (h^t)^* - u^* a^* 
    \end{bmatrix}
    \begin{bmatrix}
      \big [ v_{1,1} \big ] \\ \big [ v_{1,2} \big ] \\ \big [ v_2 \big ] \\ \big [ v_3 \big ]
    \end{bmatrix}
    = \\ \begin{bmatrix}
    \rho_1^* \big [ v_{1,1} \big ] + \rho_1^* u \big [ v_2 \big ] + \rho_1^* D_1^* \big [ v_3 \big ] \\
    \rho_2^* \big [ v_{1,2} \big ] + \rho_2^* u \big [ v_2 \big ] + \rho_2^* D_2^* \big [ v_3 \big ] \\
    \rho_1^* \big [ v_{1,1} \big ] + \rho_2^* \big [ v_{1,2} \big ] + \rho^* (u^* + a^*) \big [ v_2 \big ] + \rho^* ((h^t)^* + u^* a^*) \big [ v_3 \big ] \\
    \rho_1^* \big [ v_{1,1} \big ] + \rho_2^* \big [ v_{1,2} \big ] + \rho^* (u^* - a^*) \big [ v_2 \big ] + \rho^* ((h^t)^* - u^* a^*) \big [ v_3 \big ]
    \end{bmatrix}
\end{multline}
We can see that $\mathcal{R}R^{T}[\mathbf{v}]$ is well-defined if $\rho_k^{*} [v_{1,k}]$ is well-defined as well. Given that:
\begin{equation*}
    \rho_k^* [v_{1,k}] = \bigg(- c_{vk}[ln(T)] - \bigg[\frac{u^2}{2 T}\bigg] \bigg) \rho_k^* + \frac{R}{m_k} [\ln \rho_k] \rho_k^*,
\end{equation*}
we see that with $\rho_k^* = \rho_k^{ln}$ the dissipation operator is well-defined. For total density, one might be tempted to take $\rho^* = \sum_{k=1}^N\rho_k^{*} = \sum_{k=1}^N \rho_k^{ln}$. This definition makes $\rho^* = 0$ possible, which is undesirable given that $Y_k^* = \rho_k^* / \rho^*$. $\rho^* =\rho^{\ln}$ is a safer choice. \\
\indent The exact resolution of stationary contact discontinuities is a desirable property in the calculation of boundary and shear layers (even though it might produce carbuncles on blunt-body calculations, see \cite{Quirk_carbuncle} paragraph 2.4.). In this case, $[p] = 0, \ [u] = \overline{u} = 0$ and the EC flux we derived reduces to:
\begin{equation*}
    f_{1,k} = 0, \ f_2 = \frac{1}{\overline{1/T}} \bigg(\sum_{k=1}^N\frac{R}{m_k} \overline{\rho_k}\bigg), \ f_3 = 0.
\end{equation*}
From the ideal gas law we can state that $f_2$ is exactly the pressure on both sides of the contact. The dissipation term must therefore cancel out if we want the ES scheme to exactly preserve stationary contact discontinuities.
\begin{lemma}
 The dissipation operator given by equation (\ref{eq:ES_diss1}) vanishes at stationary contact discontinuities if the averaged state satisfies the relationship:
 \begin{equation}\label{eq:StatContact}
     \rho^* h^* = \sum_{k=1}^N\rho_k^{*} \bigg( e_{0k} + c_{vk} \frac{1}{(1/T)^{ln}}\bigg) + \overline{p}, \ \rho_k^{*} = \rho_k^{ln}.
 \end{equation}
\end{lemma}
\begin{proof}
Since $u = 0$, we have: 
\begin{gather*}
    R^T = \begin{bmatrix}
            1   &        &  0  &     0    &   d_1^*/(\gamma-1)  \\
                & \ddots &     &  \vdots  &             \vdots                \\
            0   &        &  1  &     0    &   d_N^*/(\gamma-1)  \\
            Y_1^* & \hdots & Y_N^* &     a^*    &  h^*             \\
            Y_1^* & \hdots & Y_N^* &    -a^*    &  h^* 
        \end{bmatrix}, \ 
    [\mathbf{v}] = \begin{bmatrix} 
                        \big[\frac{g_1}{T}\big] & \hdots & \big[\frac{g_N}{T}\big] & 0 & -\big[\frac{1}{T}\big]
                    \end{bmatrix}^T,
\end{gather*}
therefore:
\begin{equation*}
    R^{T}[\mathbf{v}] = 
        \begin{bmatrix}
            \big[ \frac{g_1}{T} \big] - d_1^*/(\gamma-1) \big[ \frac{1}{T} \big] \\
            \vdots \\
            \big[ \frac{g_N}{T} \big] - d_N^*/(\gamma-1) \big[ \frac{1}{T} \big] \\
            \sum_{k=1}^NY_k^* \big[\frac{g_k}{T}\big] - \big[\frac{1}{T}\big]h^* \\
            \sum_{k=1}^NY_k^* \big[\frac{g_k}{T}\big] - \big[\frac{1}{T}\big]h^*
        \end{bmatrix}.
\end{equation*}
$|u^*| = 0$ so the product of the eigenvalue matrix $|\Lambda|$ and the squared scaling matrix $T^2$ simplifies to:
\begin{equation*}
    |\Lambda| T^2 = \frac{\rho^* a^*}{2 \gamma^* r^*}
            \begin{bmatrix}
                0 & 0 & 0 & 0 \\
                0 & 0 & 0 & 0 \\
                0 & 0 & 1 & 0 \\
                0 & 0 & 0 & 1
            \end{bmatrix}.
\end{equation*}
Therefore, the dissipation term $R |\Lambda|T^{2}R^{T}[\mathbf{v}]$ cancels out if:
\begin{equation*}
    \sum_{k=1}^NY_k^* \bigg[\frac{g_k}{T}\bigg] - \bigg[\frac{1}{T}\bigg]h^* = 0.
\end{equation*}
This is equivalent to:
\begin{equation}\label{eq:SC1}
    \sum_{k=1}^N\rho_k^* \bigg( e_{0k}\bigg[ \frac{1}{T}\bigg] -\big[s_k\big] \bigg) - \bigg[\frac{1}{T}\bigg] \rho^* h^* = 0
\end{equation}
Using:
\begin{equation*}
    -[s_k] = c_{vk} \bigg[\frac{1}{T}\bigg]\frac{1}{(1/T)^{ln}} + r_k \frac{[\rho_k]}{\rho^{ln}},
\end{equation*}
equation (\ref{eq:SC1}) can then be rewritten as:
\begin{equation}\label{eq:SC2}
    \sum_{k=1}^N \bigg(\frac{\rho_k^*}{\rho_k^{ln}}\bigg) r_k[\rho_k]  - \bigg[\frac{1}{T}\bigg] \bigg(\rho^* h^* - \sum_{k=1}^N\rho_k^* \bigg(e_{0k} + c_{vk}\frac{1}{(1/T)^{ln}}\bigg)\bigg)= 0
\end{equation}
The ideal gas law along with the assumption of constant pressure allows us to relate the jumps in partial densities and temperature:
\begin{equation}\label{eq:cons_p}
    \sum_{k=1}^Nr_k [\rho_k] = \overline{p} \bigg[\frac{1}{T}\bigg].
\end{equation}
If $\rho_k^* = \rho_k^{ln}$, then using equation (\ref{eq:cons_p}), equation (\ref{eq:SC2}) simplifies to:
\begin{equation*}
    \bigg[ \frac{1}{T} \bigg] \bigg( \bar{p} + \sum_{k=1}^N\rho_k^* \bigg( e_{0k} + c_{vk} \frac{1}{(1/T)^{ln}}\bigg) - \rho^* h^* \bigg) = 0.
\end{equation*}
This leads to the condition (\ref{eq:StatContact})
\end{proof}
The remaining averages are taken as $u^* = \overline{u}, \ T^* = 1/(1/T)^{ln}, \ r^* = \bar{r}, \ \gamma^* = \overline{\gamma}$ and $a^* = \sqrt{\gamma^* r^* T^*}$. 
\subsection{Additional considerations}
\subsubsection{Time integration}\label{sec:TE}
\indent The first-order finite volume scheme we derived is ES at the semi-discrete level only. Entropy stability or entropy conservation at the fully discrete level can be obtained using a variety of techniques \cite{Barth, LeFloch, Tadmor_acta, DeepThesis, Diosady, Gouasmi1, Gouasmi2, Fried} which can be applied to the multicomponent compressible Euler system. However, this typically requires implicit time-integration schemes. For simplicity, we use explicit Runge-Kutta schemes in time, which do not guarantee entropy stability at the fully discrete level. \\
\indent For the sake of completeness, is it worthwhile to recall the two fundamental results by Tadmor \cite{Tadmor_acta} at the fully discrete level. Denote $\lambda = \Delta t / \Delta x$, and consider the Forward Euler (FE) scheme in time applied to (\ref{eq:FVM}):
\begin{equation}\label{eq:FVM_FE}
    \mathbf{u}_j^{n+1} - \mathbf{u}_j^{n} + \lambda (\mathbf{f}_{j+\frac{1}{2}}^{n} - \mathbf{f}_{j-\frac{1}{2}}^{n}) = 0,
\end{equation}
where the interface flux can be either EC or ES. Tadmor showed that (\ref{eq:FVM_FE}) implies the following equation for entropy:
\begin{equation}
    U(\mathbf{u}_j^{n+1}) - U(\mathbf{u}_j^{n}) + \lambda (F_{j+\frac{1}{2}}^{n} - F_{j-\frac{1}{2}}^{n}) = - \mathcal{E}_j^{n} + \mathcal{E}_j^{FE}, \ \mathcal{E}_j^{FE} > 0.
\end{equation}
If the interface flux is EC, $\mathcal{E}_j^{n} = 0$ and the fully discrete scheme violates the entropy inequality in every cell. The numerical solution will develop non-physical oscillations growing in time. If the interface flux is ES, then the fully discrete entropy stability depends on the relative magnitudes of $\mathcal{E}^{n}_{j}$ (entropy production in space) and $\mathcal{E}^{FE}$ (entropy loss in time). In the scalar case, time step conditions (similar to CFL conditions) such that the fully discrete scheme is entropy-stable can be derived \cite{Zakerzadeh}. The systems case is more complicated. In practice, a small enough time step is enough to ensure that the entropy losses in time do not overtake the entropy produced in space. \\
\indent Tadmor also showed that the Backward Euler (BE) scheme in time applied to (\ref{eq:FVM})
\begin{equation}\label{eq:FVM_BE}
    \mathbf{u}_j^{n+1} - \mathbf{u}_j^{n} + \lambda (\mathbf{f}_{j+\frac{1}{2}}^{n+1} - \mathbf{f}_{j-\frac{1}{2}}^{n+1}) = 0,
\end{equation}
implies:
\begin{equation}\label{eq:FVM_ES_BE}
    U(\mathbf{u}_j^{n+1}) - U(\mathbf{u}_j^{n}) + \lambda (F_{j+\frac{1}{2}}^{n+1} - F_{j-\frac{1}{2}}^{n+1}) = - \mathcal{E}_j^{n+1} - \mathcal{E}_j^{BE}, \ \mathcal{E}_j^{BE} > 0.
\end{equation}
In this case, the entropy inequality is always satisfied, even if an EC flux is used in space. The expressions for $\mathcal{E}_j^{BE}$ and $\mathcal{E}_j^{FE}$ can be found in \cite{Tadmor_acta} (section 7).

\subsubsection{Positivity}\label{sec:positive}
\indent We showed in sections \ref{sec:EC_flux} and \ref{sec:ES_flux} that despite the fact that the entropy variables are undefined in the limit $\rho_k \rightarrow 0$, the EC flux is well-defined and that the ES flux remains well-defined if the averaged partial densities are properly chosen (we show a similar result for high-order TecNO schemes in section \ref{sec:TecNO}). This does not guarantee that the resulting scheme will not produce negative densities and/or negative pressures. In the one-dimensional test problems, the first order semi-discrete ES scheme fortunately did not produce negative densities or pressure but the high-order TecNO scheme (section \ref{sec:TecNO}) systematically did. In the original setup of the two-dimensional shock-bubble interaction problem \cite{SBI_Marquina, Quirk}, the first-order scheme had the same issue. \\
\indent There are schemes which are conservative, entropy stable and can preserve the correct sign of density and pressure. The Godunov scheme as well as the Lax-Friedrichs scheme qualify if an appropriate CFL condition is met \cite{Tadmor2}. Note also the recent work of Guermond \textit{et al.} \cite{Guermond_IDP, Guermond_IDP2}. A common trait of these schemes is that they take root in the notion of a Riemann problem and the assumption that there exists a solution satisfying all entropy inequalities. These schemes also require an algorithm capable of computing the maximum speed of propagation (or an upper bound) for the Riemann problem at each interface. One such algorithm is discussed in Guermond \& Popov \cite{Guermond_speed}. We could in principle adopt a hybrid approach where these schemes are used in areas where ours fails to maintain positive densities and pressure. Whether this can effectively be accomplished is left for future work. 

\subsubsection{Construction for thermally perfect gases}
Here we briefly discuss the construction of an ES scheme in the case where the specific heats are not constant but functions of temperature. In this configuration, the specific internal energies and entropies are defined as: 
\begin{equation*}
    e_k := e_{0k} + \int_{0}^T c_{vk}(\tau)d\tau, \ s_k := \int_{0}^T \frac{c_{vk}(\tau)}{\tau}d\tau - r_k \ln \rho_k.    
\end{equation*}
The structure that ES schemes build on \cite{Giovangigli, Chalot} is still present. The multicomponent system still admits an additional conservation equation for the thermodynamic entropy of the mixture $\rho s$ and $(U, F) = (-\rho s, \ -\rho u s)$ is a valid entropy pair for $\rho_k > 0, \ T > 0$. In practice, the specific heats are represented using polynomial interpolation:
\begin{equation*}
    c_{vk} := c_{k0} + \sum_{j=1}^P c_{kj}T^{j},
\end{equation*}
where $c_{kj}, \ 0 \leq j \leq P$ are constants. This gives:
\begin{equation*}
    e_k = e_{0k} + c_{k0} T + \sum_{j=1}^p \frac{c_{kj}}{j+1}T^{j+1}, \ s_k = c_{k0}\ln T + \sum_{j=1}^P \frac{c_{kj}}{j} T^{j} - r_k \ln \rho_k.
\end{equation*}
The expressions of the entropy variables and potential functions remain unchanged. Regarding the construction of the EC flux, we have for $1 \leq k \leq N$:
\begin{equation}\label{eq:dvk}
    [v_{1,k}] = \bigg[ \frac{1}{T}\bigg(g_k - \frac{u^2}{2}\bigg)\bigg] = \bigg[ \frac{e_k}{T} - s_k \bigg] - \bigg[ \frac{u^2}{2T}\bigg] = e_{0k} \bigg[\frac{1}{T} \bigg] - c_{k0} [\ln T] - \sum_{j=1}^P \frac{c_{kj}}{j (j+1)}[T^{j}] + r_k [\ln \rho_k] - \frac{\overline{u^2}}{2}\bigg[\frac{1}{T}\bigg] - \overline{\frac{1}{T}}\overline{u}[u]
\end{equation}
For a given quantity $a$, let's define the product operator $a^{\times} = a_L a_R$.  For $a \neq 0$, we have the following identity:
\begin{equation*}
    [a] = - a^{\times} \ \bigg[\frac{1}{a}\bigg]
\end{equation*}
It can be easily showed, by induction for instance, that for $j \geq 1$, there exists an averaging operator $f_j(a)$, consistent with $a^{j-1}$, such that $[a^{j}] = j  f_j(a) [a]$ ($f_1(a) = 1$, $f_2(a) = \overline{a}$, $f_3(a) = (2/3)\overline{a} \ \overline{a} + (1/3)\overline{a^2}$ and so on using product rules). Equation (\ref{eq:dvk}) can therefore be rewritten as:
\begin{align*}
    [v_{1,k}] =& \ e_{0k} \bigg[\frac{1}{T} \bigg] - c_{k0} [\ln T] + \sum_{j=1}^P \frac{c_{kj}}{(j+1)} (f_j(T) T^{\times}) \bigg[ \frac{1}{T} \bigg] + r_k [\ln \rho_k]  - \frac{\overline{u^2}}{2}\bigg[\frac{1}{T}\bigg] - \overline{\frac{1}{T}}\overline{u}[u] \\
              =& \ \bigg[ \frac{1}{T} \bigg] \bigg( 
              e_{0k} + c_{k0} \frac{1}{(1/T)^{ln}} + \sum_{j=1}^P \frac{c_{kj}}{(j+1)} (f_j(T) T^{\times})  - \frac{\overline{u^2}}{2} \bigg) + [\rho_k] \frac{r_k}{\rho_k^{ln}} - [u] \overline{\bigg(\frac{1}{T}\bigg)}\overline{u}
\end{align*}
Repeating the procedure outlined in section \ref{sec:EC_flux}, we obtain an EC flux:
\begin{align}\label{eq:ECfluxMC_TPG}
    f_{1,k} =& \  \rho_{k}^{ln}\overline{u}, \nonumber \\
    f_2     =& \ \frac{1}{\overline{1/T}} \bigg(\sum_{k=1}^Nr_k \overline{\rho_k}\bigg) + \overline{u} \sum_{k=1}^Nf_{1,k}, \\
    f_3     =& \ \sum_{k=1}^N \bigg(e_{0k} + c_{k0}\frac{1}{(1/T)^{ln}} + \sum_{j=1}^P \frac{c_{kj}}{(j+1)} (f_j(T) T^{\times}) - \frac{1}{2}\overline{u^2} \bigg) f_{1,k} + \overline{u} f_2. \nonumber
\end{align}
which is consistent ($f_j(T)T^{\times}$ is consistent with $T^{j+1}$) and differs from the expression we obtained in the calorically perfect gas case (equation (\ref{eq:ECfluxMC})) in the total energy component only. \\
\indent We will not go into the details of the upwind dissipation operator. Barth's eigenscaling theorem applies because $H$ still symmetrizes $A$ from the right. Therefore  the scaling matrix exists ($T^{2} = R^{-1}HR^{-T}$) and the dissipation operator constructed in section \ref{sec:ES_flux} can be constructed for thermally perfect gases as well.
\subsubsection{Total mass form}
\indent Instead of solving the conservation equations for the N partial densities $\rho_k$, one might want to solve for the conservation of the total density $\rho$ and N-1 partial densities. The state and flux vectors are then:
\begin{equation*}
    \mathbf{\tilde{u}} := \begin{bmatrix} \rho & \rho_2  & \hdots & \rho_N & \rho u & \rho e^t \end{bmatrix}^T, \
    \mathbf{\tilde{f}} := \begin{bmatrix} \rho u & \rho_2 u & \hdots & \rho_N u & \rho u^2 + p & (\rho e^t + p)u \end{bmatrix}^T.
\end{equation*}
The entropy variables in this configuration can be easily obtained using a chain rule:
\begin{equation*}
    \mathbf{\tilde{v}}^T = -\frac{\partial \rho s}{\partial \mathbf{\tilde{u}}} = - \frac{\partial \rho s}{\partial \mathbf{u}} \bigg( \frac{\partial \mathbf{\tilde{u}}}{\partial \mathbf{u}} \bigg)^{-1} = \mathbf{v}^T \bigg( \frac{\partial \mathbf{\tilde{u}}}{\partial \mathbf{u}} \bigg)^{-1} = \mathbf{v}^T \bigg( \frac{\partial \mathbf{u}}{\partial \mathbf{\tilde{u}}} \bigg)
\end{equation*}
$\mathbf{\tilde{u}}$ and $\mathbf{u}$ only differ in the first component, $ \rho_1 = \rho - \sum_{k=2}^{N} \rho_k$ therefore:
\begin{equation*}
    \frac{\partial \mathbf{u}}{\partial \mathbf{\tilde{u}}} =  \begin{bmatrix}
        1 &   -1    & \hdots & -1 & 0 & 0 \\
          &    1    &        &  0 & 0 & 0 \\
          &         & \ddots &  0 & 0 & 0 \\
        0 &    0    & \hdots &  1 & 0 & 0 \\
        0 &    0    & \hdots &  0 & 1 & 0 \\
        0 &    0    & \hdots &  0 & 0 & 1
    \end{bmatrix}, \
    \mathbf{\tilde{v}} =
    \frac{1}{T} \begin{bmatrix}
    (g_1 - \frac{1}{2}u^2) & (g_2 - g_1) & \hdots & (g_N - g_1) & u & -1
    \end{bmatrix}^T.
\end{equation*}
The corresponding potential functions are unchanged because:
\begin{equation*}
    \mathbf{\tilde{v}} \cdot \mathbf{\tilde{f}} = \mathbf{v}^T \bigg( \frac{\partial \mathbf{u}}{\partial \mathbf{\tilde{u}}} \bigg) \bigg(  \frac{\partial \mathbf{\tilde{u}}}{\partial \mathbf{u}} \bigg) \mathbf{f} = \mathbf{v} \cdot \mathbf{f}, \ \mathbf{\tilde{v}} \cdot \mathbf{\tilde{u}} = \mathbf{v}^T \bigg( \frac{\partial \mathbf{u}}{\partial \mathbf{\tilde{u}}} \bigg) \bigg(  \frac{\partial \mathbf{\tilde{u}}}{\partial \mathbf{u}} \bigg) \mathbf{u} = \mathbf{v} \cdot \mathbf{u}
\end{equation*}
Accordingly, an EC flux $\mathbf{\tilde{f}_{EC}}$ for the total mass form can simply be obtained by mapping an EC flux in the first form $\mathbf{f_{EC}}$:
\begin{equation*}
    \mathbf{\tilde{f}_{EC}} = \bigg(\frac{\partial \mathbf{\tilde{u}}}{\partial \mathbf{u}}\bigg) \mathbf{f}_{EC}.
\end{equation*}
Likewise, applying the same mapping to an ES flux $\mathbf{f^*}$ in the first form given by:
\begin{equation*}
    \mathbf{f^*} = \mathbf{f_{EC}} - D [\mathbf{v}]
\end{equation*}
where $D$ is positive definite, results in a flux $\mathbf{\tilde{f}^*}$ given by:
\begin{equation*}
   \mathbf{\tilde{f}^*} = \bigg(\frac{\partial \mathbf{\tilde{u}}}{\partial \mathbf{u}}\bigg) \mathbf{f^*} = \mathbf{\tilde{f}_{EC}} - \bigg(\frac{\partial \mathbf{\tilde{u}}}{\partial \mathbf{u}}\bigg) D [\mathbf{v}] = \mathbf{\tilde{f}_{EC}} - \bigg(\frac{\partial \mathbf{\tilde{u}}}{\partial \mathbf{u}}\bigg) D \bigg(\frac{\partial \mathbf{\tilde{u}}}{\partial \mathbf{u}}\bigg)^T  [\mathbf{\tilde{v}}] = \mathbf{\tilde{f}_{EC}} - \tilde{D}  [\mathbf{\tilde{v}}].
\end{equation*}
$\tilde{D}$ is positive definite by congruence therefore $\mathbf{\tilde{f}^*}$ is entropy stable. 

\section{High-order discretizations}
\indent In this section, we are essentially interested in how the fundamental issues highlighted in the previous sections manifest in a high-order setting. We examine two high-order ES formulations: TecNO schemes \cite{Fjordholm} and Discontinuous Galerkin \cite{Reed, Cockburn} (DG) schemes discretizing the entropy variables \cite{Barth, Hughes}. High-order ES schemes are not limited to these two options (formulations based on Summation-By-Parts operators \cite{Fisher, Fried} for instance are actively being developed), but the issues their formulation raises are no different. 
\subsection{Discontinuous Galerkin}
\indent In \cite{Hughes}, Hughes \textit{et al.} showed, under the assumption of exact numerical quadrature, that continuous finite element solutions to the compressible Navier-Stokes equations become consistent with the entropy equation when the entropy variables are discretized instead of conservative variables. Furthermore, they showed that with suitably defined dissipation operators, the discrete solution satisfies the Clausius-Duhem inequality. \\
\indent For Discontinuous Galerkin discretizations of the Compressible Euler Equations, the same result follows if ES fluxes are used. Yet, the formulation of ES DG schemes by Barth \cite{Barth} does not proceed along these lines. The entropy variables are discretized but entropy-stable fluxes are defined in a different way. This is discussed in \ref{appendix:Barth}. \\
\indent The fact that the entropy variables are undefined in the limit $\rho_k \rightarrow 0$ poses a daunting problem, unless the flow configuration is such that one can expect $\rho_k > 0$ at all times. A way around this issue (other than not discretizing the entropy variables) has not been found by the authors. There might be other entropy functions for which the corresponding entropy variables are well-defined in this limit. However, in the context of the multicomponent Navier-Stokes equations (including viscous stresses, heat conduction, multicomponent diffusion), the \textit{a priori} entropy stability resulting from discretizing the entropy variables might be lost \cite{Hughes, Giovangigli_Mat}. 
\subsection{TecNO schemes}\label{sec:TecNO}
\indent TecNO schemes (Fjordholm \textit{et al.} \cite{Fjordholm}) are high-order entropy-stable finite volume schemes that combine the high-order EC flux formulation of LeFloch \textit{et al.} \cite{LeFloch}, the stencil selection procedure of ENO schemes \cite{HartenENO, ShuOsher} and entropy-stable dissipation operators \cite{Tadmor, Roe1}. \\
%\indent The EC flux $\mathbf{f^*}$ is replaced with a high-order EC flux $\mathbf{f^*_{2p}}$ defined over a centered stencil of $2 p$ points $(v_{j-p+1}, \ \dots \ , v_{j+p})$ by:
%\begin{equation*}
%    \mathbf{f_{2p}^*} (v_{j-p+1}, \ \dots \ , v_{j+p}) = \sum_{i=1}^p \alpha_{i,p} \sum_{s=0}^{i-1} \mathbf{f^*}(v_{j-s},v_{j-s+i}).
%\end{equation*}
%The coefficients $\alpha_{i,p}$ need to satisfy %\cite{LeFloch}:
%\begin{equation*}
 %   \sum_{i=1}^{p} i \alpha_{i,p} = 1, \  \sum_{i=1}^{p} i^{2s-1}\alpha_{i,p} = 0, \ s = 2 \dots p,
%\end{equation*}
%The first equation is for consistency, the second is for 2p-th order accuracy. For $p = 2$ (4-th order) and $ p = 3$ (6-th order) the coefficients are:
%\begin{gather*}
 %   \alpha_{1,2} = \frac{4}{3}, \ \alpha_{2,2} = -\frac{1}{6} \\
 %   \alpha_{1,3} = \frac{3}{2}, \ \alpha_{2,3} = -\frac{3}{10}, \ \alpha_{3,3} = \frac{1}{30}
%\end{gather*}
\indent The dissipation term of a first order ES flux typically takes the form:
\begin{equation*}
    D [\mathbf{v}] = (R T) |\Lambda| (R T)^T [\mathbf{v}],
\end{equation*}
where $R$ is the matrix of right eigenvectors of $A$ and $T$ is a scaling matrix. The reconstruction used by TecNO schemes is motivated by the sign property of the ENO reconstruction. It was shown \cite{Fjordholm2} that for any vector $\mathbf{w} \in \mathbb{R}^{N}$, the ENO reconstruction is such that:
\begin{equation*}
    \left<\mathbf{w}\right> = B [\mathbf{w}], \ B = diag([b_0, \ \dots, \ b_{N-1}]), \ b_i \geq 0,
\end{equation*}
where $\left< \mathbf{w} \right>$ and $[\mathbf{w}]$ are the reconstructed and initial jumps, respectively. Let $\mathbf{w}$ be such that $[\mathbf{w}] = (RT)^T [\mathbf{v}]$, then the dissipation operator $\Tilde{D} [\mathbf{w}] = (RT) |\Lambda| \left< \mathbf{w} \right>$ is ES because:
\begin{equation*}
    (RT) |\Lambda| \left< \mathbf{w} \right> = (RT) |\Lambda| B [\mathbf{w}] = (RT) (|\Lambda| B) (RT)^T [\mathbf{v}]
\end{equation*}
$|\Lambda| B$ is a positive diagonal matrix, therefore the high-order dissipation operator $\Tilde{D} [ \mathbf{w} ]$ is ES. The TecNO approach does require the knowledge of the scaled eigenvectors $(RT)$. If the dissipation operator is only known in the form:
\begin{equation*}
    D[\mathbf{v}] = R |\Lambda| T^2 R^{T}[\mathbf{v}],
\end{equation*}
then choosing $\mathbf{w}$ such that $[\mathbf{w}] = (T^2 R^T)[\mathbf{v}]$ for the ENO reconstruction will result in a high-order dissipation operator:
\begin{equation*}
    \tilde{D} [ \mathbf{w} ] =  R |\Lambda| \left< \mathbf{w} \right> = R (|\Lambda| B T^2) R^T [\mathbf{v}]. 
\end{equation*}
$|\Lambda|$, $B$ and $T^2$ are all symmetric and at least positive semi-definite, however their product is no necessarily positive semi-definite. If the squared scaling matrix is diagonal then entropy stability is preserved. In the general case of a block-diagonal scaling matrix, entropy stability is no longer ensured. The same problem would arise if $\mathbf{w}$ was chosen such that $[\mathbf{w}] = R^{T}[\mathbf{v}]$ (the question would be whether $(|\Lambda| T^2 B)$ is positive definite). \\
\indent If the dissipation operator is expressed according to equation (\ref{eq:ESflux}) with the pseudo-scaling matrix $\mathcal{T}$ given by equation (\ref{eq:fake_tilde}), the vector of reconstructed variables $\mathbf{w}$ such that $ = [\mathbf{w}] = (R\mathcal{T})^T [\mathbf{v}]$ will have as many components as the number of rows of $\mathcal{T}$. For $N > 2$, the number of rows of $\mathcal{T}$ grows as $\mathcal{O}(N^2)$, so for a large enough $N$, one might prefer working with $T$ instead of $\mathcal{T}$ and avoid having to reconstruct too many variables. \\
\indent In section \ref{sec:ES_flux}, we showed that the dissipation term expressed as $R |\Lambda| T^2 R^{T}$ is well defined in the limit $\rho_k \rightarrow 0$ provided that $\rho_k^* = \rho_k^{ln}$. This was possible because one could extract from $T^2$ a diagonal matrix $\mathcal{R}$ (see equation (\ref{eq:R})) of partial densities. The TecNO algorithm requires the isolated evaluation of the scaled entropy variables defined by the jump relation $[\mathbf{w}] = (R \mathcal{T})^{T}[\mathbf{v}]$ or $[\mathbf{w}] = (R T)^{T}[\mathbf{v}]$. The matrix $\mathcal{R}$ that was extracted from the squared scaling matrix, might not be extracted from $T$ or $\mathcal{T}$ without leaving $1/\rho_k^*$ terms behind. However, $\mathcal{R}^{1/2}$ can be extracted from the pseudo-scaling matrix $\mathcal{T}$ and since:
\begin{equation*}
    [\ln(\rho_k)] = [2 \ \ln(\sqrt{\rho_k})] = 2 \ \frac{[\sqrt{\rho_k}]}{(\sqrt{\rho_k})^{ln}},
\end{equation*}
it can be shown that $[\mathbf{w}] = (R \mathcal{T})^T[\mathbf{v}]$ is well-defined in the limit $\rho_k \rightarrow 0$, provided that $\rho_k^* = ((\sqrt{\rho_k})^{ln})^2$. The decomposition:
\begin{equation}\label{eq:R_half}
    T^2 = \mathcal{R}^{1/2} \widetilde{T}^2 (\mathcal{R}^{1/2})^T, \ \widetilde{T}^2 = \frac{1}{\gamma r}\begin{bmatrix}
            (\gamma-1) Y_1 + \gamma r_2/r_1 Y_2  & -\sqrt{Y_1 Y_2}  & 0 & 0 \\
            -\sqrt{Y_1 Y_2} & (\gamma-1) Y_2 + \gamma r_1/r_2 Y_1 & 0 & 0 \\
            0 & 0 & \frac{1}{2} & 0 \\
            0 & 0 & 0 & \frac{1}{2}
          \end{bmatrix},
\end{equation}
suggests that a TecNO reconstruction based on the scaling matrix $T$ would still be defined in the limit with the same averaging. Note that this average is not compatible with the stationary contact preservation condition (\ref{eq:StatContact}) we derived in section \ref{sec:ES_flux}. That is because equation (\ref{eq:StatContact}) requires $\rho_k^{*} = \rho_k^{ln}$. 
\section{Numerical experiments}\label{sec:cases}
\indent In this section, we present and discuss numerical results on 1D and 2D test problems involving interfaces and shocks. A 3D formulation of the scheme (entropy variables, EC flux, ES flux) is provided in \ref{appendix:3D}. \\
\indent In the 1D problems, the first-order finite volume scheme with the ES flux in space and forward Euler in time with a CFL of 0.3 on three grids (100, 300, and 1000 cells, respectively) was applied. All the figures in the next three subsections use the same legend as figure \ref{fig:1D_moving_interface_velocity}-(a). A fourth-order TecNO scheme is used in the 2D problem (section \ref{sec:SBI}).
\subsection{Moving interface}
%\indent \textcolor{blue}{You need to rewrite this section, because the ES people and other people not familiar with this problem are stupid and can't think for more than 5 minutes. 
%\begin{itemize}
    %\item Explain how this is a problem specific to multicomponent flows. Show a moving contact for single Euler.
    %\item Explain that production of entropy is the problem. It's a very different problem than shocks.
    %\item Explain that these anomalies are produced by a first-order scheme. These are not the oscillations of a high-order, unstable scheme. 
    %\item Explain how this is not about violating the minimum entropy principle.
    %\item Heavily encourage these morons to actually read.
%\end{itemize}
%}
\indent The first test problem is the advection of a contact discontinuity (constant velocity and constant pressure) separating two different species. The initial conditions are given by:
\begin{align*}
\begin{cases}
    (\rho_1, \ \rho_2, \ u, \ p) =& (0.1, \ 0.0, \ 1.0, \ 1.0), \ 0 \leq x \leq 0.5, \\
    (\rho_1, \ \rho_2, \ u, \ p) =& (0.0, \ 1.0, \ 1.0, \ 1.0), \ 0.5 < x \leq 1 .
\end{cases}
\end{align*}
with $\gamma_1 = 1.4, \ \gamma_2 = 1.6$, $c_{v1} = c_{v2} = 1$. The velocity and pressure profiles at $t = 0.022 s$ and $t = 0.1s$ are shown in figures \ref{fig:1D_moving_interface_velocity} and \ref{fig:1D_moving_interface_pressure}. These profiles show overshoots and undershoots that are typically observed with conservative schemes (see figure 4 in Abgrall \& Karni \cite{Abgrall} for instance). These anomalies are often termed oscillations in the literature \cite{Abgrall, Karni, Abgrall1, SBI_Marquina, SBI_Johnsen}. Upon closer examination, we can see three wave packets propagating at different speeds. The one moving to the left has the fastest propagation speed, roughly $-3$. The remaining two are moving to the right with propagation speeds of roughly $1$ (the speed of the contact) and $2$. From figure \ref{fig:1D_moving_interface_acoustic}, which shows the acoustic eigenvalues $u \pm a$, it clearly appears that that the first and third wave packets are acoustic. \\
\indent Figure \ref{fig:1D_moving_interface_entropy} shows the total entropy $\rho s = \rho_1 s_1 + \rho_2 s_2$ profiles at $t = 0.02 s$ and $t = 0.1 s$. The wave structure of the pressure and velocity anomalies is more apparent, and we can see that each wave is carrying an spurious increase in entropy. Figure \ref{fig:1D_interface_total_entropy} shows the evolution of the total entropy over time (the contributions from the boundaries were removed). This shows that the anomalies observed do not violate entropy stability. On the contrary, it appears that inappropriate production of entropy is the issue. Additionally, we find (\ref{appendix:interface_BE}) that these anomalies are still present when the upwind dissipation operator is discarded and the Backward Euler scheme is used to ensure entropy stability (recall the discussion in section \ref{sec:TE} and equation (\ref{eq:FVM_ES_BE})). In the same vein, we find that these anomalies do not violate a minimum principle of the specific entropy \cite{Gouasmi4} (see figure \ref{fig:1D_moving_interface_spec_entropy}). \\
\indent That these anomalies are not linked to entropy stability or the minimum entropy principle does not come as a surprise considering that they were already observed with the Godunov scheme \cite{Abgrall, Karni}, which is both entropy stable and satisfies a minimum entropy principle \cite{Tadmor2}. \\
\indent It is important to understand that these anomalies are intrinsic to multicomponent flows. There are no such anomalies on moving contacts in the compressible Euler equations (\ref{appendix:interface_Euler}). Furthermore, these anomalies are produced by first-order schemes. They are therefore different than the oscillations typically observed with high-order schemes when discontinuous solutions are sought. Intrigued readers are strongly encouraged to read early studies of this problem \cite{Karni, Abgrall, Abgrall}. \\
\indent For this problem and the following, high-order TecNO schemes were found to produce non-physical values of density and pressure. 
\begin{figure}[h!]
    \centering
    \subfigure[t = 0.022 s]{\includegraphics[scale = 0.75]{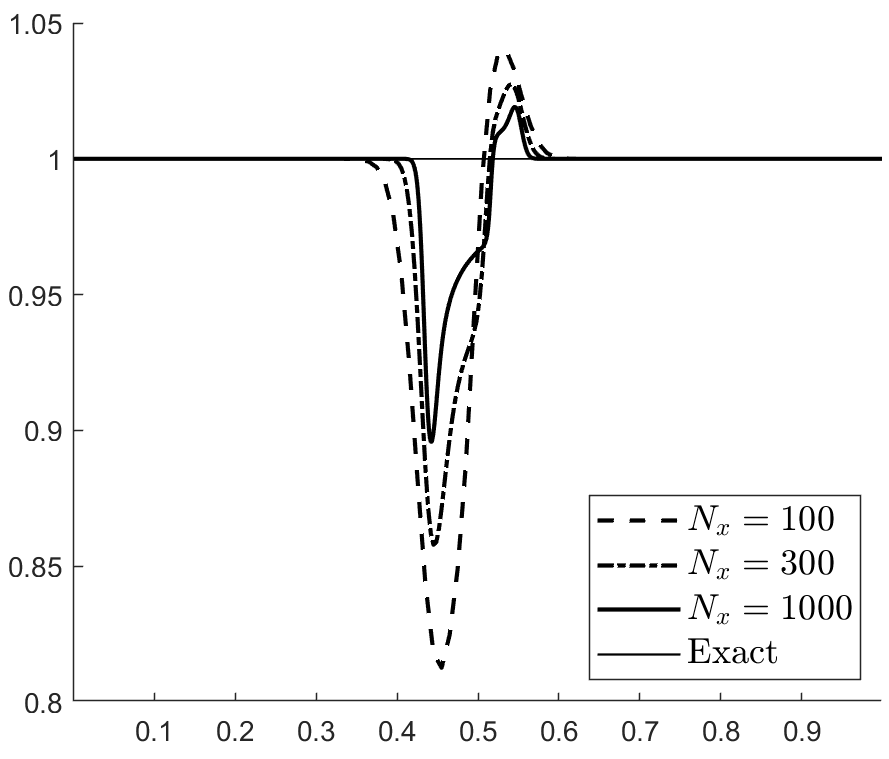}}
    %\subfigure[]{\includegraphics[scale = 0.5]{Pics/1D/MovingInterface/mv_int_velocity_2.png}}
    \subfigure[t = 0.1 s]{\includegraphics[scale = 0.75]{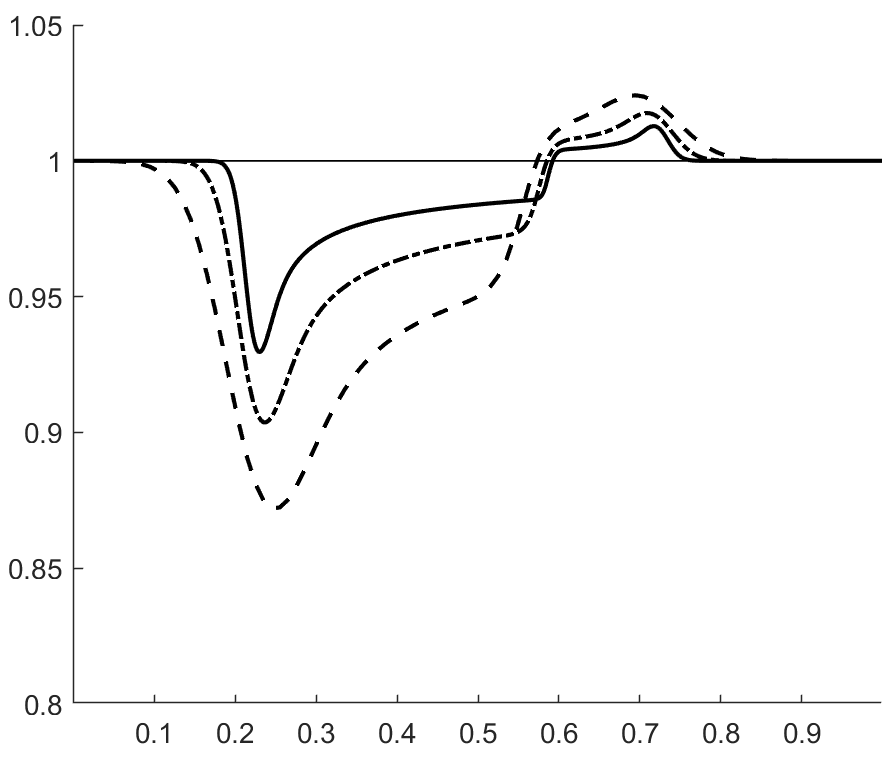}}
    \caption{Velocity profiles for the moving interface problem.}
    \label{fig:1D_moving_interface_velocity}
\end{figure}

\begin{figure}[h!]
    \centering
    \subfigure[t = 0.022 s]{\includegraphics[scale = 0.75]{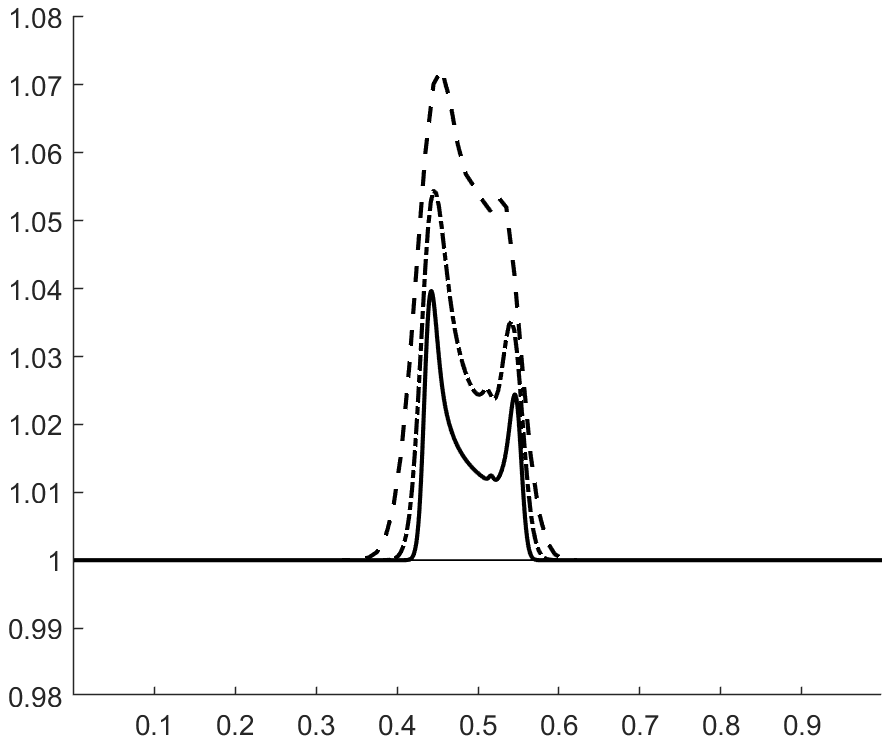}}
    \subfigure[t = 0.1 s]{\includegraphics[scale = 0.75]{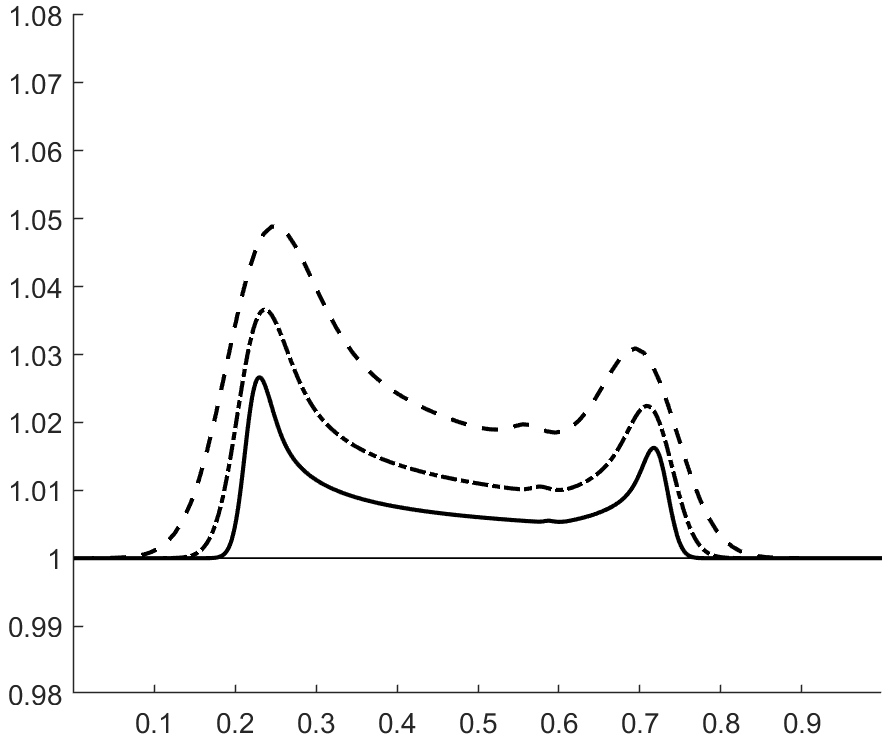}}
    \caption{Pressure profiles for the moving interface problem.}
     \label{fig:1D_moving_interface_pressure}
\end{figure}

\begin{figure}[h!]
    \centering
    \subfigure[$\lambda = u - a$]{\includegraphics[scale = 0.75]{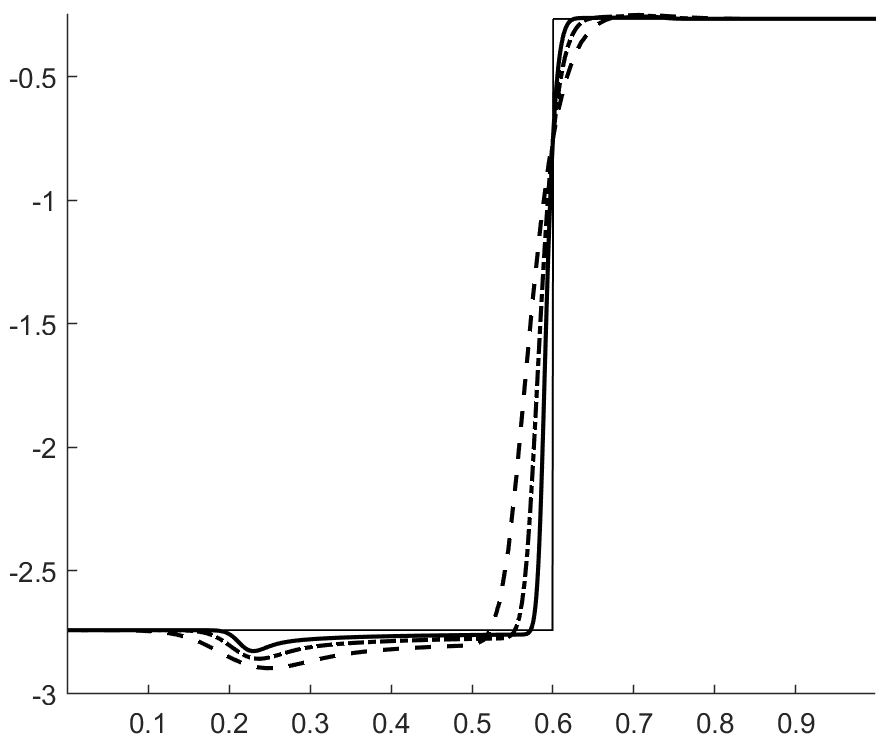}}
    \subfigure[$\lambda = u + a$]{\includegraphics[scale = 0.75]{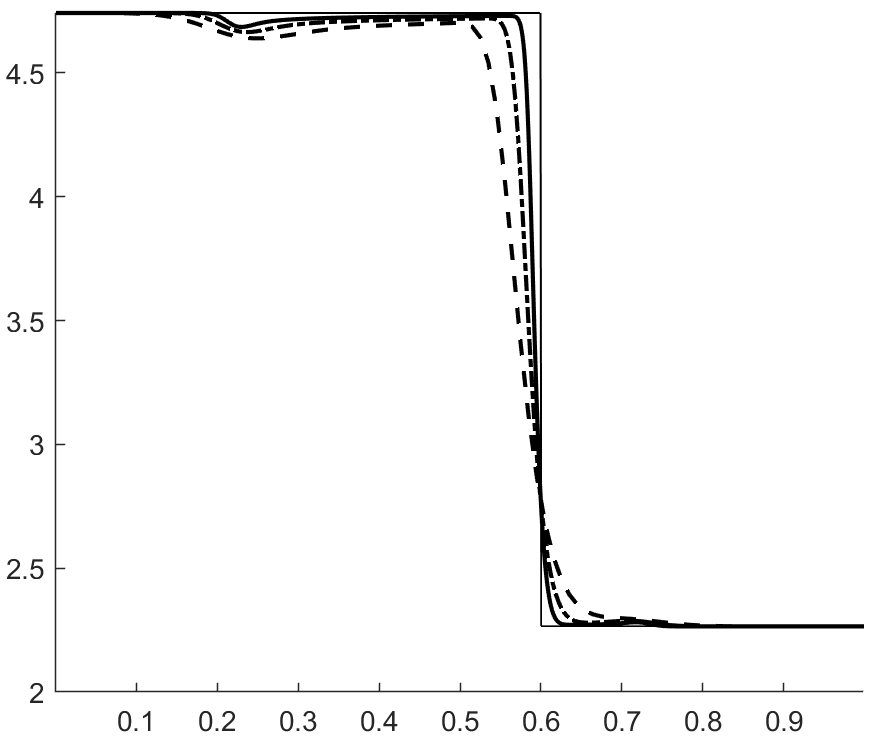}}
    \caption{Acoustic eigenvalues for the moving interface problem at $t = 0.1$ s.}
     \label{fig:1D_moving_interface_acoustic}
\end{figure}

\begin{figure}[h!]
    \centering
    \subfigure[t = 0.022 s]{\includegraphics[scale = 0.75]{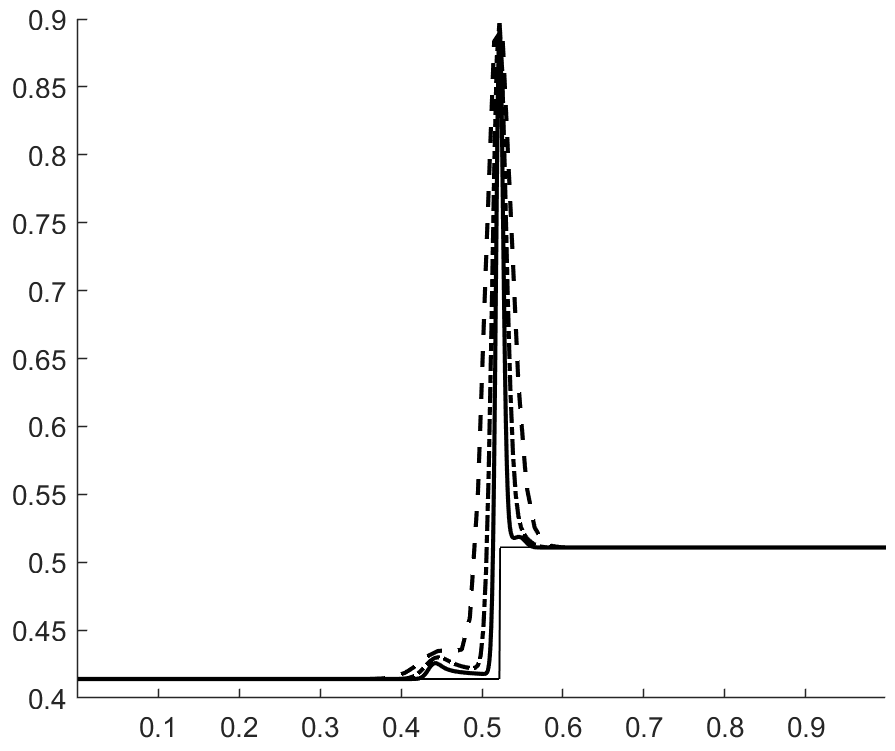}}
    %\subfigure[]{\includegraphics[scale = 0.5]{Pics/1D/MovingInterface/mv_int_pressure_2.png}}
    \subfigure[t = 0.1 s]{\includegraphics[scale = 0.75]{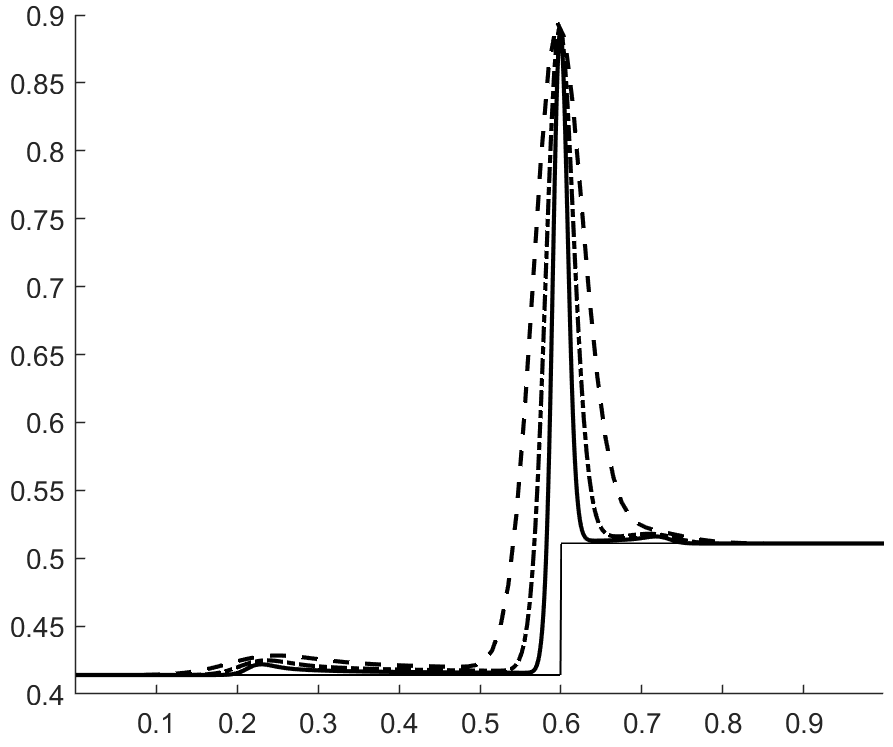}}
    \caption{Entropy ($\rho s = \rho_1 s_1 + \rho_2 s_2$) profiles for the moving interface problem.}
    \label{fig:1D_moving_interface_entropy}
\end{figure}

\begin{figure}[h!]
    \centering
    \subfigure[t = 0.022 s]{\includegraphics[scale = 0.75]{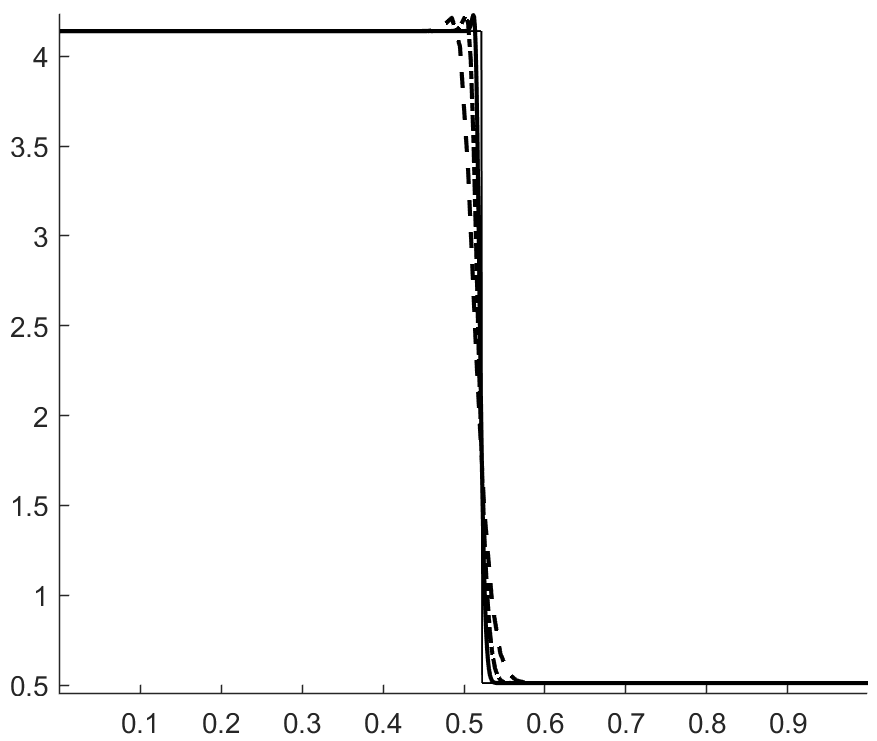}}
    %\subfigure[]{\includegraphics[scale = 0.5]{Pics/1D/MovingInterface/mv_int_pressure_2.png}}
    \subfigure[t = 0.1 s]{\includegraphics[scale = 0.75]{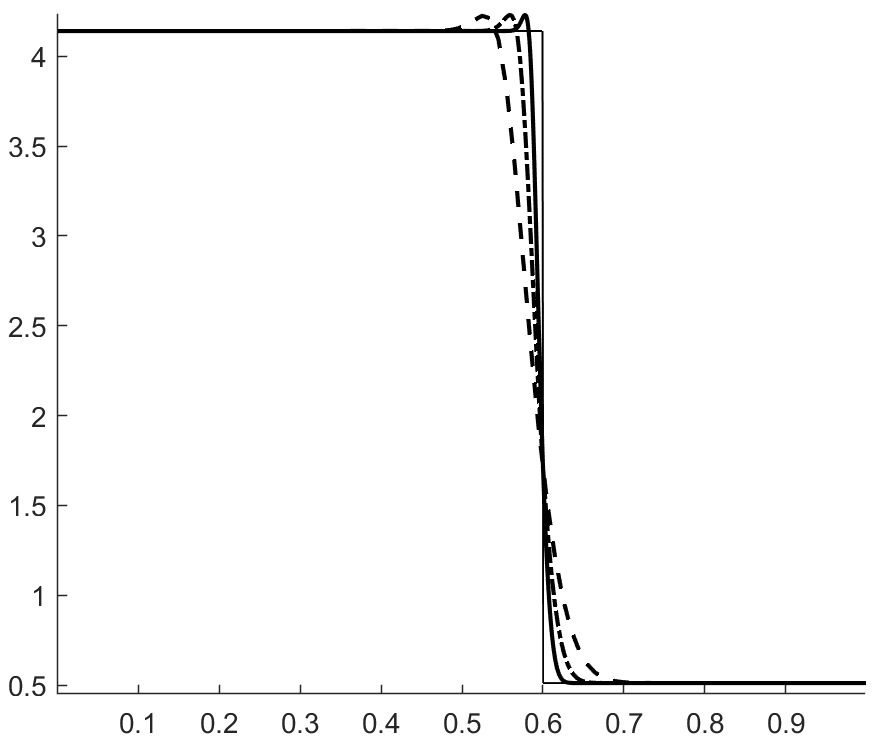}}
    \label{fig:1D_moving_interface_spec_entropy}
    \caption{Specific entropy ($s = Y_1 s_1 + Y_2 s_2$) profiles for the moving interface problem.}
\end{figure}

\begin{figure}[h!]
    \centering
    \includegraphics[scale = 0.65]{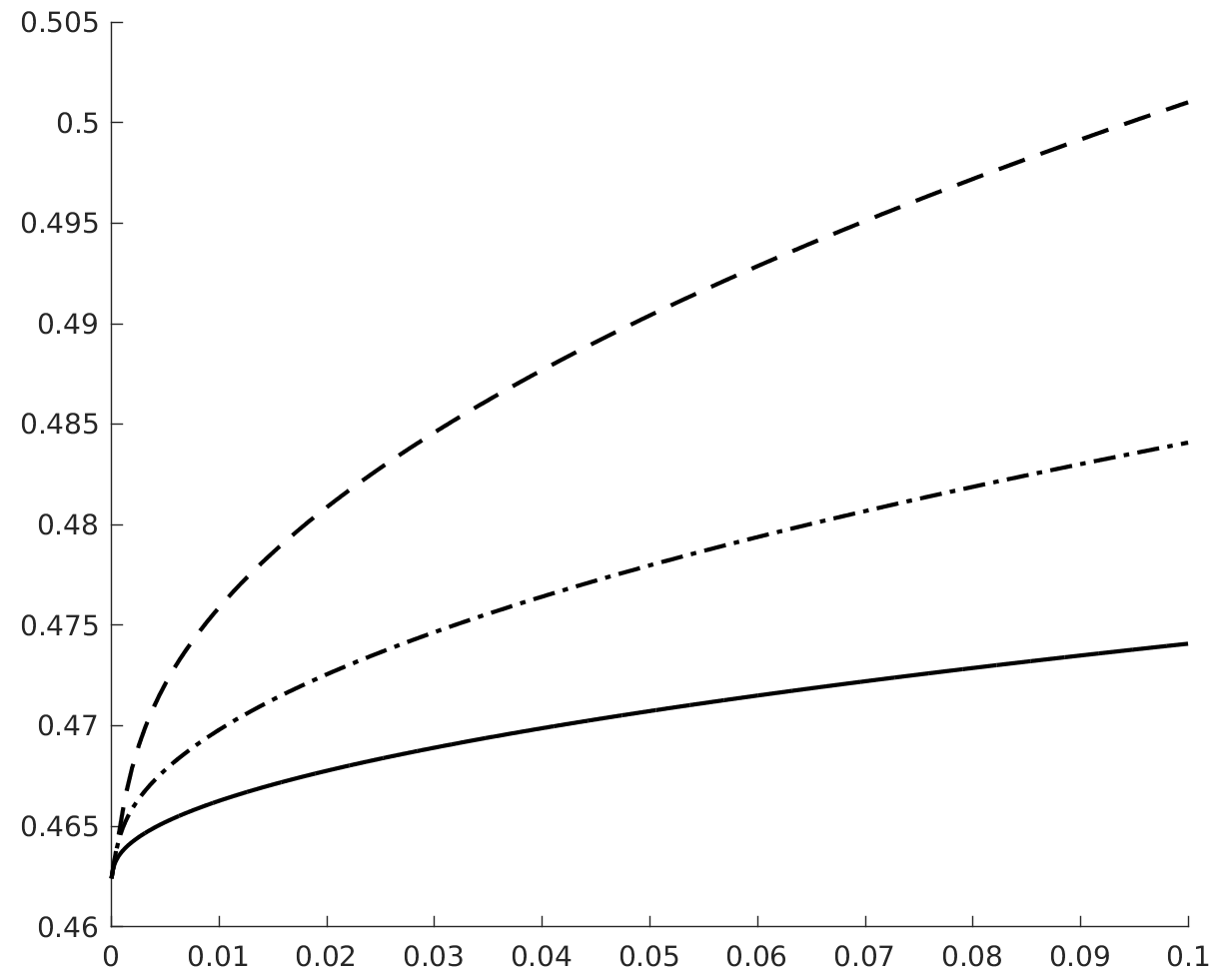}
    \caption{Total entropy over time for the moving interface problem.}
    \label{fig:1D_interface_total_entropy}
\end{figure}

\subsection{Two-species shock-tube problem}
\indent We simulate a shock-tube problem with two species. The initial conditions are given by:
\begin{align*}
\begin{cases}
    (\rho_1, \ \rho_2, \ u, \ p) =& (1, \ 0, \ 0, \ 1), \ 0 \leq x \leq 0.5, \\
    (\rho_1, \ \rho_2, \ u, \ p) =& (0, \ 0.125, \ 0, \ 0.1), \ 0.5 < x \leq 1.0,
\end{cases}
\end{align*}
with $\gamma_1 = 1.4, \ \gamma_2 = 1.6$ and $c_{v1} = c_{v2} = 1$. The exact solution to the multicomponent shock-tube problem is almost the same as the exact solution in the single-component case. The composition of the gas changes across the contact discontinuity but does not change across rarefaction waves and shock waves \cite{Fezoui}. \\
\indent Figures \ref{fig:1D_sod}(a)-(d) show the velocity, pressure, total density and specific heat ratio profiles at $t = 0.2s$. There is a good agreement with the exact solution, and we do not observe pressure and velocity oscillations around the moving contact which separates the two species. 

\begin{figure}[h!]
    \centering
    \subfigure[velocity]{\includegraphics[scale = 0.7]{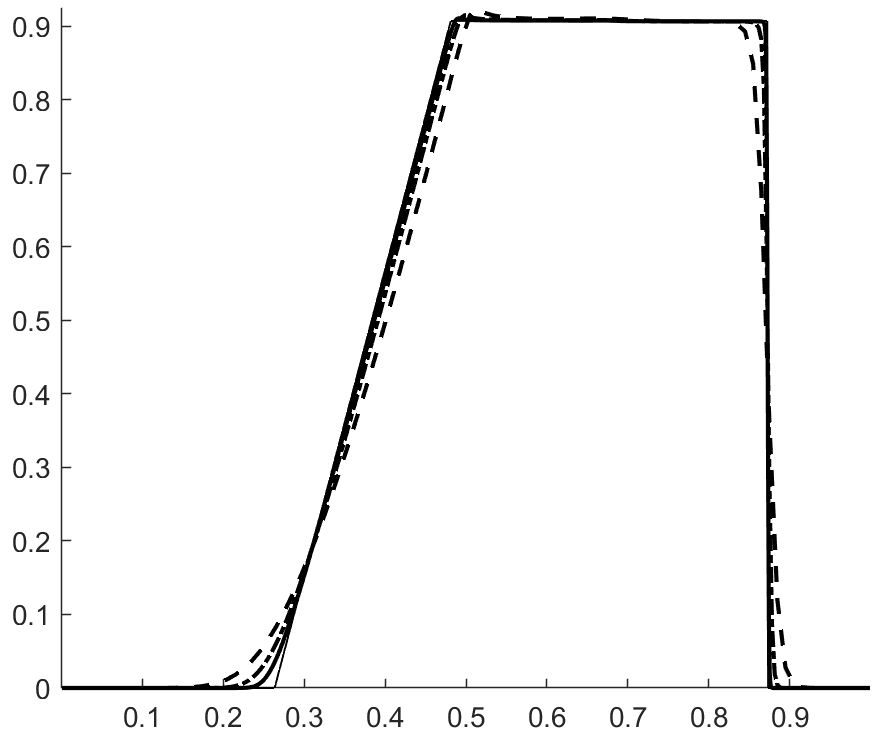}}
    %\subfigure[]{\includegraphics[scale = 0.5]{Pics/1D/MovingInterface/mv_int_pressure_2.png}}
    \subfigure[pressure]{\includegraphics[scale = 0.7]{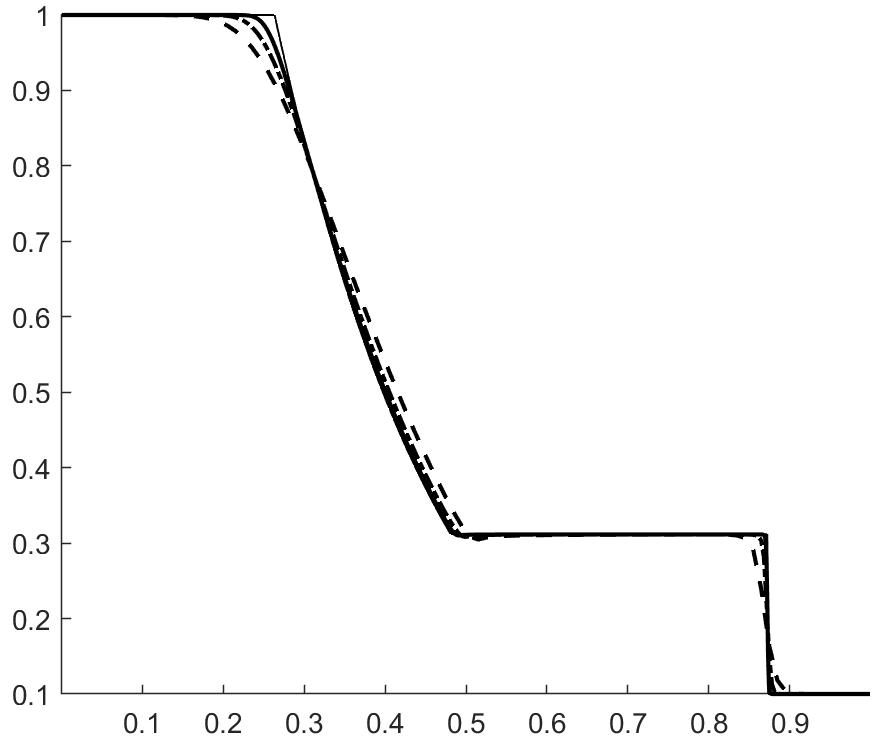}}
    \subfigure[density]{\includegraphics[scale = 0.7]{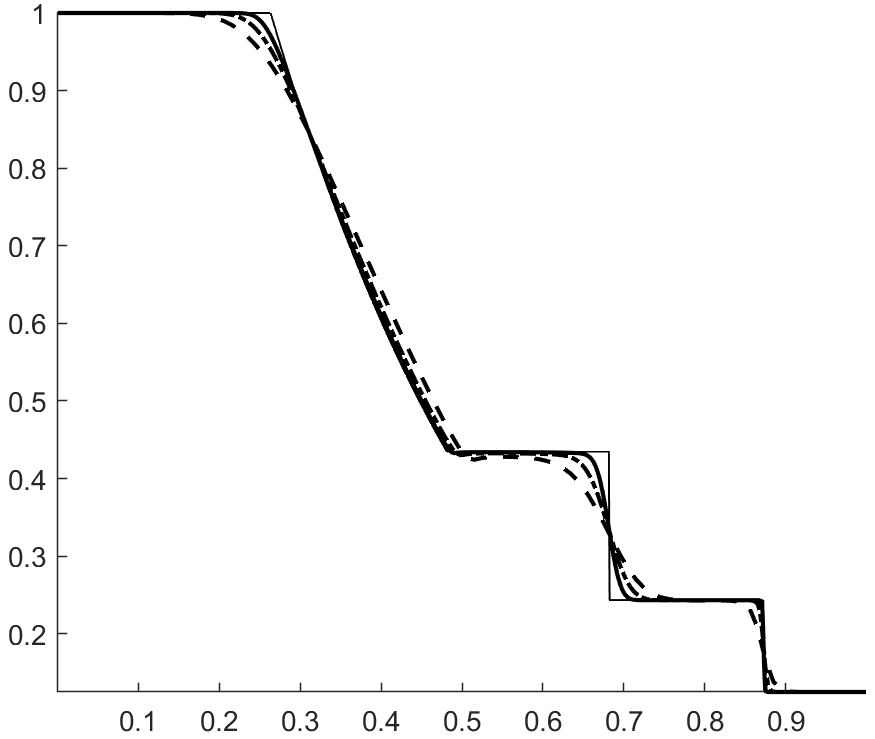}}
    \subfigure[specific heat ratio]{\includegraphics[scale = 0.7]{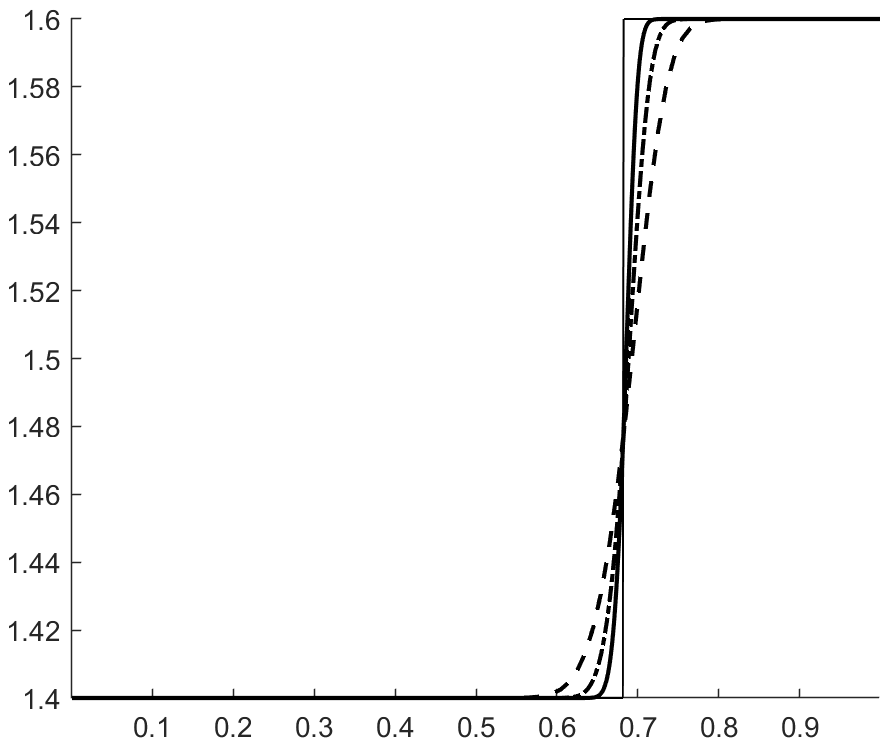}}
    \caption{Solution at t = 0.2 s for the two-species shock tube problem.}
    \label{fig:1D_sod}
\end{figure}

\subsection{Shock-interface interaction}
\indent We simulate a test problem from Quirk \& Karni \cite{Karni} which consists of a shock tube filled with air, where a shock wave moves to the right and eventually meets a stationary bubble of helium at pressure equilibrium. The initial conditions are given by:
\begin{align*}
\begin{cases}
    (\rho_1, \ \rho_2, \ u, \ p) =& \ (1.3765, \ 0, \ 0.3948, \ 1.57), \ 0 \leq x \leq 0.25, \ \mbox{Post-shock, air}, \\
    (\rho_1, \ \rho_2, \ u, \ p) =& \ (1., \ 0, \ 0., \ 1.), \ 0.25 \leq x \leq 0.4, \ \mbox{Pre-shock, air}, \\
    (\rho_1, \ \rho_2, \ u, \ p) =& \ (0., \ 0.139, \ 0., \ 1.), \ 0.4 \leq x \leq 0.6, \ \mbox{Pre-shock, helium bubble}, \\
    (\rho_1, \ \rho_2, \ u, \ p) =& \ (1., \ 0, \ 0., \ 1.), \ 0.6 \leq x \leq 1, \ \mbox{Post-shock, air}.
\end{cases}
\end{align*}
For air $c_{v1} = 0.72, \ \gamma_1 = 1.4 $. For helium $c_{v2} = 2.42, \ \gamma_2 = 1.67$. In \cite{Quirk}, this is problem is used to highlight the better behavior of Karni's non-conservative scheme \cite{Karni} over a conservative scheme using the Roe flux (see figure 2 in \cite{Quirk}). In a similar spirit, we compared our semi-discrete ES scheme with the Roe scheme. Figure \ref{fig:1D_shock_bubble} shows the pressure profile at $t = 0.35s$ obtained with each scheme. As expected, the solution with Roe's scheme is rife with oscillations unlike the solution with the present ES scheme which is free of these oscillations. The cause of this improvement is not entropy stability, but the property of preserving stationary contact discontinuities. Figure \ref{fig:1D_shock_bubble_early} shows the pressure profile before the right-moving shock a couple of instants before it meets the helium bubble. Roe's scheme does not preserve stationary contacts and therefore produces pressure anomalies which eventually pollute the solution at $t = 0.35$ (figure \ref{fig:1D_shock_bubble}). This problem is a good illustration of the importance of treating interfaces properly in the simulation of multicomponent compressible flows. The results also suggest that the current ES scheme will not produce oscillations on shock-interface interaction problems if the interface is stationary. This encouraged us to consider the next problem. 

\begin{figure}[h!]
    \centering
    \subfigure[ES flux]{\includegraphics[scale = 0.55]{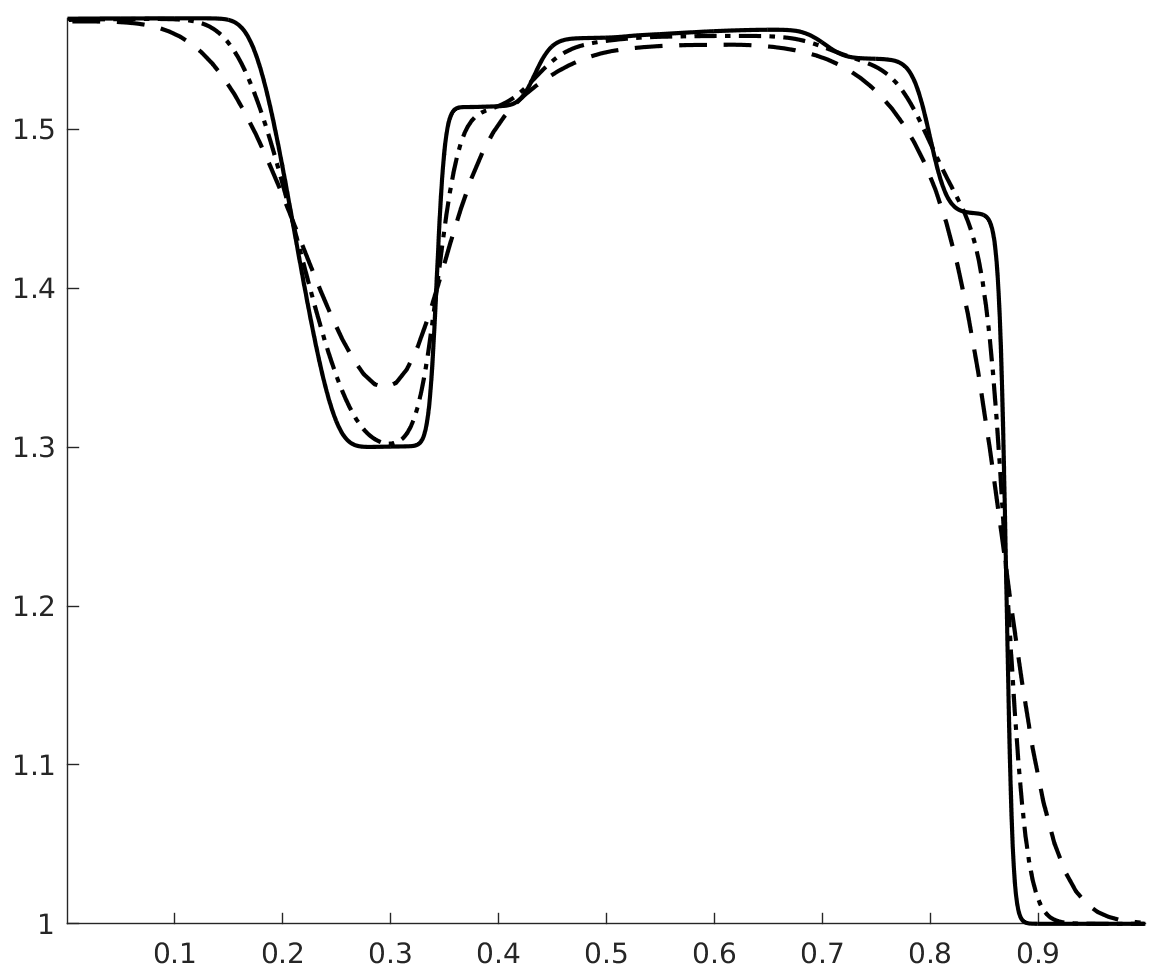}}
    \subfigure[Roe flux]{\includegraphics[scale = 0.55]{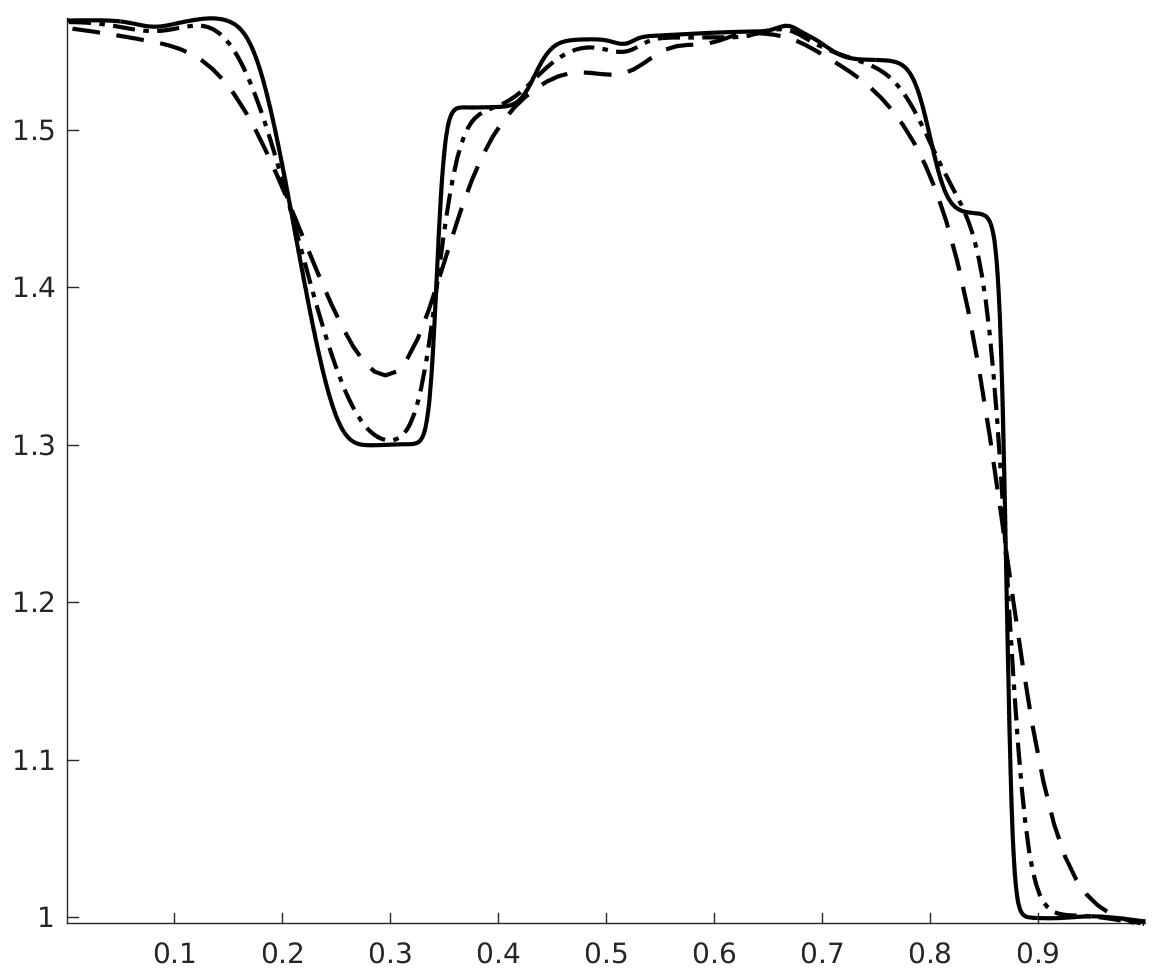}}
    \caption{Pressure profiles for the 1D shock-bubble interaction problem at t = 0.35 s.}
    \label{fig:1D_shock_bubble}
\end{figure}

\begin{figure}[h!]
    \centering
    \subfigure[ES flux]{\includegraphics[scale = 0.55]{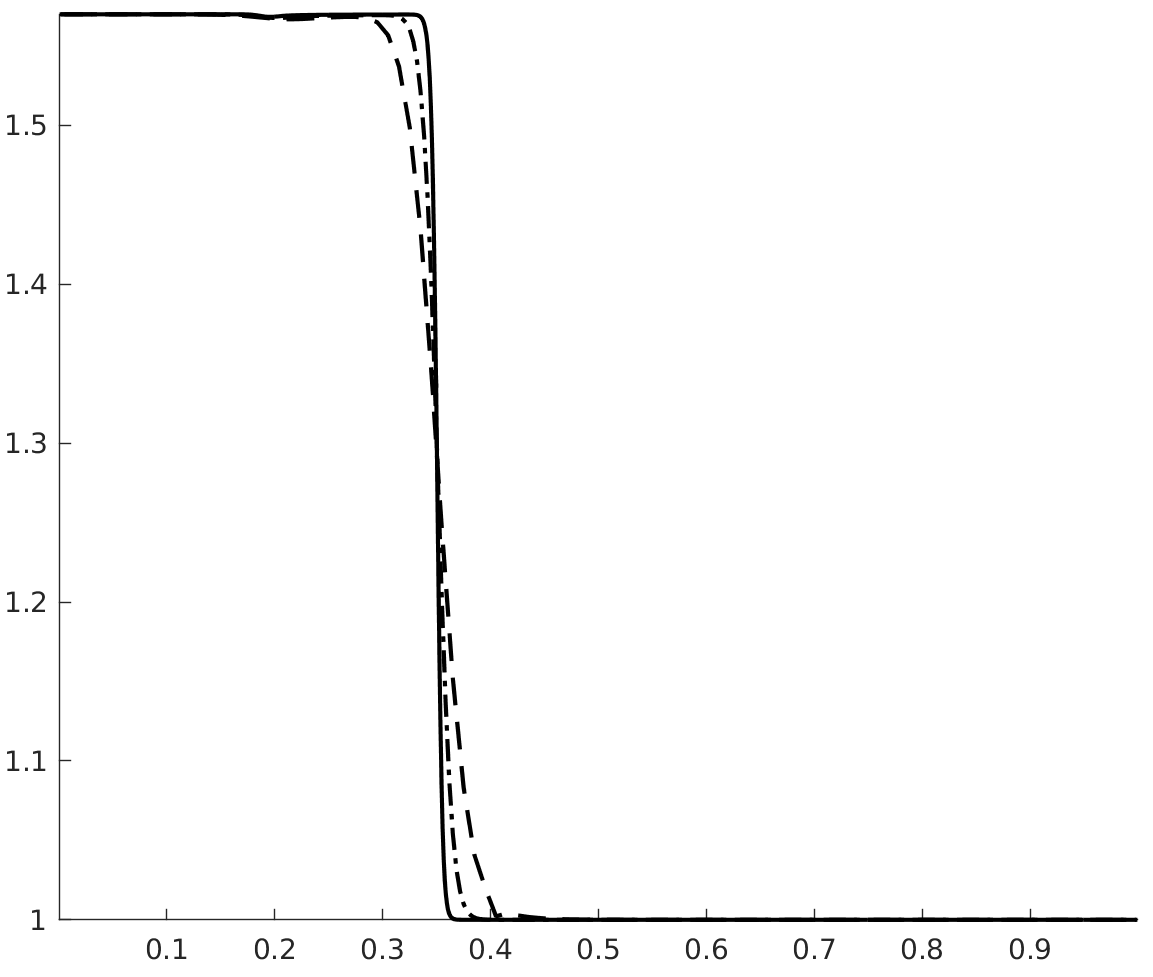}}
    \subfigure[Roe flux]{\includegraphics[scale = 0.55]{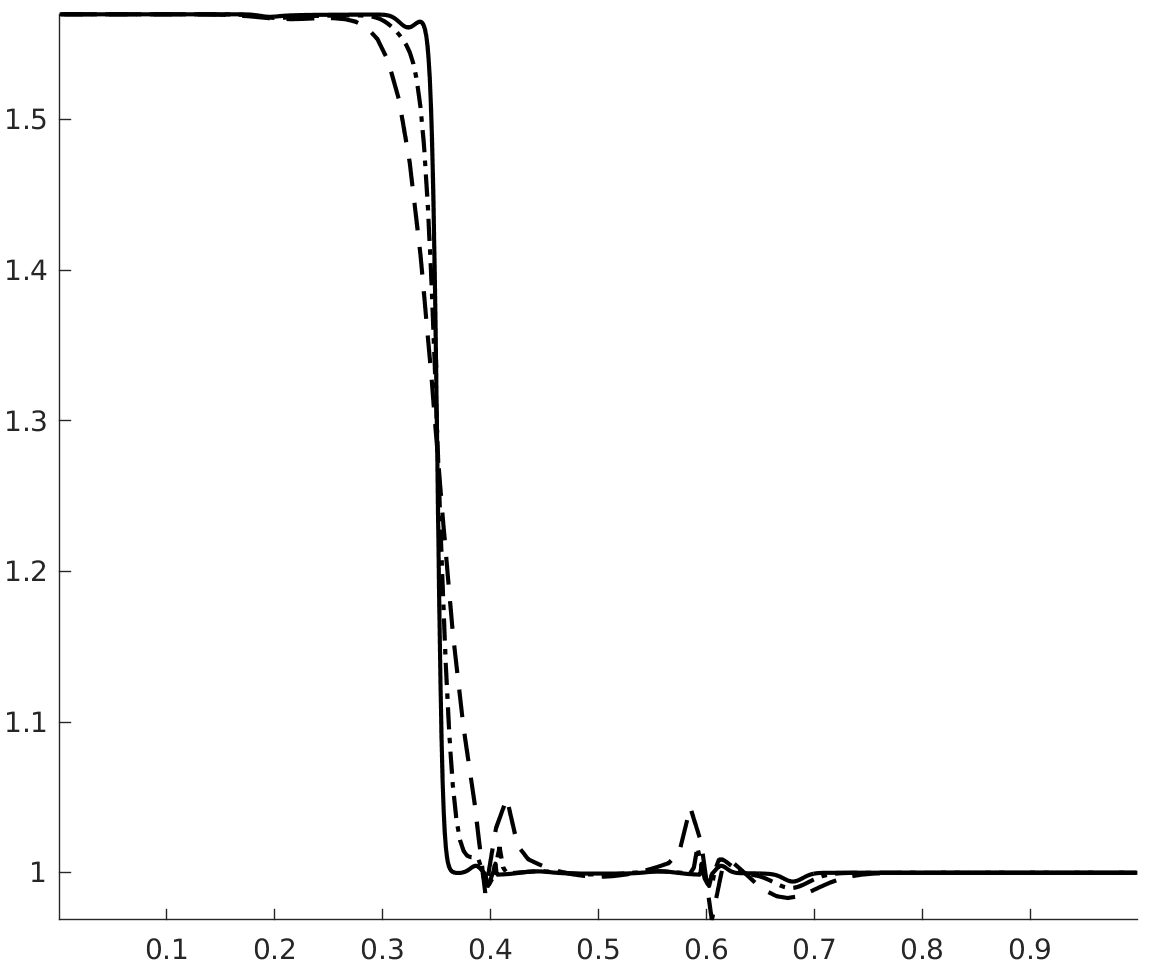}}
    \caption{Pressure profiles for the 1D shock-bubble interaction problem before the shock reaches the bubble, t = 0.069 s.}
    \label{fig:1D_shock_bubble_early}
\end{figure}

\subsection{Shock-bubble interaction}\label{sec:SBI}
\indent A test case that is commonly used in the development of numerical schemes for compressible multicomponent flows \cite{SBI_Kawai, SBI_Marquina, SBI_Johnsen, Quirk} is the interaction of a shock wave with a cylindrical gas inhomogeneity. This problem is a two-dimensional analog of the three-dimensional shock-induced mixing concept proposed by Marble \textit{et al.} \cite{Marble} in the context of supersonic scramjet design. This problem is also used in experimental and computational investigations of the Richtmyer-Meshkov instability \cite{Richtmyer, Meshkov}. \\
\indent Validating the present ES scheme against experimental data is beyond the scope of the present work. In this section, we are essentially interested in the ability of the scheme to simulate the physics relevant to this classic problem. For this purpose, we tried to reproduce the results of Marquina \& Mulet \cite{SBI_Marquina}. The computational domain (ABCD) is shown in figure \ref{fig:2D_bubble_IC}. A Mach $M_S = 1.22$ shock wave, positioned at $x = 275 \ mm$, moves to the left through quiescent air (species 1, $\gamma_1 = 1.4$ and $r_1 = 0.287 \ 10^{3}\ J.kg.^{-1}.K^{-1}$) and eventually meets a cylindrical bubble, centered at $(x, y) = [225, \ 0] \ mm$, filled with helium contaminated with 28\% of air ($\gamma_2 = 1.647, r_2 = 1.578 \ 10^3 \ J.kg^{-1}.K^{-1}$). The flow is assumed to be symmetric about the shock-tube axis (BC), therefore only the upper half of the physical domain is considered. Reflecting boundary conditions are applied on the top (AD) and bottom (BC) boundaries. The boundary conditions upstream (AB) and downstream (CD) the shock are not crucial in this problem \cite{Karni} so we simply extrapolate the flow, as in \cite{SBI_Marquina}. \\
\indent Since the current scheme is unable to guarantee positive partial densities and pressure, the initial conditions from \cite{SBI_Marquina} had to be modified. We set:
\begin{align*}
    \mbox{Region I:} \ \ (\rho_1, \ \rho_2, \ u, \ v, \ p) =& \ (\delta \rho, \ 1.225 (r_1/r_2) - \delta \rho, \ 0., \ 0., \ 101325),  \\
    \mbox{Region II:} \ \ (\rho_1, \ \rho_2, \ u, \ v, \ p) =& \ (1.225 - \delta \rho, \ \delta \rho, \ 0., \ 0., \ 101325), \\ 
    \mbox{Region III:} \ \ (\rho_1, \ \rho_2, \ u, \ v, \ p) =& \ (1.6861 - \delta \rho, \ \delta \rho, \ -113.5243, \ 0., \ 159060 ).
\end{align*}
with $\delta \rho = 0.03$ (units for density, velocity and pressure are $kg/m^3$, $m/s$ and Pa, respectively). This setup differs from the original in two aspects. First, the composition of the gas in regions I, II and III will not be the same. Second, regions I and II are in pressure and temperature equilibrium in the original setup whereas in ours, temperature equilibrium is lost. These differences make quantitative comparisons, notably with the experimental data of Haas \& Sturtevant \cite{SBI_Haas}, difficult to carry out. However, we expect the physics to remain similar qualitatively (see Picone \& Boris \cite{Picone} who studied this problem using a single gas flow model). \\
\indent For this simulation, we used a 4-th order TecNO scheme in space with a 4-th order explicit Runge-Kutta scheme in time. We used a $4000 \times 400$ grid and set the CFL number to $0.3$. Figure \ref{fig:2D_bubble_Helium} shows pseudo-schlieren images of the density gradients at different times after the shock reached the bubble. These are in good qualitative agreement with those produced by Marquina \& Mulet, figure 7 in \cite{SBI_Marquina} (see also Quirk \& Karni \cite{Quirk}, figure 9). We refer to these two references for a detailed discussion of the physical mechanisms at work. \\
\indent In figure \ref{fig:2D_bubble_XT}-(a), we show an x–t diagram of the position of the key features of the shock-bubble interaction. These features are explained in figure \ref{fig:2D_bubble_XT}-(b). The positions of these features are obtained by looking at inflection points of horizontal sections of the shading function $\phi$ used in figure \ref{fig:2D_bubble_Helium}. The upstream bubble interface is tracked on a section at a height $20$ mm from the axis. The incident shock is tracked on a section at $5$ mm from the top wall. The remaining features are tracked on a section along the symmetry line. The x-t diagram from Marquina \& Mulet, figure 5 in \cite{SBI_Marquina}, shows similar trends. The mean velocities of these features are calculated from their visually straight trajectories using linear regression, and displayed in Table \ref{table:V}.

\begin{figure}
    \centering
    \includegraphics[scale=0.85]{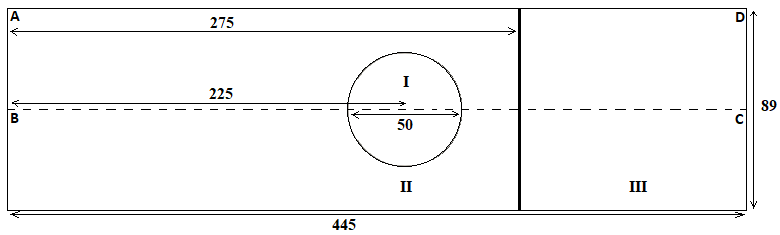}
    \caption{Computational domain (not to scale) for the 2D Shock-bubble interaction problem. Only the top half of the domain (ABCD) is simulated. Lengths in millimeters. Region I: Bubble. Region II: Pre-shock. Region III: Post-shock.}
    \label{fig:2D_bubble_IC}
\end{figure}

\begin{figure}[htbp!]
    \centering
    \subfigure[$t = 23.32 \mu s$]{\includegraphics[scale = 0.6]{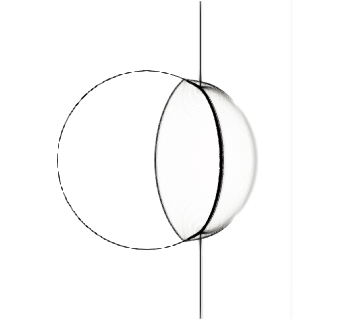}}
    \subfigure[$t = 42.98 \mu s$]{\includegraphics[scale = 0.6]{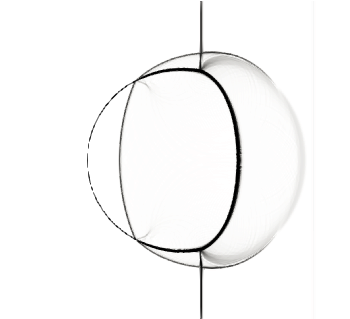}}
    \subfigure[$t = 52.81 \mu s$]{\includegraphics[scale = 0.6]{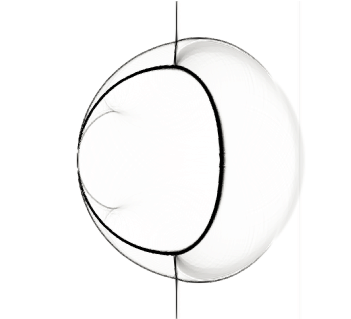}}
    \subfigure[$t = 67.55 \mu s$]{\includegraphics[scale = 0.6]{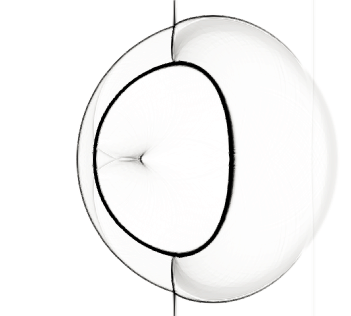}}
    \subfigure[$t = 77.38 \mu s$]{\includegraphics[scale = 0.6]{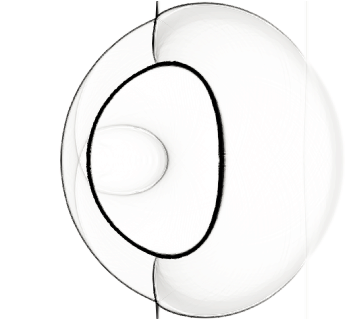}}
    \subfigure[$t = 101.95 \mu s$]{\includegraphics[scale = 0.6]{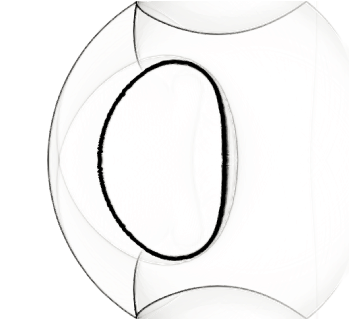}}
    \subfigure[$t = 259.21 \mu s$]{\includegraphics[scale = 0.6]{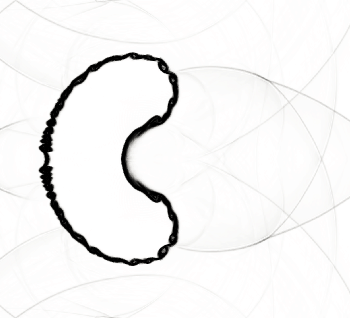}}
    \subfigure[$t = 445.95 \mu s$]{\includegraphics[scale = 0.6]{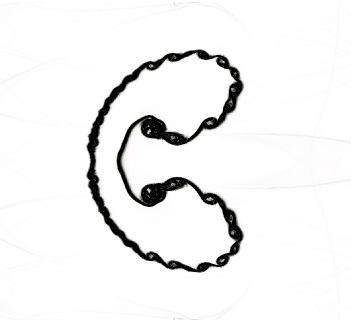}}
    \subfigure[$t = 676.91 \mu s$]{\includegraphics[scale = 0.6]{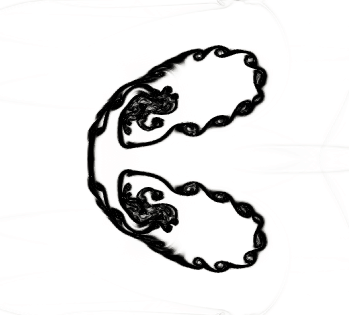}}
    \caption{Pseudo-Schlieren pictures for the shock-bubble interaction problem. $\phi = \exp( - \psi |\nabla \rho| / |\nabla \rho|_{\max}), \ \psi = 10 Y_1 + 150 Y_2$.}
    \label{fig:2D_bubble_Helium}
\end{figure}

\begin{figure}[htbp!]
    \centering
    \subfigure[]{\includegraphics[scale = 0.67]{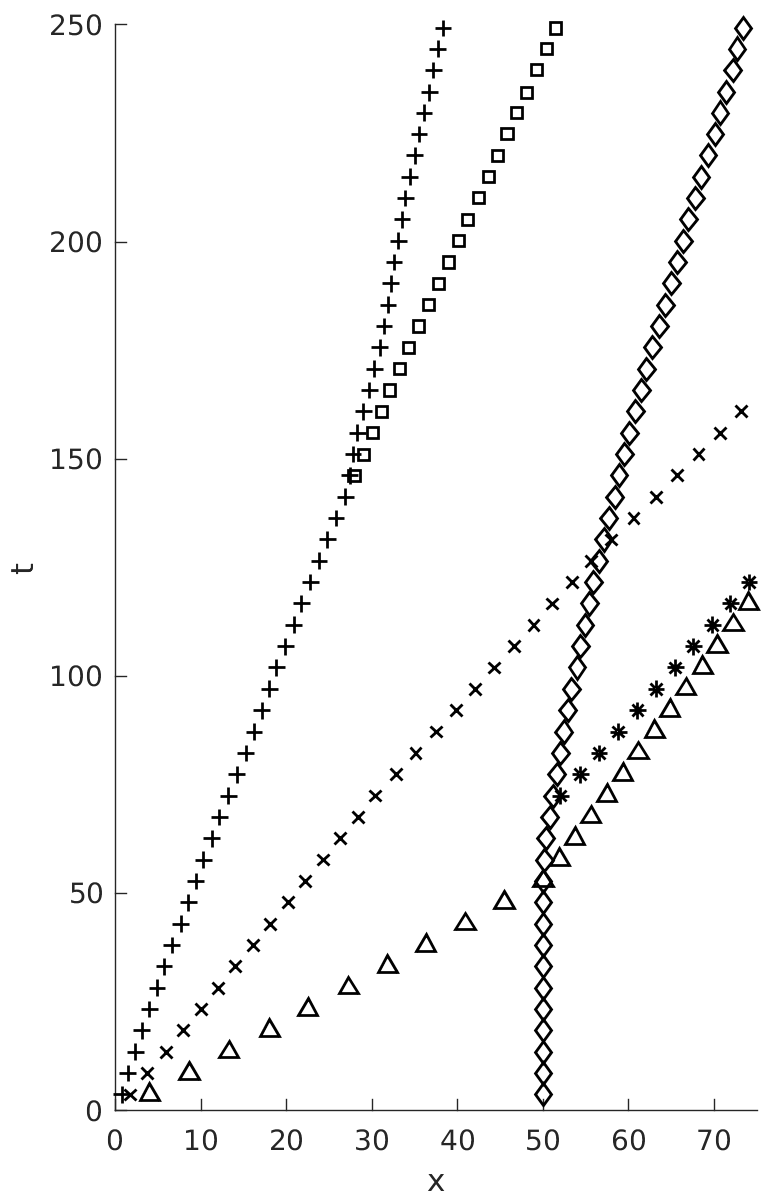}}
    \subfigure[]{\includegraphics[scale = 0.37]{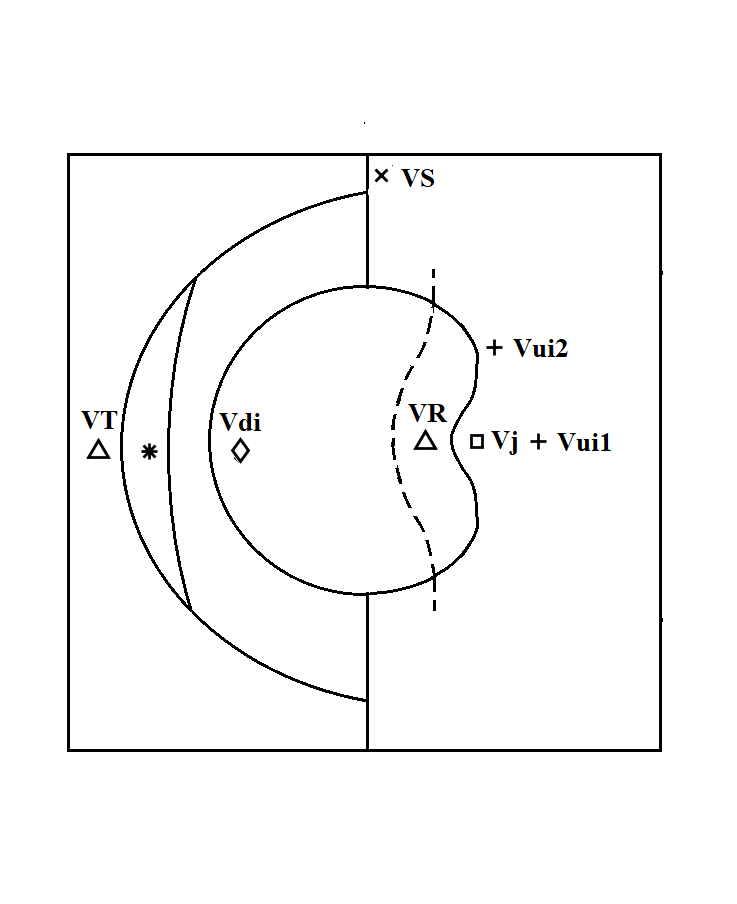}}
    \caption{(a) x–t diagram of the key features explained in (b); (b) VS: incident shock, VR: refracted shock VT: transmitted shock, Vui: upstream border of the bubble, Vdi: downstream border of the bubble, Vj: air jet head}
    \label{fig:2D_bubble_XT}
\end{figure}

\begin{table}
\centering
\begin{tabular}{ c c c c c c c c } 
 \hline
  & VS & VR & VT & Vui1 & Vui2 & Vdi & Vj \\ 
 \hline
 Gouasmi \textit{et al.} & 422 & 954 & 377 & 185 & 105 & 138 & 228  \\ 
 Marquina and Mulet \cite{SBI_Marquina} & 414 & 943 & 373 & 176 & 111 & 153 & 229  \\ 
 %Haas and Sturtevant \cite{SBI_Haas} & 410 & 900 & 393 & 170 & 113 & 145 & 230  \\ 
 \hline
\end{tabular}
\caption{Velocities in $m/s$ of the features explained in figure \ref{fig:2D_bubble_XT}. The time intervals in $\mu s$ for computing each velocity are: VS [3.66, 62.64], VR [3.66, 52.81], VT [52.81, 141.26], Vui1 [3.66, 141.26], Vui2 [146.18, 254.29], Vdi [141.26, 254.29], Vj [146.18, 254.29].}
\label{table:V}
\end{table}

\section{Conclusions}

\indent We formulated entropy-stable schemes for the multicomponent compressible Euler equations. This effort built on the theoretical ground laid out by Chalot \textit{et al.} \cite{Chalot} and Giovangigli \cite{Giovangigli} and followed a procedure pioneered by Tadmor \cite{Tadmor}: we first derived a baseline EC flux and complemented it with an entropy-producing upwind dissipation operator \cite{Roe1}. We showed that in the limit $\rho_k \rightarrow 0$, the EC flux is well-defined. This also holds for the upwind dissipation operator and the TecNO reconstruction provided that the averaged partial densities are well-chosen. Unfortunately, this does not prevent the scheme from producing negative values of partial densities and pressure. We also derived a condition on the averaged state of the dissipation matrix so that the ES scheme can exactly preserve stationary contact discontinuities. \\
\indent It is a well-known issue that conservative schemes are subject to pressure oscillations in moving interface configurations. Numerical experiments showed that the ES scheme we constructed is no exception. We stress that these anomalies, which are not present in the single component case, violate neither entropy stability nor a minimum principle of the specific entropy \cite{Gouasmi4}. The remedies to the pressure oscillations problem typically consist in giving up on conservation of total energy \cite{Abgrall1, Karni, Billet, Abgrall}, which can impair the ability of the scheme to properly capture shocks. A compromise between ensuring entropy stability and the proper treatment of moving interfaces could perhaps be achieved with the EC/ES schemes for non-conservative hyperbolic systems developed by Castro \textit{et al.} \cite{Castro}, even though non-conservative schemes have their own lot of issues \cite{Hou, Abgrall_NC}.  \\
\indent We remain careful to not make any peremptory statement about ES schemes. This work highlighted issues associated with an ES formulation that is more of a default choice than a well-grounded one. The choices of EC flux and dissipation operator are not unique and have to be explored in more depth. Ultimately, we hope that this effort will encourage further work on the proper way to manage entropy locally in a discrete field. 

\section*{Acknowledgments}
\indent Ayoub Gouasmi and Karthik Duraisamy were funded by the AFOSR through grant number FA9550-16-1-030 (Tech. monitor: Fariba Fahroo). Ayoub Gouasmi would like to thank Laslo Diosady and Philip Roe for helpful conversations, and Eitan Tadmor for discussions on the minimum entropy principle \cite{Tadmor2}. \\
\indent Resources supporting this work were provided by the NASA High-End Computing (HEC) Program through the NASA Advanced Supercomputing (NAS) Division at Ames Research Center.

\appendix
\section{Additional results - Moving interface problem}
\subsection{Single component case}\label{appendix:interface_Euler}
Here we present numerical results for a moving contact discontinuity in the compressible Euler equations. The initial conditions are given by:
\begin{align*}
\begin{cases}
    (\rho, \ u, \ p) =& (0.1, \ 1., \ 1.), \ -1.0 \leq x \leq 0.0, \\
    (\rho, \ u, \ p) =& ( 1., \ 1., \ 1.), \ 0.0 < x \leq 1.0,
\end{cases}
\end{align*}
with $\gamma = 1.4$ and $c_{v} = 1$. The anomalies observed in the multicomponent case are not present. The velocity and pressure remain constant at all times.
\begin{figure}[h!]
    \centering
    \subfigure[Velocity]{\includegraphics[scale = 0.75]{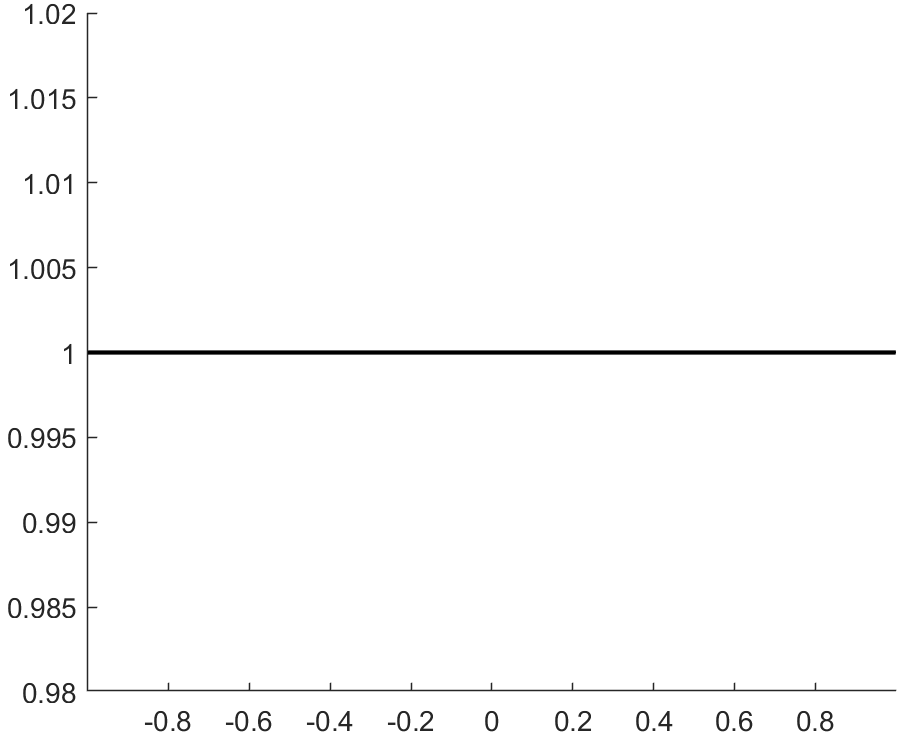}}
    \subfigure[Specific Entropy $s$]{\includegraphics[scale = 0.75]{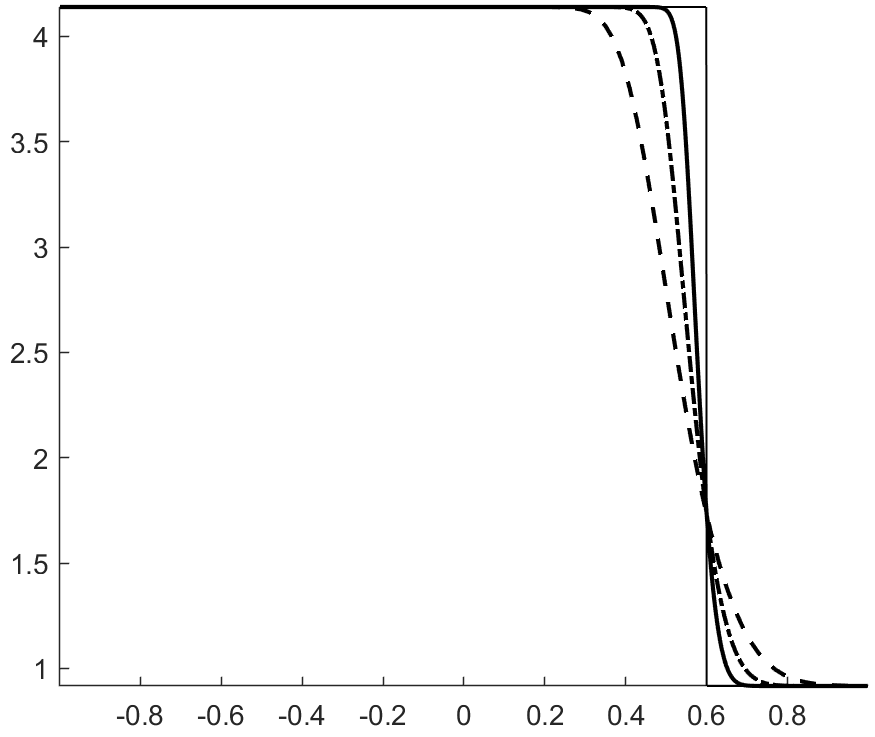}}
    \subfigure[Pressure]{\includegraphics[scale = 0.75]{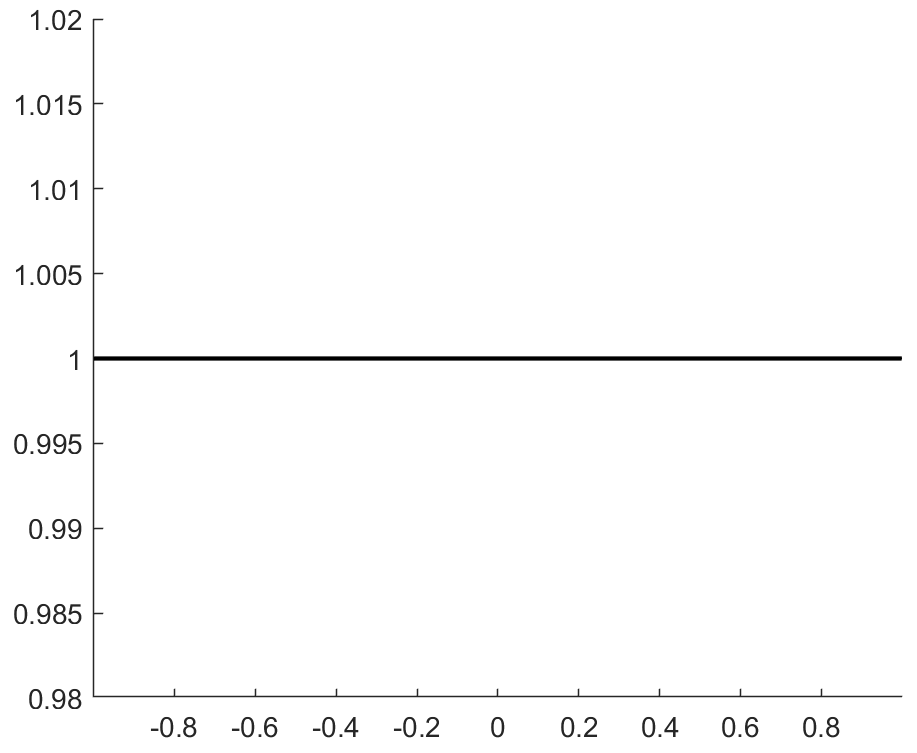}}
    \subfigure[Entropy $\rho s$]{\includegraphics[scale = 0.75]{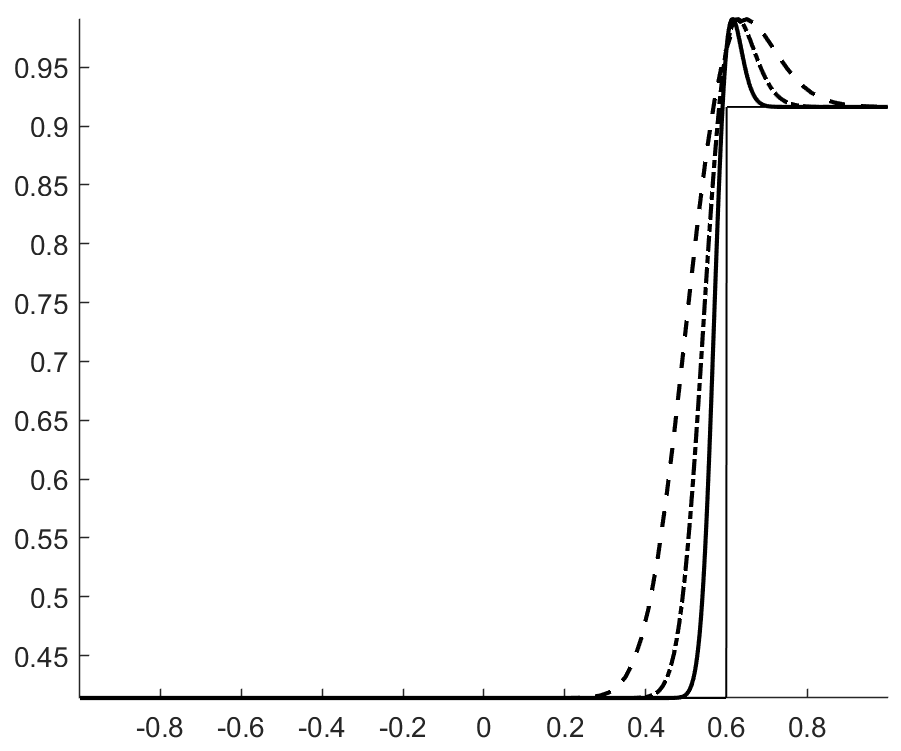}}
    \caption{Solution for the single-component moving contact discontinuity at $t = 0.6$ s. Same legend as in figure 1(a). }
    \label{fig:1D_Euler_moving_interface_velocity}
\end{figure}

\subsection{Conserving entropy in space, producing entropy in time}\label{appendix:interface_BE}
The implicit scheme did not converge in the original setup (most likely due to the generation of negative densities/pressure). We therefore considered a different setup, with non-zero partial densities, for which the pressure oscillations problem is still present. 
\begin{align*}
\begin{cases}
    (\rho_1, \ \rho_2, \ u, \ p) =& (0.3, \ 0.15, \ 1., \ 1.), \ 0 \leq x \leq 0.5, \\
    (\rho_1, \ \rho_2, \ u, \ p) =& (0.15, \ 1., \ 1., \ 1.), \ 0.5 < x \leq 1.0,
\end{cases}
\end{align*}
with $\gamma_1 = 1.4, \ \gamma_2 = 1.6$ and $c_{v1} = c_{v2} = 1$. A grid of 200 cells is used. Figure \ref{fig:1D_BE_moving_interface_pressure} shows the pressure profiles obtained with an EC flux and an ES flux in space, respectively, for two different CFL numbers. In the first case, the entropy production of the scheme only comes from the stabilization of the Backward Euler time scheme, which grows with $\Delta t$. The high frequency oscillations observed are typically observed when too little dissipation is added to EC schemes \cite{Gouasmi1, Zhong}. In the event that the scheme is EC at the fully discrete level (see \cite{Gouasmi1} for an example), these oscillations will be present but will not increase in magnitude over time. This is not desirable.
\begin{figure}[h!]
    \centering
    \subfigure[CFL $=2$]{\includegraphics[scale = 0.75]{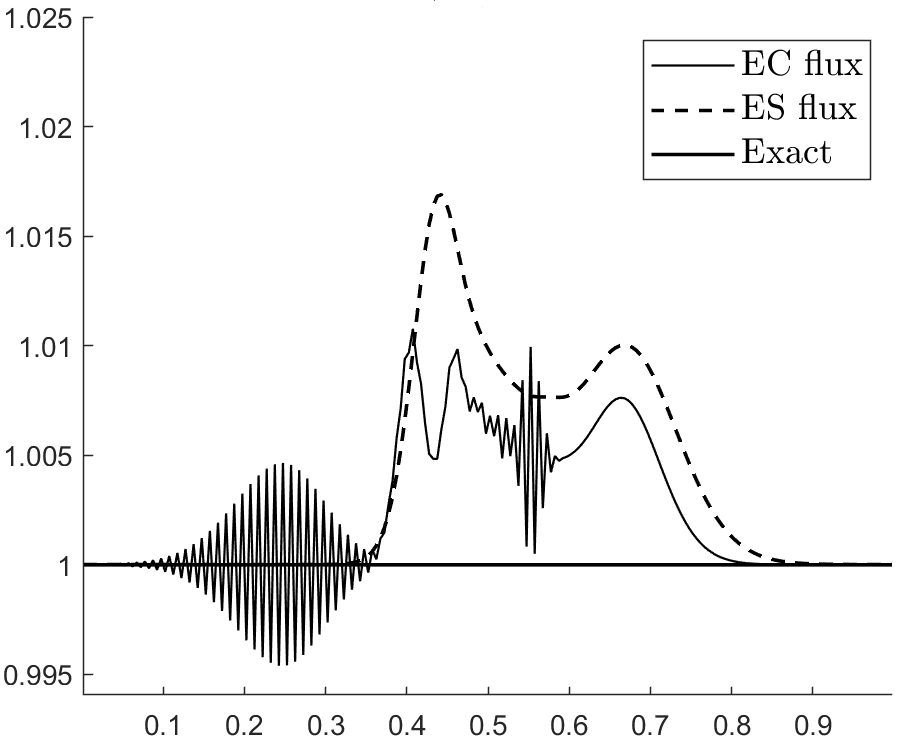}}
    \subfigure[CFL $=8$]{\includegraphics[scale = 0.75]{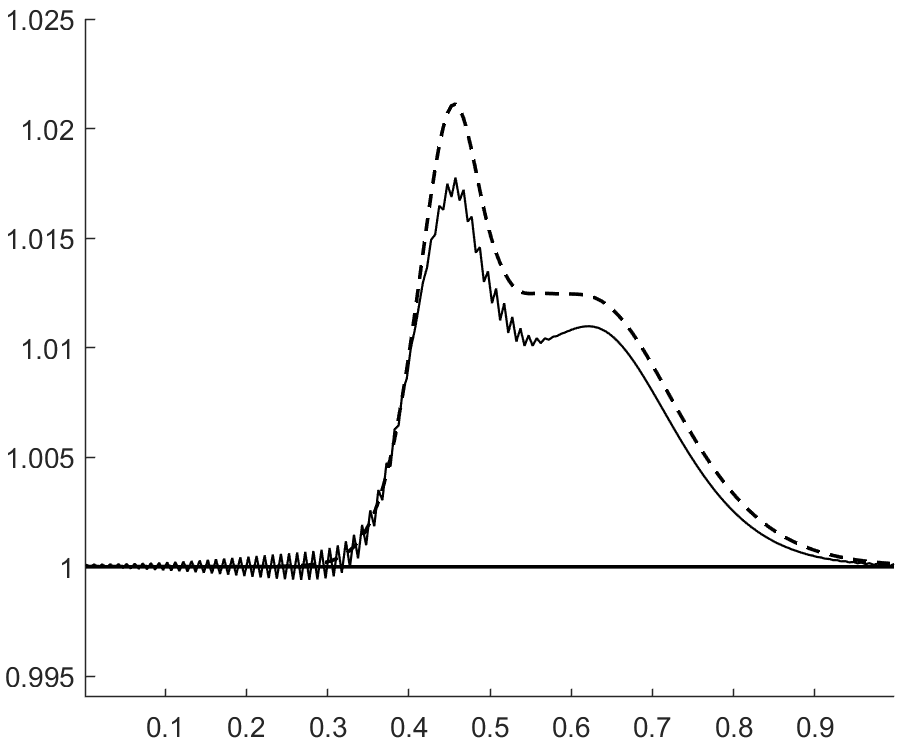}}
    \caption{pressure profiles at t = 0.1 s for a moving interface with BE in time and either an EC flux or an ES flux in space}
    \label{fig:1D_BE_moving_interface_pressure}
\end{figure}

The main conclusion we draw is that the pressure anomalies remain present when the upwind dissipation typically used in space is replaced by the dissipation of Backward Euler. The magnitude of the pressure anomalies is higher when both the interface flux and time scheme are ES.

\section{Three-dimensional version}\label{appendix:3D}
The three-dimensional multicomponent Euler equations are given by:
\begin{equation}\label{eq:3DPDE}
\frac{\partial \mathbf{u}}{\partial t}  + \frac{\partial \mathbf{f_1}}{\partial x_1} + \frac{\partial \mathbf{f_2}}{\partial x_2} + \frac{\partial \mathbf{f_3}}{\partial x_3}  = 0.
\end{equation}
The state vector $\mathbf{u}$ and flux vectors $\mathbf{f_1}, \ \mathbf{f_2}$ and $\mathbf{f_3}$ are defined by:
\begin{align*}
    \mathbf{u} =& \ \begin{bmatrix} \rho_1 & \hdots & \rho_N & \rho u & \rho v & \rho w & \rho e^t \end{bmatrix}^T, \\
    \mathbf{f_1} =& \ \begin{bmatrix} \rho_1 u & \hdots & \rho_N u & \rho u^2 + p & \rho u v & \rho u w &(\rho e^t + p)u \end{bmatrix}^T, \\
    \mathbf{f_2} =& \ \begin{bmatrix} \rho_1 v & \hdots & \rho_N v & \rho u v & \rho v^2 + p & \rho v w & (\rho e^t + p)v \end{bmatrix}^T, \\
    \mathbf{f_3} =& \ \begin{bmatrix} \rho_1 w & \hdots & \rho_N v & \rho u w & \rho v w & \rho w^2 + p & (\rho e^t + p)w \end{bmatrix}^T.
\end{align*}
The conservation equation for entropy writes:
\begin{equation}\label{eq:PDE_Entropy3D}
    \frac{\partial U}{\partial t} +  \frac{\partial F_1}{\partial x_1} + \frac{\partial F_2}{\partial x_2} + \frac{\partial F_3}{\partial x_3}  = 0, \ U = -\rho s, \ F_1 = -u \rho s, \ F_2 = -v \rho s, F_3 = -w \rho s.
\end{equation}
The vector of entropy variables is:
\begin{equation}\label{eq:Entropy_var3D}
    \mathbf{v} = \frac{1}{T} 
    \begin{bmatrix} g_1 - \frac{1}{2} (u^2 + v^2 + w^2) & \hdots & g_N - \frac{1}{2} (u^2 + v^2 + w^2) & u & v & w & -1 \end{bmatrix}.
\end{equation}
The potential flux in space $\mathcal{U}$ is unchanged. There are now three spatial potential functions to work with. They are given by:
\begin{equation}\label{eq:pot3D}
    \mathcal{F}_1 = \mathbf{v} \cdot \mathbf{f_1} - F_1 = \sum_{k=1}^N\frac{R}{m_k} \rho_k u, \ \mathcal{F}_2 = \mathbf{v} \cdot \mathbf{f_2} - F_2 = \sum_{k=1}^N\frac{R}{m_k} \rho_k v, \ \mathcal{F}_3 = \mathbf{v} \cdot \mathbf{f_3} - F_3 = \sum_{k=1}^N\frac{R}{m_k} \rho_k w.
\end{equation}
It can be easily shown that an entropy-conservative flux across an interface of normal $\mathbf{n} = (n_1, n_2, n_3)$, denoted $\mathbf{f^*} $ must satisfy:
\begin{equation*}
    [\mathbf{v}] \cdot \mathbf{f^*} = [n_1 \mathcal{F}_1 + n_2 \mathcal{F}_2 + n_3 \mathcal{F}_3].
\end{equation*}
Using the same method as in the 1D case, one obtains $\mathbf{f^*} = [f_{1,1} \ \dots \ f_{1,N} \ f_2 \ f_3 \ f_4 \ f_5]$ with:
\begin{align}\label{eq:EC_flux3D}
    f_{1,k} =& \  \rho_k^{ln}\overline{u_n}, \nonumber \\
    f_2     =& \ \frac{n_1}{\overline{1/T}} \bigg(\sum_{k=1}^Nr_k \overline{\rho_k}\bigg) + \overline{u} \sum_{k=1}^Nf_{1,k}, \nonumber \\
    f_3     =& \ \frac{n_2}{\overline{1/T}} \bigg(\sum_{k=1}^Nr_k \overline{\rho_k}\bigg) + \overline{v} \sum_{k=1}^Nf_{1,k}, \\
    f_4     =& \ \frac{n_3}{\overline{1/T}} \bigg(\sum_{k=1}^Nr_k \overline{\rho_k}\bigg) + \overline{w} \sum_{k=1}^Nf_{1,k}, \nonumber \\
    f_5     =& \ \sum_{k=1}^N(e_{0k} + c_{vk}\frac{1}{(1/T)^{ln}} - \frac{1}{2}\overline{u^2+v^2+w^2}) f_{1,k} + \overline{u} f_2 + \overline{v} f_3 + \overline{w} f_4. \nonumber
\end{align}
$u_n = n_1 u + n_2 v + n_3 w$ is the velocity normal to the interface. The temporal Jacobian \cite{Giovangigli} is given by:
\begin{gather*}
    H  = 
    \begin{bmatrix}
        \frac{\rho_1}{r_1} &        &             0              &  \frac{\rho_1}{r_1}u &  \frac{\rho_1}{r_1}v &  \frac{\rho_1}{r_1}w & \frac{\rho_1}{r_1}(e_1^t) \\
                             & \ddots &                            &    \vdots &    \vdots  & \vdots &             \vdots                        \\
                            &        &    \frac{\rho_N}{r_N}    &  \frac{\rho_N}{r_N}u &  \frac{\rho_N}{r_N}v & \frac{\rho_N}{r_N}w & \frac{\rho_N}{r_N}(e_N^t) \\
       &  &  &  \rho T + u^2 S_1 &  u v S_1 & u w S_1 & u(\rho T + S_2) \\
       & &  & &  \rho T + v^2 S_1 &  v w S_1 & v(\rho T + S_2) \\
         & & & & &  \rho T + w^2 S_1 &  w(\rho T + S_2) \\
       sym  &  &  &  &  & 
         & \rho T (2k + c_v T) + S_3.
    \end{bmatrix}, \\
    S_1 = \sum_k\frac{ \rho_k}{r_k}, \ S_2 =  \sum_k\frac{\rho_k}{r_k}(e_k^t), \ S_3 = \sum_k\frac{\rho_k}{r_k}(e_k^t)^2 
\end{gather*}
Next is the scaling matrix. Let $A_n$ be the flux Jacobian in the normal direction:
\begin{equation*}
    A_n = n_1 \frac{\partial \mathbf{f_1}}{\partial \mathbf{u}} + n_2 \frac{\partial \mathbf{f_2}}{\partial \mathbf{u}} + n_3 \frac{\partial \mathbf{f_3}}{\partial \mathbf{u}}.
\end{equation*}
A general expression for the eigenvector matrix $R$ such that $A_n = R \Lambda R^{-1}$ is the following:
\begin{gather*}
    R = \begin{bmatrix}
                    1      &        &          &    0   &    0   & Y_1  & Y_1  \\
                          & \ddots &          & \vdots & \vdots & \vdots  & \vdots  \\
                          &        &    1    &    0   &    0   & Y_N  & Y_N  \\
                    u     & \hdots &    u     &        &       & u + a n_1   & u - a n_1  \\
                    v     & \hdots &    v     &   \mathbf{r_I}  &   \mathbf{r_{II}}   & v + a n_2   & v - a n_2  \\
                    w     & \hdots &    w     &      &     & w + a n_3   & w - a n_3  \\
                 k - \frac{d_1}{\gamma-1}  & \hdots & k - \frac{d_N}{\gamma-1} &       &      & h^t + u_n a & h^t - u_n a                     
            \end{bmatrix}, \\ \Lambda = diag(\begin{bmatrix}
                u_n & \hdots & u_n & u_n & u_n+a & u_n-a
            \end{bmatrix}),
\end{gather*}
where $\mathbf{r_{I}}, \mathbf{r_{II}} \in \mathbb{R}^{4 \times 1}$ are such that:
\begin{equation*}
    \mathbf{r_I}, \mathbf{r_{II}} \in span\bigg\{ \mathbf{e_1} = \begin{bmatrix} 0 \\ -an_3 \\ an_2 \\ -a(n_3 v - n_2 w)\end{bmatrix}, \ \mathbf{e_2} =  \begin{bmatrix} -an_3 \\ 0 \\ an_1 \\ -a(n_3 u - n_1 w) \end{bmatrix}, \ \mathbf{e_3} = \begin{bmatrix} -an_2 \\ an_1 \\ 0 \\ -a(n_2 u - n_1 v) \end{bmatrix} \bigg\}, 
\end{equation*}
and $rank(\mathbf{r_I}, \ \mathbf{r_{II}}) = 2 $. If $n_1 \neq 0$, take $(\mathbf{r_{I}}, \mathbf{r_{II}}) = (\mathbf{e_3}, \mathbf{e_2})$. The squared scaling matrix is given by:
\begin{equation}\label{eq:3D_T2}
    T^2 = R^{-1} H R^{-T} = \frac{\rho}{\gamma  r}diag(\begin{bmatrix}
                T^{2Y} & T^{2N} & 1/2 & 1/2
            \end{bmatrix}),
\end{equation}
where $T^{Y} = \mathcal{T}^{Y} (\mathcal{T}^{Y})^T \in \mathbb{R}^{N \times N}$ is the same as in the 1D setting, and $T^{2N} \in \mathbb{R}^2$ is given by:
\begin{equation*}
    T^{2N} = \frac{1}{n_1^2}
            \begin{bmatrix} 
                n_1^2+n_3^2 & -n_2 n_3 \\
                -n_2 n_3 & n_1^2+n_2^2
            \end{bmatrix}.
\end{equation*}
If $n_2^2+n_3^2 = 0$, then $T^{2N} = I_{2 \times 2}$, otherwise $T^{2N} = T^{N} (T^{N})^T$ with:
\begin{equation*}
    T^{N} = \frac{1}{n_1\sqrt{n_2^2+n_3^2}}
            \begin{bmatrix} 
                -n_3 & n_1 n_2 \\
                 n_2 & n_1 n_3
            \end{bmatrix}.
\end{equation*}
If $n_1 = 0$, then if $n_2 \neq 0$, take $(\mathbf{r_{I}}, \mathbf{r_{II}}) = (\mathbf{e_3}, -\mathbf{e_1})$. equation (\ref{eq:3D_T2}) holds with:
\begin{equation*}
    T^{2N} = \frac{1}{n_2^2}
            \begin{bmatrix} 
                n_2^2+n_3^2 & -n_1 n_3 \\
                -n_1 n_3 & n_2^2+n_1^2
            \end{bmatrix}.
\end{equation*}
Since $n_1 = 0$, this simplifies to $T^{2N} = diag([1+(n_3/n_2)^2 \ 1])$ and $T^{N} = diag([\sqrt{1 + (n_3/n_2)^2} \  1])$. If $n_1 = 0$ and $n_2 = 0$, then $n_3 \neq 0$. In this final case, take $(\mathbf{r_{I}}, \mathbf{r_{II}}) = (\mathbf{e_2}, \mathbf{e_1})$. This leads to $T^{2N} = n_3^2 diag([1 \ 1])$ and $T^{N} = n_3 diag([1 \ 1])$. \\
\indent In each case, we can write $T^{2} = \mathcal{T}\mathcal{T}^T$ with:
\begin{equation*}
    \mathcal{T} =  \sqrt{\frac{\rho}{\gamma r}}diag(\begin{bmatrix}
                \mathcal{T}^{Y} & T^{N} & 1/\sqrt{2} & 1/\sqrt{2}
            \end{bmatrix}).
\end{equation*}
The entropy-stable flux $\mathbf{f^*}$ writes:
\begin{equation*}
    \mathbf{f^*} = \mathbf{f_{EC}} - \frac{1}{2}( R \mathcal{T}) |\Lambda| (R \mathcal{T})^T [\mathbf{v}].
\end{equation*}

\section{Some perspective on the fluxes of Tadmor and Barth} \label{appendix:Barth}
In Barth \cite{Barth}, the entropy variables are discretized but the entropy stability of an inviscid interface flux $\mathbf{f^*}$ is defined by the condition that $\mathbf{f^{*}}$ produces more entropy than the so-called Symmetric Mean-Value (SMV) flux he proposed. This flux actually has two distinct expressions \cite{Barth, Barth2}. In chronological order, they are given by:
\begin{align*}
    \mathbf{f_{SMV, I}} =& \ \frac{1}{2}(\mathbf{f}(\mathbf{v}_L) + \mathbf{f}(\mathbf{v}_R)) - \frac{1}{2}\bigg( \int_{0}^1 (1 - \theta) (|A(\mathbf{\overline{v}}(\theta))|_H \ + \ |A(\mathbf{\overline{\overline{v}}}(\theta))|_H) d\theta \bigg) [\mathbf{v}], \\
    \mathbf{f_{SMV, II}} =& \ \frac{1}{2}(\mathbf{f}(\mathbf{v}_L) + \mathbf{f}(\mathbf{v}_R)) - \frac{1}{2}\bigg( \int_{0}^1  |A(\mathbf{\overline{v}}(\theta))|_H d\theta \bigg) [\mathbf{v}], \\
    \mathbf{\overline{v}}(\theta) =& \ \mathbf{v}_L + \theta [\mathbf{v}], \ \mathbf{\overline{\overline{v}}}(\theta) = \mathbf{v}_R - \theta [\mathbf{v}], \ |A|_H = |A|H = R |\Lambda| R^T.
\end{align*}
Barth \cite{Barth, Barth2} defines an entropy-stable flux by the condition that it produces more entropy than the SMV flux:
\begin{equation*}
    [\mathbf{v}] \cdot \mathbf{f^*} \leq [\mathbf{v}] \cdot \mathbf{f_{SMV}}.
\end{equation*}
The SMV flux is entropy stable, therefore it does not define a genuine threshold for entropy stability. 
Tadmor \cite{Tadmor} defines an entropy-stable flux by comparison with an entropy-conservative flux $\mathbf{f_{EC}}$ he derived. This condition writes:
\begin{equation*}
    [\mathbf{v}] \cdot \mathbf{f^*} \leq [\mathbf{v}] \cdot \mathbf{f_{EC}} = [\mathcal{F}], \ \mathbf{f_{EC}} = \int_0^1 \mathbf{f}(\mathbf{\overline{v}}(\theta)) d\theta.
\end{equation*}
Using integration by parts, this flux can be cast in viscosity form \cite{Tadmor}:
\begin{align*}
    \mathbf{f_{EC}} =& \ (\theta - \frac{1}{2}) \mathbf{f}(\mathbf{\overline{v}}(\theta)) \bigg |_{\theta = 0}^{\theta = 1} - \int_0^1 (\theta - \frac{1}{2}) \frac{d}{d\theta} \mathbf{f}(\mathbf{\overline{v}}(\theta)) d\theta \\
                               =& \ \frac{1}{2}(\mathbf{f}(\mathbf{v}_L) + \mathbf{f}(\mathbf{v}_R) ) - \frac{1}{2} \bigg( \int_{0}^1 (2 \theta - 1) A(\mathbf{\overline{v}}(\theta)) H(\mathbf{\overline{v}}(\theta)) d\theta \bigg) [\mathbf{v}].
\end{align*}
Let $Q_{SMV}$ and $Q_{EC}$ be the matrices such that:
\begin{equation*}
    \mathbf{f_{EC}} = \frac{1}{2}(\mathbf{f}(\mathbf{v}_L) + \mathbf{f}(\mathbf{v}_R) ) - \frac{1}{2} Q_{EC} [\mathbf{v}], \ 
    \mathbf{f_{SMV}} = \frac{1}{2}(\mathbf{f}(\mathbf{v}_L) + \mathbf{f}(\mathbf{v}_R) ) - \frac{1}{2} Q_{SMV} [\mathbf{v}].
\end{equation*}
Using:
\begin{equation*}
    [\mathbf{v}] \cdot AH [\mathbf{v}] = [\mathbf{v}] \cdot R \Lambda R^{T} [\mathbf{v}] \leq  [\mathbf{v}] \cdot R |\Lambda| R^{T} [\mathbf{v}] = [\mathbf{v}] \cdot |A|_{H} [\mathbf{v}], 
\end{equation*}
it can be easily shown that:
\begin{equation*}
    [\mathbf{v}] \cdot Q_{SMV, II} [\mathbf{v}] \geq [\mathbf{v}] \cdot Q_{EC} [\mathbf{v}] \ \rightarrow [\mathbf{v}] \cdot \mathbf{f_{SMV, II}} \leq [\mathbf{v}] \cdot \mathbf{f_{EC}}. 
\end{equation*}
Note that a flux similar to $\mathbf{f_{SMV, II}}$ was hinted at by Tadmor (Example 3.2 in \cite{Tadmor1}). Noting that $\mathbf{\overline{v}}(\theta) = \mathbf{\overline{\overline{v}}}(1-\theta)$, and applying the variable change $\theta \leftarrow 1 - \theta$ to half of $Q_{EC}$ leads to another expression for $Q_{EC}$:
\begin{align*}
    Q_{EC} = \int_{0}^1 (\theta - \frac{1}{2}) A(\mathbf{\overline{v}}(\theta)) H(\mathbf{\overline{v}}(\theta)) d\theta + 
                \int_{0}^1 (\frac{1}{2} - \theta) A(\mathbf{\overline{\overline{v}}}(\theta)) H(\mathbf{\overline{\overline{v}}}(\theta)) d\theta,
\end{align*}
which resembles $Q_{SMV, I}$. From here, it is not clear whether
\begin{equation*}
    [\mathbf{v}] \cdot Q_{EC} [\mathbf{v}] \leq [\mathbf{v}] \cdot Q_{SMV, I} [\mathbf{v}],
\end{equation*}
which would imply:
\begin{equation*}
    [\mathbf{v}] \cdot \mathbf{f_{SMV, I}} \leq [\mathbf{v}] \cdot \mathbf{f_{EC}}.
\end{equation*}
In a later publication \cite{BarthMHD}, Barth defines an EC/ES condition:
\begin{equation}\label{eq:ECfluxBarth}
    [\mathbf{v}] \cdot (\mathbf{f^*} - \mathbf{f}(\mathbf{v}(\theta)) ) \leq 0, \ \forall \theta \in [0 \ 1],
\end{equation}
which is slightly different from Tadmor's:
\begin{equation}\label{eq:ECfluxTadmor}
    [\mathbf{v}] \cdot \mathbf{f^*} - [\mathbf{\mathcal{F}}] \leq 0.
\end{equation}
However using the definition of the potential function $\mathcal{F}$:
\begin{equation*}
    \frac{\partial \mathcal{F}}{\partial \mathbf{v}}^T  = \mathbf{f}
\end{equation*}
and a corollary of the fundamental theorem of calculus:
\begin{equation*}
    [\mathbf{\mathcal{F}} ] = [\mathbf{v}] \cdot \bigg (\int_{0}^1 \frac{\partial \mathcal{F}}{\partial \mathbf{v}}^T (\mathbf{\overline{v}}(\theta))d\theta \bigg).
\end{equation*}
Tadmor's condition (\ref{eq:ECfluxTadmor}) can be rewritten as:
\begin{equation}\label{eq:ECfluxTadmor2}
    \int_0^1 [\mathbf{v}] \cdot (\mathbf{f}^* - \mathbf{f}(\mathbf{\overline{v}}(\theta))) d\theta \leq 0,
\end{equation}
which is the requirement that some path-integral is negative. Barth's condition (\ref{eq:ECfluxBarth}) is the more stringent requirement that the integrand in (\ref{eq:ECfluxTadmor2}) is negative. 


\section*{References}
\begin{thebibliography}{9}
\bibitem{Friedrichs}
Friedrichs, K. O., and Lax, P. D. : Systems of Conservation Equations with a Convex Extension, P. Natl. Acad. Sci. U.S.A, 68 (8) pp. 1686-1688, 1971.

\bibitem{Lax}
Lax, P. : Shock Waves and Entropy. Contributions to Nonlinear Functional Analysis, pp. 603–634, 1971. 

\bibitem{Harten}
Harten A. : On the symmetric form of systems of conservation laws with entropy, J. Comput. Phys., 49 (1) pp. 151-164 , 1983.

\bibitem{Mock}
Mock, M. S. : Systems of conservation laws of mixed type, J. Hyperbol. Differ. Eq., 70 (1) pp. 70-88, 1980.

\bibitem{NES1}
Chiodaroli, E., A counterexample to well-posedness of entropy solutions to the compressible euler system, J. Hyperbol. Differ. Eq., 11, pp. 493-519, 2014.

\bibitem{NES2}
de Lellis, D., and Szekelyhidi, L. : On admissibility criteria for weak solutions of the euler equations, Arch. Ration. Mech. An., 195, pp. 225-260, 2010.

\bibitem{NES3}
Elling, V., The carbuncle phenomenon is incurable, Acta Math. Sci., 29 (6), pp. 1647-1656, 2009.

\bibitem{Hughes}
Hughes, T.J. R., Franca, L. P., and Mallet, M. : A new finite element formulation for computational fluid dynamics: I. Symmetric forms of the compressible Euler and Navier-Stokes equations and the second law of thermodynamics, Comput. Method. Appl. M., 54 (2) pp. 223-234, 1986.

\bibitem{Reed}
Reed, W. H., and Hill, T. R. : Triangular mesh methods for the neutron transport equation, Technical Report LA-UR-73-479, Los Alamos Scientific Laboratory, 1973.

\bibitem{Cockburn}
Cockburn B., Karniadakis G.E., and Shu C.W. : The Development of Discontinuous Galerkin Methods. In: Discontinuous Galerkin Methods. Lecture Notes in Computational Science and Engineering, vol 11. Springer, Berlin, Heidelberg, 2000.

\bibitem{ShuOsher}
Shu, C.W., and Osher, S. : Efficient implementation of essentially non-oscillatory shock capturing schemes, J. Comput. Phys., 77 (2) pp. 439-471, 1988.

\bibitem{Giovangigli}
Giovangigli, V. : Multicomponent flow modeling, Chapters 1 and 8, Birkhauser, Boston, 1999.

\bibitem{Giovangigli_Mat}
Giovangigli, V., and Matuszewski, L. : Structure of Entropies in Dissipative Multicomponent Fluids, Kin. Rel. Models, 6 pp. 373-406, 2013.

\bibitem{Chalot}
Chalot, F., Hughes, T.J.R. and Shakib, F. : Symmetrization of conservation laws with entropy for high-temperature hypersonic computations, Comput. Syst. Engrg., 1 (2-4) pp. 495-521, 1990.

\bibitem{Tadmor}
Tadmor, E. : The numerical viscosity of entropy stable schemes for systems of conservation laws. I, Math. Comput., 49 pp. 91-103, 1987.

\bibitem{Tadmor1}
Tadmor, E. : The entropy dissipation by numerical viscosity in nonlinear conservative difference schemes, in: Nonlinear Hyperbolic Problems. Lecture Notes in Mathematics, 1270. Springer, Berlin, Heidelberg, 1987.

\bibitem{Tadmor2}
Tadmor, E. : A minimum entropy principle in the gas dynamics equations, Appl. Numer. Math., 2 (3-5) pp. 211-219, 1986.

\bibitem{Tadmor_acta}
Tadmor, E. : Entropy stability theory for difference approximations of nonlinear conservation laws and related time-dependent problems, Acta Numer., 12, pp. 451–512, 2003

\bibitem{Zhong}
Zhong, W., and Tadmor, E. : Entropy Stable Approximations of Navier-Stokes equations with no artificial numerical viscosity, J. Hyperbol. Differ. Eq., 3 (3) pp. 529-559, 2006.

\bibitem{Guermond_IDP}
Guermond, J.L., Popov, B. : Invariant Domain and First-Order Continuous Finite Element Approximation for hyperbolic systems, SIAM J. Numer. Anal., 54 (4), pp. 2466-2489, 2016.

\bibitem{Guermond_IDP2}
Guermond, J.L., Nazarov, M., Popov, B., Tomas, I. : Second-order invariant domain preserving approximation of the Euler equations using convex limiting, SIAM J. Sci. Comput., 40 (5), pp. 3211–3239, 2018.

\bibitem{Guermond_speed}
Guermond, J.L., and Popov, B. : Fast estimation from above of the maximum wave speed in the Riemann problem for the Euler equations, J. Comput. Phys., 328 pp. 908-926, 2016.

\bibitem{Roe}
Roe, P. L. : Approximate Riemann Solvers, Parameter vectors, and difference schemes, J. Comput. Phys., 43 (2) pp. 357-372, 1981.

\bibitem{Roe1}
Roe, P.L. : Affordable, entropy consistent flux functions. In: Eleventh International Conference on Hyperbolic Problems: Theory, Numerics and Applications, 2006.

\bibitem{Merriam}
Merriam, M. : An Entropy-Based Approach to Nonlinear Stability, NASA Technical Memorandum 101086, 1989.

\bibitem{Godunov}
Godunov, S.K. : An interesting class of quasilinear systems, Dokl. Akad. Nauk. SSSR, 139 pp. 521-523, 1961.

%\bibitem{GodunovMHD}
%Godunov, S.K. : Symmetric form of the magnetohydrodynamic equation, Chislennye Metody Mekh. Sploshnoi Sredy, 3 (1) pp. 26-34, 1972.

\bibitem{Barth}
Barth, T.J. : Numerical Methods for Gasdynamic Systems on Unstructured Meshes. In: Kröner D., Ohlberger M., Rohde C. (eds) An Introduction to Recent Developments in Theory and Numerics for Conservation Laws. Lecture Notes in Computational Science and Engineering, vol 5. Springer, Berlin, Heidelberg, 1999.

\bibitem{Barth2}
Barth, T.J. : Simplified Discontinuous Galerkin Methods for Systems of Conservation Laws with Convex Extension. In: Cockburn B., Karniadakis G.E., Shu CW. (eds) Discontinuous Galerkin Methods. Lecture Notes in Computational Science and Engineering, vol 11. Springer, Berlin, Heidelberg, 2000.

\bibitem{BarthMHD}
Barth, T.J. : On the Role of Involutions in the Discontinuous Galerkin Discretization of Maxwell and Magnetohydrodynamic Systems, in: Arnold D.N., Bochev P.B., Lehoucq R.B., Nicolaides R.A., Shashkov M. (eds) Compatible Spatial Discretizations. The IMA Volumes in Mathematics and its Applications, vol 142. Springer, New York, 2006.

\bibitem{Ismail}
Ismail, F., and Roe, P.L. : Affordable, entropy-consistent Euler flux functions II: Entropy production at shocks, J. Comput. Phys., 228 (15) pp. 5410-5436, 2009.

\bibitem{IsmailThesis}
Ismail, F. : Toward a reliable prediction of shocks in hypersonic flow: resolving carbuncles with entropy and vorticity control, PhD thesis, University of Michigan, 2006. 

\bibitem{Fjordholm}
Fjordholm, U.S., Mishra, S., and Tadmor, E. : Arbitrarily High-order Accurate Entropy Stable Essentially Nonoscillatory Schemes for Systems of Conservation Laws, SIAM J. Numer. Anal., 50 (2) pp. 544-573, 2012.

\bibitem{Fjordholm2}
Fjordholm, U.S., Mishra, S., and Tadmor, E. : ENO  Reconstruction and ENO Interpolation Are Stable, Found. Comput. Math., 13 (2) pp. 139-159, 2012.

\bibitem{LeFloch}
LeFloch, P.G., Mercier, J. M., and Rhode, C. : Fully Discrete, Entropy Conservative Schemes of Arbitrary Order, SIAM J. Numer. Anal., 40 (5) pp. 1968-1992, 2002.

\bibitem{HartenENO}
Harten, A., Engquist, B., Osher, S., and Chakravarthy, S. : Uniformly high order essentially non-oscillatory schemes III, J. Comput. Phys., 71 (2) pp. 231-303, 1987.

\bibitem{Osher}
Osher, S. : Riemann Solvers, the Entropy Condition, and Difference, SIAM J. Numer. Anal., 21 (2) pp. 217–235, 1984.

\bibitem{Abgrall0}
Abgrall, R. : Généralisation du solveur de Roe pour le calcul d’écoulements de mélanges de gaz parfaits à concentrations variables. La Recherche Aérospatiale, 6, pp 31–43, 1988.

\bibitem{Billet}
Billet, G., and Abgrall, R. : An adaptive shock-capturing algorithm for solving unsteady reactive flows, Comput. Fluids, 32 (10) 1473-1495, 2003.

\bibitem{Fernandez}
Fernandez, G., and Larrouturou, B. : Hyperbolic Schemes for Multi-Component Euler Equations. In: Ballmann J., Jeltsch R. (eds) Nonlinear Hyperbolic Equations — Theory, Computation Methods, and Applications. Notes on Numerical Fluid Mechanics, vol 24. Vieweg-Teubner Verlag, 1989.

\bibitem{Larrouturou}
Larrouturou, B. : How to preserve the mass fractions positivity when computing compressible multi-component flows, J. Comput. Phys., 95 (1) 59-84, 2001.

\bibitem{Fezoui}
Larrouturou, B., and Fezoui, L. : On the equations of multi-component perfect of real gas inviscid flow, Nonlinear Hyperbolic Problems, pp.69-98, 1989.

\bibitem{Habbal}
Habbal, A., Dervieux, A., Guillard, H., and Larrouturou, B. : Explicit calculation of reactive flows with an upwind finite element hydro-dynamical code. Technical Report 690, INRIA, 1987.

\bibitem{Abgrall}
Abgrall, R., and Karni, S. : Computations of Compressible Multifluids, J. Comput. Phys., 169 (2) pp. 594-623, 2000.

\bibitem{Karni}
Karni, S. : Multicomponent flow calculation by a consistent primitive algorithm. J. Comput. Phys., 112 (1) pp. 31-43, 1994.

\bibitem{Abgrall1}
Abgrall, R. : How to prevent pressure oscillations in multicomponent flow calculations: a quasi-conservative approach. J. Comput. Phys., 125 (1) pp. 150-160, 1996.

\bibitem{Richtmyer}
Richtmyer, R.D. : Taylor instability in shock acceleration of compressible fluids, Comm. Pure Appl. Math., 13 (2) pp. 297-319, 1960.

\bibitem{Meshkov}
Meshkov, E. E. : Instability of the interface of two gases accelerated by a shock wave, Fluid Dynam., 4 (5) pp. 101-104, 1969.

\bibitem{Picone}
Picone, J. M., and Boris, J. P. : Vorticity generation by shock propagation through bubbles in a gas, J. Fluid Mech., 189 pp. 23-51, 1988.

\bibitem{Marble}
Marble, F. E, Hendricks, G.J., and Zukoski, E. : Progress towards shock enhancement of supersonic combustion processes, 23rd Joint Propulsion Conference, 1987.

\bibitem{Quirk}
Karni, S., and Quirk, J.J. : On the dynamics of a shock-bubble interaction, J. Fluid Mech., 318 pp. 129-163, 1996.

\bibitem{Quirk_carbuncle}
Quirk, J.J. : A contribution to the great Riemann solver debate, Int. J. Numer. Fl., 18 pp. 555-574, 1994.

\bibitem{SBI_Kawai}
Kawai, S., and Terashima, H. : A high-resolution scheme for compressible multicomponent flows with shock waves, Int. J. Numer. Fl., 66 (10) pp. 1207-1225, 2011.

\bibitem{SBI_Haas}
Haas, J.F., and Sturtevant, B. : Interactions of weak shock waves with cylindrical and spherical gas inhomogeneities, J. Fluid Mech., 181 pp. 41-76, 1987.

\bibitem{SBI_Marquina}
Marquina, A., and Mulet, P. : A flux-split algorithm applied to conservative models for multicomponent compressible flows, J. Comput. Phys.,
185 (1) pp. 120-138, 2003.

\bibitem{SBI_Johnsen}
Johnsen, E., and Colonius, T. : Implementation of WENO schemes in compressible multicomponent flow problems, J. Comput. Phys., 219 (2) pp. 715-732, 2006.

\bibitem{Jameson}
Jameson, A. : Formulation  of  kinetic  energy  preserving  conservative  schemes  for  gas  dynamics and direct numerical simulation of one-dimensional viscous compressible flow in a shock tube using entropy and kinetic energy preserving schemes, J. Sci. Comput., 34(2) pp. 188–208, 2008.

\bibitem{Chandrasekhar}
Chandrasekhar, P. : Kinetic energy preserving and entropy stable finite volume schemes for compressible Euler and Navier–Stokes equations, Commun. Comput. Phys., 14 pp. 1252–1286, 2013.

\bibitem{Subbareddy}
Subbareddy, P., and Candler, G.V. : A fully discrete, kinetic energy consistent finite-volume scheme for compressible flows, J. Comput. Phys., 228 (5) pp. 1347-1364, 2009.

%\bibitem{ChandraKling}
%Chandrasekhar, P., and Klingenberg, C. : Entropy stable finite volume scheme for ideal compressible MHD on 2-D Cartesian meshes, SIAM Journal on Numerical Analysis, 54 (2), pp. 1313–1340, 2016.

%\bibitem{Winters}
%Winters, A. R., and Gassner, G. J. : Affordable, entropy conserving and entropy stable flux functions for the ideal MHD equations, Journal of Computational Physics, 304(1) pp. 72-108, 2016.

\bibitem{Derigs}
Derigs, D., Winters, A. R., Gassner, G. J., and Walch, S. : A novel high-order, entropy stable, 3D AMR MHD solver with guaranteed positive pressure, J. Comput. Phys., 317 pp. 223-256, 2016.

%\bibitem{Derigs2}
%Derigs, D., Winters, A. R., Gassner, G. J., and Walch, S. : A novel averaging technique for discrete entropy-stable dissipation operators for ideal MHD, Journal of Computational Physics, 330(1) pp. 624-632, 2017.

\bibitem{Ihme}
Ma, P.C., Lv, Y., and Ihme, M. : An entropy-stable hybrid scheme for simulations of transcritical real-fluid flows, J. Comput. Phys., 340 pp. 330-357, 2017.

%\bibitem{Pares}
%Pares, C. : Numerical method for nonconservative hyperbolic systems: a theoretical framework, SIAM Journal on Numerical Analysis, 44(1) pp. 300–321, 2006.

\bibitem{Castro}
Castro, M.J., Fjordholm, U.S., Mishra, S. and Parés, C. : Entropy conservative and entropy stable schemes for nonconservative hyperbolic systems, SIAM J. Numer. Anal., 51(3) pp. 1371-1391, 2012.

\bibitem{Gouasmi1}
Gouasmi, A., Murman, S.M., and Duraisamy, K. : Entropy conservative schemes and the receding flow problem, J. Sci. Comput., 78 (2) pp. 971-994, 2018.

\bibitem{Gouasmi2}
Gouasmi, A., Duraisamy, K. and Murman, S.M. : On entropy stable temporal fluxes, arXiv:1807.03483, 2018.

\bibitem{Gouasmi4}
Gouasmi, A., Duraisamy, K. and Murman, S.M., and Tadmor, E. : A minimum entropy principle in the compressible multicomponent Euler equations, ESAIM-Math. Model. Num., \textit{accepted}, 2019.

\bibitem{MurmanCTR}
Murman, S.M., and Frontin, C. : Analysis of numerical dissipation in entropy-stable schemes for turbulent flows, Center for Turbulence Research, Proceedings of the Summer Program 2018.

\bibitem{eddy0}
Murman, S.M., Diosady, L.T., Garai, A., and Ceze, M. : A Space-Time Discontinuous-Galerkin Approach for Separated Flows, 54th AIAA Aerospace Sciences Meeting, 2016.

\bibitem{Diosady}
Diosady, L. T., and Murman, S. M. : Higher-Order Methods for Compressible Turbulent Flows Using Entropy Variables, 53rd AIAA Aerospace Sciences Meeting, 2015.

\bibitem{eddy}
Carton de Wiart, C., Diosady, L.T., Garai, A., Burgess, N., Blonigan, P., and Murman, S.M. : Design of a modular monolithic implicit solver for multi-physics applications, AIAA SciTech Forum, 2018.

\bibitem{Pazner}
Pazner, W., and Persson, P-O : Analysis and Entropy Stability of the Line Based Discontinuous Galerkin Method, J. Sci. Comput., 80 (1) pp. 376-402, 2019.

\bibitem{Fernandez_ES}
Fernandez, P., Nguyen, N-C, and Peraire, J. : Entropy-stable hybridized discontinuous Galerkin methods for the compressible Euler and Navier-Stokes equations, arXiv:1808.05066, 2018.

\bibitem{Fisher}
Fisher, T.C. and Carpenter, M.H. : High-order entropy stable finite difference schemes for nonlinear conservation laws: Finite domains, J. Comput. Phys., 252(1) pp 518-557, 2013.

\bibitem{Fried}
Friedrich, L., Schnücke, G., Winters, A.R., Del Rey Fernández, D.C., Gassner, G. J., and Carpenter, M.H. : Entropy Stable Space–Time Discontinuous Galerkin Schemes with Summation-by-Parts Property for Hyperbolic Conservation Laws, J. Sci. Comput., 2019.

%\bibitem{Hesthaven}
%Hesthaven, J. and Mönkeberg, F. : Entropy stable essentially non-oscillatory methods based on RBF reconstructions, Mathematical Modelling and Numerical Analysis, 2018. 

\bibitem{DeepThesis}
Ray, D. : Entropy-stable finite difference and finite-volume schemes for compressible flows, PhD thesis, TIFR, 2017.

\bibitem{Hou}
Hou, T.Y., and Le Floch : Why nonconservative schemes converge to wrong solutions: error analysis, Math. Comput., 62 (206), pp. 497-530, 1994.

\bibitem{Abgrall_NC}
Abgrall, R., and Karni, S. : A comment on the computation of non-conservative products, J. Comput. Phys., 229(8), pp. 2759-2763, 2010.

\bibitem{Zakerzadeh}
Zakerzadeh, H., and, Fjordholm, U.S. : High-order accurate, fully discrete entropy stable schemes for scalar conservation laws, IMA J. Numer. Anal., 36 (2), pp. 633–654, 2016.

\bibitem{HO}
Slotnick J, Khodadoust A, Alonso J, Darmofal D, Gropp W, et al. : CFD Vision 2030 study: a path to revolutionary computational aerosciences. NASA Tech. Rep. CR-2014-218178, Langley Res. Cent., Hampton, VA.
\end{thebibliography}
\end{document}